\documentclass[11pt, oneside]{article}
\usepackage{geometry}                	
\usepackage{amssymb}
\usepackage{mleftright}
\usepackage{amsmath,amsthm,mathtools,color}
\usepackage{usebib}
\usepackage{esint}
\usepackage{graphicx, mathptmx}
\usepackage{mathtools, stackengine, mleftright}
\usepackage[mathcal]{euscript} 
\usepackage{newtxtext, newtxmath, dsfont, mathrsfs, accents, verbatim} 
\usepackage{enumerate}
\usepackage{epsfig}
\usepackage{delarray}
\usepackage{makeidx}
\usepackage[titletoc,title]{appendix}
\usepackage{color}
\numberwithin{equation}{section}
\newtheorem{theorem}{Theorem}[section]
\newtheorem{definition}[theorem]{Definition}
\newtheorem{lemma}[theorem]{Lemma}
\newtheorem{proposition}[theorem]{Proposition}
\newtheorem{Remark}[theorem]{Remark}

\newcommand{\divx}{\mathop{\mathrm{div}}}
\newcommand{\argmin}{\mathop{\mathrm{arg~min}}}
\newcommand{\esssup}{\mathop{\mathrm{ess~sup}}}
\newcommand{\essinf}{\mathop{\mathrm{ess~inf}}}
\newcommand{\dist}{\mathop{\mathrm{dist}}}
\allowdisplaybreaks[4]

\title{A weak solution to a perturbed one-Laplace system by $p$-Laplacian is continuously differentiable}

\author{Shuntaro Tsubouchi\footnote{Graduate School of Mathematical Sciences, The University of Tokyo, Japan. The author was partly supported by Grant-in-Aid for JSPS Fellows (No. 22J12394). \textit{Email}: \texttt{tsubos@ms.u-tokyo.ac.jp}}}
\date{}
\begin{document}
\maketitle
\begin{abstract}
In this paper we aim to show continuous differentiability of weak solutions to a one-Laplace system perturbed by $p$-Laplacian with $1<p<\infty$. The main difficulty on this equation is that uniform ellipticity breaks near a facet, the place where a gradient vanishes. We would like to prove that derivatives of weak solutions are continuous even across the facets. This is possible by estimating H\"{o}lder continuity of Jacobian matrices multiplied with its modulus truncated near zero. To show this estimate, we consider an approximated system, and use standard methods including De Giorgi's truncation and freezing coefficient arguments.
\end{abstract}
\bigbreak
\bigbreak

\section{Introduction}\label{Section: Introduction}
In this paper, we consider a very singular elliptic system given by
\[-b\divx \left(\lvert Du\rvert^{-1}\nabla u^{i}\right)-\divx\left(\lvert Du\rvert^{p-2}\nabla u^{i} \right)=f^{i}\quad \textrm{for }i\in\{\,1,\,\dots\,,\,N\,\}\]
in a fixed domain $\Omega\subset {\mathbb R}^{n}$ with a Lipschitz boundary.
Here \(b\in(0,\,\infty),\,p\in(1,\,\infty)\) are fixed constants, and the dimensions \(n,\,N\) satisfy \(n\ge 2,\,N\ge 1\). For each $i\in\{\,1,\,\dots\,,\,N\,\}$, the scalar functions \(u^{i}=u^{i}(x_{1},\,\dots\,,\,x_{n})\) and \(f^{i}=f^{i}(x_{1},\,\dots\,,\,x_{n})\) are respectively unknown and given in \(\Omega\).
The vector $\nabla u^{i}=(\partial_{x_{1}}u^{i},\,\dots\,,\,\partial_{x_{n}}u^{i})$ denotes the gradient of the scalar function $u^{i}$ with $\partial_{x_{\alpha}}u^{i}=\partial u^{i}/\partial x_{\alpha}$ for $i\in\{\,1,\,\dots\,,\,N\,\},\,\alpha\in\{\,1,\,\dots\,,\,n\,\}$. 
The matrix $Du=(\partial_{x_{\alpha}}u^{i})_{i,\,\alpha}=(\nabla u^{i})_{i}$ denotes the $N\times n$ Jacobian matrix of the mapping $u=(u^{1},\,\dots\,,\,u^{n})$.
Also, the divergence operator $\divx X=\sum_{j=1}^{n}\partial_{x_{j}}X^{j}$ is defined for an ${\mathbb R}^{n}$-valued vector field $X=(X^{1},\,\dots\,,\,X^{n})$ with $X^{j}=X^{j}(x_{1},\,\dots\,,\,x_{n})$ for $j\in\{\,1,\,\dots\,,\,n\,\}$. 
This elliptic system is denoted by
\begin{equation}\label{Eq (Section 1) 1+p-Laplace}
L_{b,\,p}u\coloneqq -b\divx \mleft(\lvert Du\rvert^{-1}Du\mright)-\divx\mleft(\lvert Du\rvert^{p-2}Du\mright)=f\quad\textrm{in }\Omega
\end{equation}
with $f\coloneqq (f^{1},\,\dots\,,\,f^{N})$, or more simply by $-b\Delta_{1}u-\Delta_{p}u=f$.
Here the divergence operators \(\Delta_{1}\) and \(\Delta_{p}\), often called one-Laplacian and $p$-Laplacian respectively, are given by
\[\Delta_{1}u\coloneqq \mathrm{div}\,\mleft(\lvert Du\rvert^{-1}Du\mright),\quad\Delta_{p}u\coloneqq \mathrm{div}\,\mleft(\lvert Du\rvert^{p-2}Du\mright).\]

In the case $b=0$, the system (\ref{Eq (Section 1) 1+p-Laplace}) becomes so called the $p$-Poisson system. For this problem, H\"{o}lder regularity of $Du$ is well-established both in scalar and system cases \cite{MR709038}, \cite{MR672713}, \cite{MR721568}, \cite{MR727034}, \cite{MR474389}, \cite{MR0244628} (see also \cite{MR1230384}, \cite{MR743967}, \cite{MR814022}, \cite{MR783531} for parabolic $p$-Laplace problems). In the limiting case $p=1$, H\"{o}lder continuity of $Du$ generally fails even if $f$ is sufficiently smooth. In fact, even in the simplest case $n=N=1$, any absolutely continuous non-decreasing function of one variable $u=u(x_{1})$ satisfies $-\Delta_{1}u=0$. In particular, even for the one-Lapalce equation, it seems impossible in general to show continuous differentiability ($C^{1}$-regularity) of weak solutions. This problem is substantially because ellipticity of one-Laplacian degenerates in the direction of $Du$, which differs from $p$-Laplacian. Also, in the multi-dimensional case $n\ge 2$, diffusivity of $\Delta_{1}u$ is non-degenerate in directions that are orthogonal to $Du$. It should be mentioned that this ellipticity becomes singular in a facet $\{Du=0\}$, the place where a gradient vanishes.
Since one-Laplacian $\Delta_{1}$ contains anisotropic diffusivity, consisting of degenerate and singular ellipticity, this operator seems analytically difficult to handle in existing elliptic regularity theory.

Therefore, it is a non-trivial problem whether a solution to (\ref{Eq (Section 1) 1+p-Laplace}), which is a one-Laplace system perturbed by $p$-Laplacian, is in $C^{1}$. In the special case where a solution is both scalar-valued and convex, Giga and the author answered this problem affirmatively in a paper \cite{giga2021continuity}. Most of the arguments therein, based on convex analysis and a strong maximum principle, are rather elementary, although the strategy may not work for the system case.
After that work was completed, the author have found it possible to show $C^{1}$-regularity of weak solutions without convexity of solutions. Instead, we would like to use standard methods in elliptic regularity theory, including a freezing coefficient argument and De Giorgi's truncation. This approach is valid even for the system problem, and the main purpose of this paper is to establish $C^{1}$-regularity results in the vectorial case $N\ge 2$.

It is worth mentioning that for the scalar case $N=1$, $C^{1}$-regularity results have been established in the author's recent work \cite{T-scalar}, where a generalization of the operator $\Delta_{1}$ is also discussed. Although the basic strategy in this paper are the same with \cite{T-scalar}, some computations in this paper are rather simpler, since the diffusion operator is assumed to have a symmetric structure, often called the Uhlenbeck structure. It should be recalled that this structure will be often required when one considers everywhere regularity for vector-valued problems. For this reason, we do not try to generalize $\Delta_{1}$ in this system problem.

More generally, in this paper we would like to discuss an elliptic system
\begin{equation}\label{Eq (Section 1) main system}
{\mathcal L}u\coloneqq -b\Delta_{1}u-{\mathcal L}_{p}u=f\quad\textrm{in }\Omega,
\end{equation}
where ${\mathcal L}_{p}$ generalizes $\Delta_{p}$. The detailed conditions of ${\mathcal L}_{p}$ are described later in Section \ref{Subsect: Generalized results in system}.

\subsection{Our strategy}
We would like to briefly describe our strategy in this paper. It should be mentioned that even for the system problem, the strategy itself is the same as the scalar case \cite{T-scalar}. For simplicity, we consider the model case (\ref{Eq (Section 1) main system}). Our result is
\begin{theorem}\label{Theorem: A weak solution is continuously differentiable}
Let $p\in(1,\,\infty),\,q\in(n,\,\infty\rbrack$, and $f\in L^{q}(\Omega;\,{\mathbb R}^{N})$. Assume that $u$ is a weak solution to the system (\ref{Eq (Section 1) 1+p-Laplace}). Then, $u$ is continuously differentiable.
\end{theorem}
The main difficulty on showing $C^{1}$-regularity is that the system (\ref{Eq (Section 1) 1+p-Laplace}) becomes non-uniformly elliptic near a facet. To explain this, we compute the Hessian matrix $D_{\xi}^{2}E(\xi_{0})$ for $\xi_{0}\in{\mathbb R}^{Nn}\setminus\{ 0\}$, where $E(\xi)\coloneqq b\lvert \xi\rvert+\lvert\xi\rvert^{p}/p\,(\xi\in{\mathbb R}^{Nn})$ denotes the energy density. As a result, this $Nn\times Nn$ real symmetric matrix satisfies
\begin{align*}
(\textrm{ellipticity ratio of $D_{\xi}^{2}E(\xi_{0})$})&\coloneqq \frac{(\textrm{the largest eigenvalue of $D_{\xi}^{2}E(\xi_{0})$})}{(\textrm{the lowest eigenvalue of $D_{\xi}^{2}E(\xi_{0})$})}\\&\le \frac{1+b\delta^{1-p}}{p-1}\eqqcolon {\mathcal R}(\delta)
\end{align*}
for all $\xi_{0}\in{\mathbb R}^{Nn}$ with $\lvert \xi_{0}\rvert>\delta>0$ with $\delta$ sufficiently close to $0$. 
It should be mentioned that the bound ${\mathcal R}(\delta)$ will blow up as $\delta$ tends to $0$. In other words, the system (\ref{Eq (Section 1) 1+p-Laplace}) becomes non-uniformly elliptic as $Du$ vanishes. In particular, it will be difficult to show H\"{o}lder continuity of derivatives across a facet of a solution to (\ref{Eq (Section 1) 1+p-Laplace}). We would like to emphasize that our problem is substantially different from either a non-standard growth problem or a $(p,\,q)$-growth problem, where ellipticity ratios will become unbounded as a gradient blows up \cite{MR2291779}, \cite{MR4258810}.
To see our computation above, however, it will be possible that a mapping ${\mathcal G}_{\delta}(Du)$ with
\begin{equation}\label{Eq (Section 1) Truncated vector field}
{\mathcal G}_{\delta}(\xi)\coloneqq \mleft(\lvert \xi\rvert-\delta\mright)_{+}\frac{\xi}{\lvert \xi\rvert}\quad \textrm{for }\xi\in{\mathbb R}^{Nn},\,\delta\in(0,\,1)
\end{equation}
is $\alpha$-H\"{o}lder continuous for some constant $\alpha\in(0,\,1)$, which may depend on $\delta$. This observation is expectable because the mapping ${\mathcal G}_{\delta}(Du)$ is supported in a place $\{\lvert Du\rvert>\delta\}$, where the problem (\ref{Eq (Section 1) 1+p-Laplace}) becomes uniformly elliptic. Although this $\alpha=\alpha(\delta)$ might degenerate as we let $\delta\to 0$, we are able to conclude that $Du$ is also continuous.
In fact, by the definition of ${\mathcal G}_{\delta}$, it is easy to check that the mapping ${\mathcal G}_{\delta}(Du)$ uniformly converges to ${\mathcal G}_{0}(Du)=Du$ as $\delta\to 0$. Thus, to prove $C^{1}$-regularity of solutions, it suffices to prove continuity of the mapping ${\mathcal G}_{\delta}(Du)$ for each fixed $\delta\in(0,\,1)$. 

When we show H\"{o}lder continuity of ${\mathcal G}_{\delta}(Du)$, one of the main difficulties is that the system (\ref{Eq (Section 1) 1+p-Laplace}) becomes very singular near facets of solutions. In particular, it seems impossible to justify regularity on second order Sobolev derivatives of solutions across the facets, based on difference quotient methods.
Therefore, we will have to relax the very singular operator $L_{b,\,p}$ by regularized operators that are non-degenerate and uniformly elliptic, so that higher regularity on Sobolev derivatives are guaranteed. For the model case (\ref{Eq (Section 1) 1+p-Laplace}), the approximation problem is given by
\begin{equation}\label{Eq (Section 1) Approximating problem special case}
L_{b,\,p}^{\varepsilon}u_{\varepsilon}\coloneqq -\divx\mleft(\frac{Du_{\varepsilon}}{\sqrt{\varepsilon^{2}+\lvert Du_{\varepsilon}\rvert^{2}}}\mright)-\divx\mleft(\mleft(\varepsilon^{2}+\lvert Du_{\varepsilon}\rvert^{2}\mright)^{p/2-1}Du_{\varepsilon}\mright)=f_{\varepsilon}\quad \textrm{for each }\varepsilon\in(0,\,1),
\end{equation}
where $f_{\varepsilon}$ convergences to $f$ in a weak sense. The relaxed operator $L_{b,\,p}^{\varepsilon}$ naturally appears when one approximates the density $E$ by
\[E_{\varepsilon}(z)\coloneqq b\sqrt{\varepsilon^{2}+\lvert z\rvert^{2}}+\frac{1}{p}\mleft(\varepsilon^{2}+\lvert z\rvert^{2}\mright)^{p/2}\quad \textrm{for }\varepsilon\in(0,\,1),\,z\in{\mathbb R}^{Nn}.\]
Then, the system (\ref{Eq (Section 1) Approximating problem special case}) is uniformly elliptic, in the sense that
\begin{align*}
(\textrm{ellipticity ratio of $D_{\xi}^{2}E_{\varepsilon}(\xi_{0})$})&\le C\mleft(1+\mleft(\varepsilon^{2}+\lvert \xi_{0}\rvert^{2}\mright)^{(1-p)/2}\mright)\le C\mleft(1+\varepsilon^{1-p}\mright)\quad \textrm{for all }\xi_{0}\in{\mathbb R}^{Nn},
\end{align*}
where the positive constant $C$ depends only on $b$, and $p$. 
This ellipticity ratio appears to be dominated by $\sqrt{\varepsilon^{2}+\lvert\xi_{0}\rvert^{2}}$, rather than by $\lvert\xi_{0}\rvert$. In particular, to measure ellipticity ratios for (\ref{Eq (Section 1) Approximating problem special case}), it is natural to adapt $V_{\varepsilon}\coloneqq\sqrt{\varepsilon^{2}+\lvert Du_{\varepsilon}\rvert^{2}}$ as another modulus. 
For this reason, we have to consider another mapping ${\mathcal G}_{\delta,\,\varepsilon}(Du_{\varepsilon})$, where
\begin{equation}\label{Eq (Section 1) relaxed truncation}
{\mathcal G}_{\delta,\,\varepsilon}(\xi)\coloneqq \left(\sqrt{\varepsilon^{2}+\lvert \xi\rvert^{2}}-\delta\right)_{+}\frac{\xi}{\lvert \xi\rvert}\quad \textrm{for }\xi\in{\mathbb R}^{Nn}\quad \textrm{with }0<\varepsilon<\delta.
\end{equation}
Then, our problem is reduced to a priori H\"{o}lder estimates of ${\mathcal G}_{2\delta,\,\varepsilon}(Du_{\varepsilon})$, where $u_{\varepsilon}$ solves (\ref{Eq (Section 1) Approximating problem special case}) with $0<\varepsilon<\delta/4$ and $0<\delta<1$. To obtain these a priori estimates, we appeal to a freezing coefficient argument and De Giorgi's truncation. Roughly speaking, the former method can be applied when $Du_{\varepsilon}$ does not vanish in a suitable sense, and otherwise the latter is fully used. To judge whether $Du_{\varepsilon}$ degenerates or not, we will measure superlevel sets of $V_{\varepsilon}$. It should be mentioned that distinguishing based on sizes of superlevel sets is found in existing works on $C^{1,\,\alpha}$-regularity for $p$-Laplace problems both in elliptic and parabolic cases (see e.g., \cite[Chapter IX]{MR1230384}, \cite{MR2635642}).

Our method works in the system case, as long as the energy density $E$ is spherically symmetric. Here we recall a well-known fact that when one considers everywhere regularity of a vector-valued solution, a diffusion operator is often required to have a symmetric structure, called the Uhlenbeck structure. For this reason, generalization of the one-Laplace operator $\Delta_{1}$ is not discussed in this paper (see also \cite[\S 2.4]{T-scalar} for detailed explanations). It should be emphasized that this restriction is not necessary in the scalar case. In fact, generalization of $\Delta_{1}$ is given in the author's recent work \cite{T-scalar}, which only focuses on the scalar case, and deals with more general approximation arguments based on the convolution of standard mollifiers.

\subsection{Mathematical models and comparisons to related works on very degenerate problems}\label{Subsect: Models and comparisons}
In Section \ref{Subsect: Models and comparisons}, we briefly describe mathematical models.

The equation (\ref{Eq (Section 1) 1+p-Laplace}) is derived from a minimizing problem of a functional
\[{\mathcal F}(u)\coloneqq \int_{\Omega}E(Du)\,{\mathrm d}x-\int_{\Omega}\langle f\mid u\rangle\,{\mathrm{d}}x\quad \textrm{with}\quad E(z)\coloneqq b\lvert \xi\rvert+\frac{1}{p}\lvert \xi\rvert^{p}\quad \textrm{for }\xi\in{\mathbb R}^{Nn}\]
under a suitable boundary condition.
Here $\langle\,\cdot\mid\cdot\,\rangle$ denotes the standard inner product in ${\mathbb R}^{N}$.
The density $E$ often appears in mathematical modeling of materials, including motion of Bingham fluids and growth of crystal surface. 

In a paper \cite{spohn1993surface}, Spohn gave a mathematical model of crystal surface growth under roughening temperatures. From a thermodynamic viewpoint (see \cite{MR3289366} and the references therein), the evolution of a scalar-valued function $h=h(x,\,t)$ denoting the height of crystal in a two-dimentional domain $\Omega$, is modeled as
\[\partial_{t} h+\Delta \mu=0\quad \textrm{with}\quad\Delta=\Delta_{2}.\]
Here $\mu$ is a scalar-valued function denoting a chemical potential, and considered to satisfy the Euler--Lagrange equation
\[\mu=-\frac{\delta \Phi}{\delta h}\quad\textrm{with}\quad \Phi(h)=\beta_{1}\int_{\Omega}\lvert\nabla h\rvert\,dx+\beta_{3} \int_{\Omega}\lvert\nabla h\rvert^{3}\,dx,\quad \beta_{1},\,\beta_{3}>0.\]
The energy functional $\Phi$ is called a crystal surface energy, whose density is essentially the same as \(E\) with \(p=3\).
Finally, the resulting evolution equation for \(h\) is given by 
\[k\partial_{t}h=\Delta L_{k\beta_{1},\,3}h\quad\textrm{with}\quad k=\frac{1}{3\beta_{3}}.\]
When $h$ is stationary, then $h$ must satisfy an equation $L_{k\beta_{1},\,3}h=f$, where $f=-k\mu$ is harmonic and therefore smooth. Our Theorem \ref{Theorem: A weak solution is continuously differentiable} implies that a gradient $\nabla h$ is continuous.

When $p=2$, the density $E$ also appears when modeling motion of a material called Bingham fluid. Mathematical formulations concerning Bingham fluids are found in \cite[Chapter VI]{MR0521262}. In particular, when the motion of fluids is sufficiently slow, the stationary flow model results in an elliptic system 
\[\frac{\delta \Psi}{\delta V}\equiv \mu_{2}(L_{\mu_{1}/\mu_{2},\,2}V)=-\nabla \pi\quad \textrm{with }\Psi(V)\coloneqq \mu_{1}\int_{U}\lvert DV\rvert\,{\mathrm{d}}x+\frac{\mu}{2}\int_{U}\lvert DV\rvert^{2}\,{\mathrm{d}}x.\]
Here the unknown ${\mathbb R}^{3}$-valued vector field $V=V(x)$ denotes velocity of Bingham fluids in a domain $U\subset{\mathbb R}^{3}$ and satisfies an incompressible condition $\divx V=0$. The other unknown scalar-valued function $\pi=\pi(x)$ denotes the pressure.  Bingham fluids contain two different aspects of elasticity and viscosity, and corresponding to these properties, the positive constants $\mu_{1}$ and $\mu_{2}$ respectively appear in the energy functional $\Psi$.
From Theorem \ref{Theorem: A weak solution is continuously differentiable}, we conclude that if the pressure function $\pi$ satisfies $\nabla\pi\in L^{q}(U;\,{\mathbb R}^{3})$ for some $q>3$, then the velocity field $V$ is continuously differentiable.
Also, when one considers a stationary laminar Bingham flow in a cylindrical pipe $U=\Omega\times {\mathbb R}\subset{\mathbb R}^{3}$, a scalar problem appears. To be precise, we let $V$ be of the form $V=(0,\,0,\,u(x_{1},\,x_{2}))$, where $u$ is an unknown scalar function. Clearly this flow $V$ is incompressible. Under this setting, we have a resulting elliptic equation 
\[L_{\mu_{1}/\mu_{2},\,2}u=f\quad \textrm{in }\Omega\subset{\mathbb R}^{2} \quad \textrm{with}\quad f=-\frac{\pi^{\prime}}{\mu_{2}},\]
where the pressure function $\pi$ depends only on $x_{3}$, and its derivative $\pi^{\prime}=-\mu_{2} f$ must be constant. Also in this laminar model, from our main result, it follows that the component $u$ is continuously differentiable.

The density function $E$ also appears in mathematical modeling of congested traffic dynamics \cite{MR2407018}. There this model results in a minimizing problem 
\begin{equation}\label{Eq (Section 1) An optimal transport model}
\sigma_{\mathrm{opt}}\in \argmin\mleft\{\int_{\Omega}E(\sigma)\,{\mathrm d}x\mathrel{}\middle|\mathrel{} \left.\begin{array}{c}
\sigma\in L^{p}(\Omega;\,{\mathbb R}^{n}),\\ -\divx \sigma=f\textrm{ in }\Omega,\,\sigma\cdot\nu=0\textrm{ on }\partial\Omega
\end{array} \right. \mright\}.
\end{equation}
In a paper \cite{MR2651987}, it is shown that the optimal traffic flow $\sigma_{\mathrm{opt}}$ of the problem (\ref{Eq (Section 1) An optimal transport model}) is uniquely given by $\nabla E^{\ast}(\nabla v)$, where $v$ solves an elliptic equation
\begin{equation}\label{Eq (Section 1) Very degenerate eq}
-\divx(\nabla E^{\ast}(\nabla v))=f\in L^{q}(\Omega)\quad \textrm{in }\Omega
\end{equation}
under a Neumann boundary condition. Here \[E^{\ast}(z)=\frac{1}{p^{\prime}}\mleft(\lvert z\rvert-b\mright)_{+}^{p^{\prime}}\] is the Legendre transform of $E$, and $p^{\prime}\coloneqq p/(p-1)$ denotes the H\"{o}lder conjugate of $p\in(1,\,\infty)$. 

The problem whether the flow $\sigma_{\mathrm{opt}}=\nabla E^{\ast}(\nabla v)$ is continuous has been answered affirmatively under the assumption $q\in(n,\,\infty\rbrack$. There it is an interesting question whether vector filed ${\mathcal G}_{b+\delta}(\nabla v)$ with $\delta>0$ fixed, defined similarly to (\ref{Eq (Section 1) Truncated vector field}), is continuous.
We should note that the equation (\ref{Eq (Section 1) Very degenerate eq}) is also non-uniformly elliptic around the set $\{\lvert \nabla v\rvert\le b\}$, in the sense that the ellipticity ratio of $\nabla^{2}E^{\ast}(z_{0})$ will blow up as $\lvert z_{0}\rvert\to b+0$. However, continuity of truncated vector fields ${\mathcal G}_{b+\delta}(\nabla u)$ with $\delta>0$ is expectable, since there holds
\begin{align*}
(\textrm{ellipticity ratio of $\nabla^{2}E^{\ast}(z_{0})$})&=\frac{(\textrm{the largest eigenvalue of $\nabla^{2}E^{\ast}(z_{0})$})}{(\textrm{the lowest eigenvalue of $\nabla^{2}E^{\ast}(z_{0})$})}\\&\le (p-1)\left(1+(\delta-b)^{-1}\right)
\end{align*}
for all $z_{0}\in{\mathbb R}^{n}$ satisfying $\lvert z_{0}\rvert\ge b+\delta$ when $\delta$ is sufficiently close to $0$. This estimate suggests that for each fixed $\delta>0$, the truncated vector field ${\mathcal G}_{b+\delta}(\nabla v)$ should be H\"{o}lder continuous. It should be noted that it is possible to show ${\mathcal G}_{b+\delta}(\nabla v)$ uniformly converges to ${\mathcal G}_{b}(\nabla v)$ as $\delta\to 0$, and thus ${\mathcal G}_{b}(\nabla v)$ will be also continuous.

When $v$ is scalar-valued, continuity of ${\mathcal G}_{b+\delta}(\nabla v)$ with $\delta>0$ was first shown by Santambrogio--Vespri \cite{MR2728558} in 2010 for the special case $n=2$ with $b=1$. The proof therein is based on oscillation estimates on the Dirichlet energy, which works under the condition $n=2$ only. Later in 2014, Colombo--Figalli \cite{MR3133426} established a more general proof that works for any dimension $n\ge 2$ and any density function $E^{\ast}$, as long as the zero-levelset of $E^{\ast}$ is sufficiently large enough to define a real-valued Minkowski gauge. This Minkowski gauge becomes a basic modulus for judging uniform ellipticity of equations they treated. Here we would like to note that their strategy will not work for our problem (\ref{Eq (Section 1) 1+p-Laplace}), since the density function $E^{\ast}$ seems structurally different from $E$. In fact, in our problem, the zero-levelset of $E$ is only a singleton, and therefore it seems impossible to adapt the Minkowski gauge as a real-valued modulus. The recent work by B\"{o}gelein--Duzaar--Giova--Passarelli di Napoli \cite{BDGPdN} is motivated by extending these regularity results to the vectorial case. There they considered an approximation problem of the form
\begin{equation}\label{Eq (Section 1) Relaxed very degenerate equation}
-\varepsilon\Delta v_{\varepsilon}-\divx\,(D_{\xi}E^{\ast}(Dv_{\varepsilon}))=f\quad \textrm{in }\Omega
\end{equation}
with $b=1$. The paper \cite{BDGPdN} provides a priori H\"{o}lder continuity of ${\mathcal G}_{1+2\delta}(Dv_{\varepsilon})$ for each fixed $\delta\in(0,\,1)$, whose estimate is independent of an approximation parameter $\varepsilon\in(0,\,1)$. It will be worth noting that the modulus $\lvert Dv_{\varepsilon}\rvert$, which measures ellipticity ratio of (\ref{Eq (Section 1) Relaxed very degenerate equation}), has a symmetric structure. This fact appears to be fit to prove everywhere H\"{o}lder continuity of ${\mathcal G}_{1+2\delta}(Dv_{\varepsilon})$ for the system case.

Although our proofs on a priori H\"{o}lder estimates are inspired by \cite[\S 4--7]{BDGPdN}, there are mainly three different points between our proofs and theirs. 
The first is how to approximate systems. To be precise, our regularized problem (\ref{Eq (Section 1) Approximating problem special case}) is based on relaxing the principal part $L_{b,\,p}u$ itself. For this reason, we should treat a different mapping ${\mathcal G}_{2\delta,\,\varepsilon}$, instead of ${\mathcal G}_{2\delta}$. Compared with ours, for their approximation problems (\ref{Eq (Section 1) Relaxed very degenerate equation}), the principal part itself is not changed at all, and thus it seems not necessary to introduce another mapping than ${\mathcal G}_{2\delta}$. The second lies in structural differences between the densities $E$ and $E^{\ast}$. In particular, when showing a Campanato-type energy growth estimate, most of our computations given in Section \ref{Section: Campanato estimates} will differ from theirs given in \cite[\S 4--5]{BDGPdN}. The third is that we have to control an approximation parameter $\varepsilon$ by a truncation parameter $\delta$, and in fact, we will restrict $\varepsilon\in(0,\,\delta/4)$ when showing a priori H\"{o}lder continuity of ${\mathcal G}_{2\delta,\,\varepsilon}(Du_{\varepsilon})$. This restriction will be necessary because we have to deal with two different symmetric moduli $\lvert Du_{\varepsilon}\rvert$ and $V_{\varepsilon}=\sqrt{\varepsilon^{2}+\lvert Du_{\varepsilon}\rvert^{2}}$.

\subsection{Our generalized results and outlines of the paper}\label{Subsect: Generalized results in system}
Throughout this paper, we let $p\in(1,\,\infty)$, and the $p$-Laplace-type operator ${\mathcal L}_{p}$ is assumed to be of the form
\[{\mathcal L}_{p}u=\divx\mleft(g_{p}^{\prime}(\lvert Du\rvert^{2})Du\mright).\]
Here $g_{p}\in C(\lbrack0,\,\infty))\cap C_{{\mathrm{loc}}}^{2,\,\beta_{0}}((0,\,\infty))$ is a non-negative function with $\beta_{0}\in(0,\,1\rbrack$. Moreover, there exist constants $0<\gamma\le \Gamma<\infty$ such that there hold
\begin{equation}\label{Eq (Section 1) Growth g-p-prime}
\lvert g_{p}^{\prime}(\sigma)\rvert\le \Gamma \sigma^{p/2-1}\quad \textrm{for all }\sigma\in(0,\,\infty),
\end{equation}
\begin{equation}\label{Eq (Section 1) Growth g-p-pprime}
\lvert g_{p}^{\prime\prime}(\sigma)\rvert\le \Gamma \sigma^{p/2-2}\quad \textrm{for all }\sigma\in(0,\,\infty),
\end{equation}
\begin{equation}\label{Eq (Section 1) Ellipticity g_p}
\gamma (\sigma+\tau)^{p/2-1}\le g_{p}^{\prime}(\sigma+\tau)+2\sigma\min\mleft\{\,g_{p}^{\prime\prime}(\sigma+\tau),\,0\,\mright\} \quad \textrm{for all }\sigma\in\lbrack 0,\,\infty),\,\tau\in(0,\,1).
\end{equation}
For $\beta_{0}$-H\"{o}lder continuity of $g_{p}^{\prime\prime}$, we assume that
\begin{equation}\label{Eq (Section 1) Growth g-p-ppprime}
\mleft\lvert g_{p}^{\prime\prime}(\sigma_{1})- g_{p}^{\prime\prime}(\sigma_{2})\mright\rvert\le \Gamma\mu^{p-4-2\beta_{0}}\lvert \sigma_{1}-\sigma_{2}\rvert^{\beta_{0}}
\end{equation}
for all $\sigma_{1},\,\sigma_{2}\in\lbrack\mu^{2}/4,\,7\mu^{2}\rbrack$ with $\mu\in(0,\,\infty)$.

A typical example is
\begin{equation}\label{Eq (Section 1) gp example}
g_{p}(\sigma)\coloneqq \frac{2\sigma^{p/2}}{p}\quad \textrm{for }\sigma\in\lbrack 0,\,\infty).
\end{equation}
In fact, this $g_{p}$ satisfies (\ref{Eq (Section 1) Growth g-p-prime})--(\ref{Eq (Section 1) Ellipticity g_p}) with \(\gamma\coloneqq \min\{\,1,\,p-1\,\},\quad\Gamma\coloneqq \max\mleft\{\,1,\,\lvert p-2\rvert/2\,\mright\}\). Moreover, we have
\[\lvert g_{p}^{\prime\prime\prime}(\sigma)\rvert\le {\tilde\Gamma} \sigma^{p/2-3}\quad \textrm{for all }\sigma\in(0,\,\infty)\]
with ${\tilde\Gamma}\coloneqq \lvert (p-2)(p-4)\rvert/4$. From this estimate, it is easy to find a constant $\Gamma=\Gamma(p)\in(1,\,\infty)$ such that (\ref{Eq (Section 1) Growth g-p-ppprime}) holds with $\beta_{0}=1$.
In this case, the operator ${\mathcal L}_{p}$ becomes $\Delta_{p}$. Thus, the system (\ref{Eq (Section 1) main system}) generalizes (\ref{Eq (Section 1) 1+p-Laplace}). 

Our main result is the following Theorem \ref{Theorem: C1-regularity}, which clearly yields Theorem \ref{Theorem: A weak solution is continuously differentiable}.
\begin{theorem}\label{Theorem: C1-regularity}
Let $p\in(1,\,\infty),\,q\in(n,\,\infty\rbrack$, and $f\in L^{q}(\Omega;\,{\mathbb R}^{N})$. Assume that $u$ is a weak solution to (\ref{Eq (Section 1) main system}) in a Lipschitz domain $\Omega\subset{\mathbb R}^{n}$, where $g_{p}$ satisfies (\ref{Eq (Section 1) Growth g-p-prime})--(\ref{Eq (Section 1) Growth g-p-ppprime}). Then, for each fixed $\delta\in(0,\,1)$ and $x_{\ast}\in\Omega$, there exists an open ball $B_{\rho_{0}}(x_{\ast})\Subset\Omega$ such that ${\mathcal G}_{2\delta}(Du)\in C^{\alpha}(B_{\rho_{0}/2}(x_{\ast});\,{\mathbb R}^{Nn})$. Here the exponent $\alpha\in(0,\,1)$ and the radius $\rho_{0}\in(0,\,1)$ depend at most on $b$, $n$, $N$, $p$, $q$, $\beta_{0}$, $\gamma$, $\Gamma$, $\lVert f\rVert_{L^{q}(\Omega)}$, $\lVert Du\rVert_{L^{p}(\Omega)}$, $d_{\ast}\coloneqq\dist (x_{\ast},\,\partial\Omega)$, and $\delta$.
Moreover, we have
\begin{equation}\label{Eq (Section 1) local bounds of Jacobian multiplied with G-delta}
\mleft\lvert {\mathcal G}_{2\delta}(Du(x))\mright\rvert\le \mu_{0}\quad \textrm{for all }x\in B_{\rho_{0}}(x_{\ast}),
\end{equation}
\begin{equation}\label{Eq (Section 1) local continuity of Jacobian multiplied with G-delta}
\mleft\lvert {\mathcal G}_{2\delta}(Du(x_{1}))-{\mathcal G}_{2\delta}(Du(x_{2}))\mright\rvert\le \frac{2^{n/2+2\alpha+2}\mu_{0}}{\rho_{0}^{\alpha}}\lvert x_{1}-x_{2}\rvert^{\alpha}\quad \textrm{for all }x_{1},\,x_{2}\in B_{\rho_{0}/2}(x_{\ast}),
\end{equation}
where the constant $\mu_{0}\in(0,\,\infty)$ depends at most on $b$, $n$, $p$, $q$, $\gamma$, $\Gamma$, $\lVert f\rVert_{L^{q}(\Omega)}$, $\lVert Du\rVert_{L^{p}(\Omega)}$, and $d_{\ast}$.
In particular, the Jacobian matrix $Du$ is continuous in \(\Omega\).
\end{theorem}

This paper is organized as follows.

Section \ref{Section: Approximation} provides approximation problems for the system (\ref{Eq (Section 1) main system}), which is based on the relaxation of energy densities. After fixing some notations in Section \ref{Subsect: Notation}, we give a variety of quantitative estimates related to the relaxed densities in Section \ref{Subsect: Basic estimates}. Next in Section \ref{Subsect: Convergence}, we will justify that solutions of regularized problem, denoted by $u_{\varepsilon}$, converges to the original problem (\ref{Eq (Section 1) main system}). 
Our main theorems are proved in Section \ref{Subsect: A priori Hoelder}, which presents a priori H\"{o}lder estimates on Jacobian matrices $Du_{\varepsilon}$ that are suitably truncated near facets. There we state three basic a priori estimates on approximated solutions, consisting of local Lipschitz bounds (Proposition \ref{Prop: Lipschitz bounds}), a De Giorgi-type oscillation lemma (Proposition \ref{Prop: De Giorgi's truncation}) and Campanato-type growth estimates (Proposition \ref{Prop: Schauder estimate}). The proofs of Propositions \ref{Prop: Lipschitz bounds}--\ref{Prop: Schauder estimate} are given later in the remaining Sections \ref{Section: Weak formulations}--\ref{Section: Appendix}. From these estimates, we will deduce local a priori H\"{o}lder estimates of ${\mathcal G}_{2\delta,\,\varepsilon}(Du_{\varepsilon})$, which are independent of an approximation parameter $\varepsilon\in(0,\,\delta/4)$ (Theorem \ref{Theorem: A priori Hoelder estimate}). From Proposition \ref{Prop: Convergence on weak system} and Theorem \ref{Theorem: A priori Hoelder estimate}, we finally give the proof of Theorem \ref{Theorem: C1-regularity}.

Sections \ref{Section: Weak formulations}--\ref{Section: Appendix} are devoted to prove Propositions \ref{Prop: Lipschitz bounds}--\ref{Prop: Schauder estimate}. 
Among them, most of the proofs of Propositions \ref{Prop: Lipschitz bounds}--\ref{Prop: De Giorgi's truncation} are rather easier.
For the reader's convenience, we would like to provide a brief proof of Proposition \ref{Prop: Lipschitz bounds} in the appendix (Section \ref{Section: Appendix}). For Proposition \ref{Prop: De Giorgi's truncation}, we briefly describe sketches of the proof in Section \ref{Subsect: De Giorgi's truncation} after showing a basic weak formulation in Section \ref{Subsect: Basic weak formulations}. We omit some standard arguments in Section \ref{Subsect: De Giorgi's truncation} concerning De Giorgi's levelset lemmata, since they are already found in \cite[\S 7]{BDGPdN} (see also \cite[\S 4.2]{T-scalar}).
Section \ref{Section: Campanato estimates} is focused on the proof of Proposition \ref{Prop: Schauder estimate}. In Section \ref{Subsect: Energy estimates}, we obtain a variety of energy estimates from the weak formulation deduced in Section \ref{Subsect: Basic weak formulations}. Section \ref{Subsect: freezing coefficient method} establishes a freezing coefficient argument under a good condition telling that a derivative does not vanish. After providing two basic lemmata on our shrinking arguments in Section \ref{Subsect: Shrinking lemmata}, we give the proof of Proposition \ref{Prop: Schauder estimate} in Section \ref{Subsect: Proof of Campanato-decay}.

Finally, we would like to mention again that $C^{1}$-regularity for the scalar case is treated in the author's recent work \cite{T-scalar}.
There, approximation of the density function is discussed, which is based on the convolution of Friedrichs' mollifier and makes it successful to generalize one-Laplace operator. It should be emphasized that this generalization may not work in the system case, when it comes to everywhere continuous differentiability. This is essentially because the general singular operator discussed therein will lack the Uhlenbeck structure, except the one-Laplacian (see \cite[\S 2.4]{T-scalar} for further explanations).
Also, compared with \cite{T-scalar}, some of the computations in this paper becomes rather simpler, since we consider a special case where the energy density is spherically symmetric.
Some estimates in Section \ref{Subsect: Basic estimates}--\ref{Subsect: Convergence} are used without proofs and some proofs in Section \ref{Subsect: De Giorgi's truncation} are omitted. The full computations and proofs of them are given in \cite[\S 2.1 \& 4.2]{T-scalar}.

\section{Approximation schemes}\label{Section: Approximation}
The aim of Section \ref{Section: Approximation} is to give approximation schemes for the problem (\ref{Eq (Section 1) main system}). The basic idea is that the energy density 
\begin{equation}
E(\xi)=\frac{1}{2}g(\lvert \xi\rvert^{2})\quad (\xi\in{\mathbb R}^{Nn})\quad \textrm{with}\quad
g(\sigma)\coloneqq 2b\sqrt{\sigma}+g_{p}(\sigma)\quad(0\le \sigma<\infty)
\end{equation}
is to be regularized by 
\begin{equation}\label{Eq (Section 2) Relaxed Energy Density}E_{\varepsilon}(\xi)=\frac{1}{2}g_{\varepsilon}(\lvert \xi\rvert^{2})\,(\xi\in{\mathbb R}^{Nn})\quad \textrm{with}\quad
g_{\varepsilon}(\sigma)\coloneqq 2b\sqrt{\varepsilon^{2}+\sigma}+g_{p}(\varepsilon^{2}+\sigma)\quad (0\le \sigma<\infty)
\end{equation}
for $\varepsilon\in(0,\,1)$ denoting the approximation parameter.
Then, the relaxed operator ${\mathcal L}^{\varepsilon}$ will be given by
\[{\mathcal L}^{\varepsilon}u_{\varepsilon}\coloneqq-\divx\mleft(g_{\varepsilon}^{\prime}(\lvert Du_{\varepsilon}\rvert^{2})Du_{\varepsilon} \mright)\equiv-b\divx\mleft(\frac{Du_{\varepsilon}}{\sqrt{\varepsilon^{2}+\lvert Du_{\varepsilon}\rvert^{2}}}\mright)-\divx\mleft(g_{p,\,\varepsilon}^{\prime}(\lvert Du_{\varepsilon}\rvert^{2})Du_{\varepsilon} \mright),\]
similarly to $L_{b,\,p}$ given in (\ref{Eq (Section 1) Approximating problem special case}). In Section \ref{Section: Approximation}, we will show that this approximation scheme works and deduce some basic estimates on relaxed mapping.
\subsection{Notations}\label{Subsect: Notation}
We first fix some notations throughout the paper.

We denote ${\mathbb Z}_{\ge 0}\coloneqq \{\,0,\,1,\,2,\,\dots\,\,\}$ by the set of all non-negative integers, and ${\mathbb N}\coloneqq {\mathbb Z}_{\ge 0}\setminus\{ 0\}$ by the set of all natural numbers.

For given $k\in{\mathbb N}$, we denote $\langle\,\cdot\mid\cdot\,\rangle_{k}$ by the standard inner product over the Euclidean space ${\mathbb R}^{k}$. That is, for $k$-dimensional vectors $\xi=(\xi_{1},\,\dots\,,\,\xi_{k}),\,\eta=(\eta_{1},\,\dots\,,\,\eta_{k})\in{\mathbb R}^{k}$, we define 
\[\langle \xi\mid\eta\rangle_{k}\coloneqq \sum_{j=1}^{k}\xi_{j}\eta_{j}\in{\mathbb R}^{k}.\]
We denote $\lvert\,\cdot\,\rvert_{k}$ by the Euclidean norm,
\[\textrm{i.e., } \lvert \xi\rvert_{k}\coloneqq \sqrt{\langle\xi\mid\xi\rangle}\in\lbrack0,\,\infty)\quad \textrm{for }\xi\in{\mathbb R}^{k}.\]
We also define a tensor product $\xi\otimes \eta$ as a $k\times k$ real matrix,
\[\textrm{i.e., }\xi\otimes\eta\coloneqq (\xi_{j}\eta_{k})_{j,\,k}=\begin{pmatrix}\xi_{1}\eta_{1} & \dots &\xi_{1}\eta_{k}\\ \vdots & \ddots & \vdots\\ \xi_{k}\eta_{1}&\dots &\xi_{k}\eta_{k}\end{pmatrix}.\]

We denote $\mathrm{id}_{k}$ by the $k\times k$ identity matrix. For a $k\times k$ real matrix $A$, we define the operator norm
\[\lVert A\rVert_{k}\coloneqq \sup\mleft\{\lvert Ax\rvert_{k}\mathrel{}\middle| \mathrel{}x\in{\mathbb R}^{k},\,\lvert x\rvert_{k}\le 1\mright\}.\]
We also introduce the positive semi-definite ordering on all real symmetric matrices. In other words, for $k\times k$ real symmetric matrices $A$, $B$, we write $A\leqslant B$ or $B\geqslant A$ when the difference $B-A$ is positive semi-definite.

For notational simplicity, we often omit the script $k$ denoting the dimension. In particular, we simply denote the norms by $\lvert\,\cdot\,\rvert$ or $\lVert\,\cdot\,\rVert$, and the inner product by $\langle\,\cdot\mid\cdot\,\rangle$.

We often regard an $N\times n$ real matrix $\xi=(\xi_{\alpha}^{i})_{1\le \alpha\le n,\,1\le i\le N}$ as an $Nn$-dimensional vector by ordering $\xi=(\xi_{1}^{1},\,\dots\,,\,\xi_{1}^{N},\,\dots\,,\,\xi_{n}^{1},\,\dots\,,\,\xi_{n}^{N})\in{\mathbb R}^{Nn}$. Then, similarly to the above, for given $N\times n$ matrices $\xi=(\xi_{\alpha}^{i}),\,\eta=(\eta_{\alpha}^{i})$, we are able to define the inner product $\langle \xi\mid\eta\rangle_{Nn}\in{\mathbb R}$, and a tensor product $\xi\otimes\eta$ as an $(Nn)\times (Nn)$ real-valued matrix.
We also note that under our setting, the norm $\lvert \xi\rvert$ of a $N\times n$ real matrix $\xi$ is identified with the Frobenius norm of $\xi$. 

For a scalar-valued function $u=u(x_{1},\,\dots\,,\,x_{n})$, the gradient of $u$ is denoted by \(\nabla u\coloneqq (\partial_{x_{\alpha}}u)_{\alpha}\) and is often regarded as an $n$-dimensional row vector. For an ${\mathbb R}^{k}$-valued function $u=(u^{1},\,\dots\,,\,u^{k})$ with $u^{i}=u^{i}(x_{1},\,\dots\,,\,x_{n})$ for $i\in\{\,1,\,\dots\,,\,k\,\}$, the Jacobian matrix of $u$ is denoted by
\(Du\coloneqq \begin{pmatrix} D_{1} u & \cdots & D_{n}u\end{pmatrix}\equiv (\partial_{x_{\alpha}}u^{i})_{\alpha,\,i}\)
with $D_{\alpha}u=(\partial_{x_{\alpha}}u^{i})_{i}$ for each $\alpha\in\{\,1,\,\dots\,,\,n\,\}$.
These $Du$ and $D_{\alpha}u$ are often considered as an $k\times n$-matrix and a $k$-dimentional column vector respectively.

For given numbers $s\in\lbrack 1,\,\infty\rbrack$, $k\in{\mathbb N}$, $d\in{\mathbb N}$ and a fixed domain $U\subset {\mathbb R}^{n}$, we denote $L^{s}(U;\,{\mathbb R}^{d})$ and $W^{k,\,s}(U;\,{\mathbb R}^{d})$ respectively by the Lebesgue space and the Sobolev space. To shorten the notations, we often write $L^{s}(U)\coloneqq L^{s}(U;\,{\mathbb R})$ and $W^{k,\,s}(U)\coloneqq W^{k,\,s}(U;\,{\mathbb R})$.

Throughout this paper, we define 
\[g_{1}(\sigma)\coloneqq 2b\sqrt{\sigma}\quad \textrm{for }\sigma\in\lbrack0,\,\infty).\]
Clearly, $g_{1}$ is in $C(\lbrack0,\,\infty))\cap C^{3}((0,\,\infty))$ and satisfies
\begin{equation}\label{Eq (Section 2) Growth g-1-prime}
\mleft\lvert g_{1}^{\prime}(\sigma)\mright\rvert\le b\sigma^{-1/2}\quad \textrm{for all }\sigma\in(0,\,\infty),
\end{equation}
\begin{equation}\label{Eq (Section 2) Growth g-1-pprime}
\mleft\lvert g_{1}^{\prime\prime}(\sigma)\mright\rvert\le \frac{b}{2}\sigma^{-3/2}\quad \textrm{for all }\sigma\in(0,\,\infty),
\end{equation}\begin{equation}\label{Eq (Section 2) Degenerate ellipticity of g-1}
0\le g_{1}^{\prime}(\sigma+\tau)+2\sigma\min\mleft\{\,g_{1}^{\prime\prime}(\sigma+\tau),\,0\,\mright\}\quad \textrm{for all }\sigma\in(0,\,\infty),\,\tau\in(0,\,1).
\end{equation}
Also, by the growth estimate
\[\mleft\lvert g_{1}^{\prime\prime\prime}(\sigma)\mright\rvert\le \frac{3b}{4}\sigma^{-5/2}\quad \textrm{for all }\sigma\in(0,\,\infty),\]
it is easy to check that there holds
\begin{equation}\label{Eq (Section 2) Growth g-1-ppprime}
\mleft\lvert g_{1}^{\prime\prime}(\sigma_{1})-g_{1}^{\prime\prime}(\sigma_{2})\mright\rvert\le 24\cdot b\mu^{-5}\lvert \sigma_{1}-\sigma_{2}\rvert
\end{equation}
for all $\sigma_{1},\,\sigma_{2}\in\lbrack\mu^{2}/4,\,7\mu^{2}\rbrack$ with $\mu\in(0,\,\infty)$.
It should be mentioned that unlike the assumption (\ref{Eq (Section 1) Ellipticity g_p}), the inequality (\ref{Eq (Section 2) Degenerate ellipticity of g-1}) gives no quantitative monotonicity estimates for the one-Laplace operator $\Delta_{1}$. 

We consider the relaxed function $g_{\varepsilon}$ given by (\ref{Eq (Section 2) Relaxed Energy Density}).
This function can be decomposed by $g_{\varepsilon}=g_{1,\,\varepsilon}+g_{p,\,\varepsilon}$ with
\[g_{1,\,\varepsilon}(\sigma)\coloneqq 2b\sqrt{\varepsilon^{2}+\sigma},\quad g_{p,\,\varepsilon}(\sigma)\coloneqq g_{p}(\varepsilon^{2}+\sigma)\]
for $\sigma\in\lbrack 0,\,\infty)$. Corresponding to them, for each $s\in\{\,1,\,p\,\}$, we define 
\[A_{s,\,\varepsilon}(\xi)\coloneqq g_{s,\,\varepsilon}^{\prime}(\lvert \xi\rvert^{2})\xi\quad \textrm{for }\xi\in{\mathbb R}^{Nn}\]
as an $N\times n$ matrix, and
\[{\mathcal B}_{s,\,\varepsilon}(\xi)\coloneqq g_{\varepsilon}^{\prime}(\lvert \xi\rvert^{2}){\mathrm{id}_{Nn}}+2g_{s,\,\varepsilon}^{\prime\prime}(\lvert \xi\rvert^{2})\xi\otimes\xi\quad \textrm{for }\xi\in{\mathbb R}^{Nn}\]
as an $Nn\times Nn$ matrix. Then, the summations
\begin{equation}\label{Eq (Section 2) Def of A-epsilon}
A_{\varepsilon}(\xi)\coloneqq A_{1,\,\varepsilon}(\xi)+A_{p,\,\varepsilon}(\xi)\quad \textrm{for } \xi\in{\mathbb R}^{Nn}
\end{equation}
and
\begin{equation}\label{Eq (Section 2) Def of B-epsilon}
{\mathcal B}_{\,\varepsilon}(\xi)\coloneqq {\mathcal B}_{1,\,\varepsilon}(\xi)+{\mathcal B}_{p,\,\varepsilon}(\xi) \quad \textrm{for }\xi\in{\mathbb R}^{Nn}
\end{equation}
respectively denote the Jacobian matrices and the Hessian matrices of $E_{\varepsilon}$ defined by (\ref{Eq (Section 2) Relaxed Energy Density}). By direct calculations, it is easy to check that $A_{p,\,\varepsilon}(\xi)$ converges to $A_{p}(\xi)$ for each $\xi\in{\mathbb R}^{Nn}$, where
\begin{equation}\label{Eq (Section 2) Def of A-p}
A_{p}(\xi)\coloneqq \mleft\{\begin{array}{cc}
g_{p}^{\prime}(\lvert\xi\rvert^{2})\xi& (\xi\neq 0),\\ 0 & (\xi=0).
\end{array} \mright.
\end{equation}

Finally, we introduce a subdifferential set for the Euclid norm.
We denote $\partial\lvert\,\cdot\,\rvert(\xi_{0})\subset{\mathbb R}^{Nn}$ by the subdifferential of the absolute value function $\lvert\,\cdot\,\rvert_{Nn}$ at $\xi_{0}\in{\mathbb R}^{Nn}$. In other words, $\partial\lvert\,\cdot\,\rvert(\xi_{0})$ is the set of all vectors $\zeta\in{\mathbb R}^{Nn}$ that satisfies a subgradient inequality
\[\lvert \xi\rvert_{Nn}\ge \lvert\xi_{0}\rvert_{Nn}+\langle \zeta \mid \xi-\xi_{0}\rangle_{Nn}\quad \textrm{for all }\xi\in{\mathbb R}^{Nn}.\]
This set is explicitly given by
\begin{equation}\label{Eq: sub}
\partial\lvert\,\cdot\,\rvert(\xi_{0})=\mleft\{\begin{array}{cc} \mleft\{\zeta\in{\mathbb R}^{Nn}\mathrel{}\middle|\mathrel{}\lvert \zeta\rvert\le 1\mright\} & (\xi_{0}=0),\\ \{\xi_{0}/\lvert \xi_{0}\rvert\} & (\xi_{0}\neq 0), \end{array}\mright.\end{equation}
and this formulation is used when one gives the definitions of weak solutions.
\subsection{Quantitative estimates on relaxed mappings}\label{Subsect: Basic estimates}
In Section \ref{Subsect: Basic estimates}, we would like to introduce quantitative estimates related to the mappings ${\mathcal G}_{2\delta,\,\varepsilon}$ and $g_{\varepsilon}$. When deducing these estimates, we often use the assumption that the density function $E$ is spherically symmetric. As a related item, we refer the reader to \cite[\S 2.1]{T-scalar}, which deals with the scalar case without symmetry of $E$, and provides full computations of some estimates omitted in Section \ref{Subsect: Basic estimates}.

In this paper, we often assume that
\begin{equation}\label{Eq (Section 2) delta-epsilon}
0<\delta<1,\quad \textrm{and}\quad 0<\varepsilon<\frac{\delta}{4}.
\end{equation} 
It should be mentioned that the mapping ${\mathcal G}_{\delta,\,\varepsilon}\colon{\mathbb R}^{Nn}\rightarrow {\mathbb R}^{Nn}$ given by (\ref{Eq (Section 1) relaxed truncation}) makes sense as long as $0<\varepsilon<\delta$ holds. In particular, under the setting (\ref{Eq (Section 2) delta-epsilon}), the mappings ${\mathcal G}_{\delta,\,\varepsilon},\,{\mathcal G}_{2\delta,\,\varepsilon}$ are well-defined. Moreover, Lipschitz continuity of ${\mathcal G}_{2\delta,\,\varepsilon}$ follows from (\ref{Eq (Section 2) delta-epsilon}).
\begin{lemma}\label{Lemma: G-2delta-epsilon}
Let $\delta,\,\varepsilon$ satisfy (\ref{Eq (Section 2) delta-epsilon}). Then the mapping ${\mathcal G}_{2\delta,\,\varepsilon}$ satisfies
\begin{equation}\label{Eq (Section 2) Lipschitz bounds of the mapping G-2delta-epsilon}
\left\lvert{\mathcal G}_{2\delta,\,\varepsilon}(\xi_{1})-{\mathcal G}_{2\delta,\,\varepsilon}(\xi_{2})\right\rvert\le c_{\dagger}\lvert \xi_{1}-\xi_{2}\rvert\quad \textrm{for all }z_{1},\,z_{2}\in{\mathbb R}^{Nn}
\end{equation}with $c_{\dagger}\coloneqq 1+32/(3\sqrt{7})$.
\end{lemma}
In the limiting case $\varepsilon=0$, Lipschitz continuity like (\ref{Eq (Section 2) Lipschitz bounds of the mapping G-2delta-epsilon}) is found in \cite[Lemma 2.3]{BDGPdN}. Modifying the arguments therein, we can easily prove Lemma \ref{Lemma: G-2delta-epsilon}, the full proof of which is given in \cite[Lemma 2.4]{T-scalar}. 

Next, we consider the mappings $A_{\varepsilon}=A_{1,\,\varepsilon}+A_{p,\,\varepsilon}$ and ${\mathcal B}_{\varepsilon}={\mathcal B}_{1,\,\varepsilon}+{\mathcal B}_{p,\,\varepsilon}$ defined by (\ref{Eq (Section 2) Def of A-epsilon})--(\ref{Eq (Section 2) Def of B-epsilon}), and describe some basic results including monotonicity and growth estimates.
For each $s\in \{\,1,\,p\,\}$, the eigenvalues of ${\mathcal B}_{s,\,\varepsilon}(\xi)$ are given by either
\[\lambda_{1}(\xi)\coloneqq g_{s}^{\prime}(\varepsilon^{2}+\lvert\xi\rvert^{2})\quad\textrm{or}\quad \lambda_{2}(\xi)\coloneqq g_{s}^{\prime}(\varepsilon^{2}+\lvert\xi\rvert^{2})+2\lvert\xi\rvert^{2}g_{s}^{\prime\prime}(\varepsilon^{2}+\lvert\xi\rvert^{2}).\]
Combining this with (\ref{Eq (Section 1) Growth g-p-prime})--(\ref{Eq (Section 1) Ellipticity g_p}) and (\ref{Eq (Section 2) Growth g-1-prime})--(\ref{Eq (Section 2) Degenerate ellipticity of g-1}), the mappings ${\mathcal B}_{p,\,\varepsilon}$ and ${\mathcal B}_{1,\,\varepsilon}$ respectively satisfy
\begin{equation}\label{Eq (Section 2) Estimates on B-p-epsilon}
\gamma(\varepsilon^{2}+\lvert\xi\rvert^{2})^{p/2-1}\mathrm{id}_{Nn}\leqslant {\mathcal B}_{p,\,\varepsilon}(\xi) \leqslant 3\Gamma(\varepsilon^{2}+\lvert\xi\rvert^{2})^{p/2-1} \mathrm{id}_{Nn}\quad \textrm{for all }\xi\in{\mathbb R}^{Nn},
\end{equation}
and
\begin{equation}\label{Eq (Section 2) Estimates on B-1-epsilon}
O\leqslant {\mathcal B}_{1,\,\varepsilon}(\xi) \leqslant b(\varepsilon^{2}+\lvert\xi\rvert^{2})^{-1/2} \mathrm{id}_{Nn}\quad \textrm{for all }\xi\in{\mathbb R}^{Nn},
\end{equation}
where $O$ denotes the zero matrix. In particular, by elementary computations as in \cite[Lemma 3]{MR4201656}, it is easy to get
\begin{equation}\label{Eq (Section 2) Monotonicity estimate for A-p-epsilon}
\mleft\langle A_{\varepsilon}(\xi_{1})-A_{\varepsilon}(\xi_{0})\mathrel{}\middle|\mathrel{} \xi_{1}-\xi_{0}\mright\rangle\ge \mleft\{\begin{array}{cc}
c(p)\gamma\mleft(\varepsilon^{2}+\lvert\xi_{0}\rvert^{2}+\lvert\xi_{1}\rvert^{2}\mright)^{(p-1)/2}\lvert\xi_{1}-\xi_{0}\rvert^{2} & (1<p<2),\\ c(p)\gamma\lvert\xi_{1}-\xi_{0}\rvert^{p} & (2\le p<\infty),
\end{array}\mright.
\end{equation}
and
\begin{equation}\label{Eq (Section 2) Growth estimate for A-p-epsilon}
\mleft\lvert A_{p,\,\varepsilon}(\xi_{1})-A_{p,\,\varepsilon}(\xi_{0}) \mright\rvert\le \mleft\{\begin{array}{cc}
C(p)\Gamma\lvert\xi_{1}-\xi_{0}\rvert^{p-1} & (1<p<2),\\ C(p)\Gamma\mleft(\varepsilon^{p-2}+\lvert\xi_{0}\rvert^{p-2}+\lvert\xi_{1}\rvert^{p-2}\mright)\lvert\xi_{1}-\xi_{0}\rvert & (2\le p<\infty),
\end{array} \mright.
\end{equation}
for all $\xi_{0},\,\xi_{1}\in{\mathbb R}^{Nn}$. 
We often consider a special case where a variable $\xi\in{\mathbb R}^{Nn}$ may not vanish. On this setting, it is possible to deduce continuity or monotonicity estimates other than (\ref{Eq (Section 2) Monotonicity estimate for A-p-epsilon})--(\ref{Eq (Section 2) Growth estimate for A-p-epsilon}). In fact, following elementary computations given in \cite[Lemmata 2.2--2.3]{T-scalar}, from (\ref{Eq (Section 2) Estimates on B-p-epsilon}), we are able to obtain  a growth estimate 
\begin{equation}\label{Eq (Section 2) Growth estimate for A-s-epsilon with 1<s<2}
\mleft\lvert A_{p,\,\varepsilon}(\xi_{1})-A_{p,\,\varepsilon}(\xi_{0}) \mright\rvert\le C(p)\Gamma\min\mleft\{\,\lvert\xi_{0}\rvert^{p-2},\,\lvert\xi_{1}\rvert^{p-2}\mright\}\lvert\xi_{1}-\xi_{0}\rvert
\end{equation}
for all $(\xi_{0},\,\xi_{1})\in({\mathbb R}^{Nn}\times{\mathbb R}^{Nn})\setminus\{(0,\,0)\}$ provided $1<p<2$, and a monotonicity estimate
\begin{equation}\label{Eq (Section 2) Growth estimate for A-s-epsilon with s>2}
\mleft\langle A_{p,\,\varepsilon}(\xi_{1})-A_{p,\,\varepsilon}(\xi_{0})\mathrel{}\middle|\mathrel{}\xi_{1}-\xi_{0} \mright\rangle\ge C(p)\gamma\max\mleft\{\,\lvert \xi_{0}\rvert^{p-2},\,\lvert\xi_{1}\rvert^{p-2}\,\mright\}\lvert \xi_{1}-\xi_{0}\rvert^{2}
\end{equation}
for all $\xi_{0},\,\xi_{1}\in{\mathbb R}^{Nn}$ provided $2\le p<\infty$. It is worth mentioning that even for $A_{1,\,\varepsilon}$, there holds a growth estimate
\begin{equation}\label{Eq (Section 2) Growth estimate for A-s-epsilon with s=1}
\mleft\lvert A_{1,\,\varepsilon}(\xi_{1})-A_{1,\,\varepsilon}(\xi_{0}) \mright\rvert\le 2b\min\mleft\{\,\lvert \xi_{0}\rvert^{-1},\,\lvert \xi_{1}\rvert^{-1} \,\mright\}\lvert \xi_{1}-\xi_{0}\rvert
\end{equation}
for all $(\xi_{0},\,\xi_{1})\in({\mathbb R}^{Nn}\times {\mathbb R}^{Nn})\setminus\{(0,\,0)\}$ (see also \cite[Lemma 2.1]{T-scalar}). In fact, without loss of generality, we may let $\lvert \xi_{1}\rvert\ge \lvert \xi_{0}\rvert$. By the triangle inequality, we compute
\[\mleft\lvert \frac{\xi_{1}}{\sqrt{\varepsilon^{2}+\lvert \xi_{1}\rvert^{2}}} -\frac{\xi_{0}}{\sqrt{\varepsilon^{2}+\lvert \xi_{0}\rvert^{2}}} \mright\rvert= \mleft\lvert \frac{\xi_{1}-\xi_{0}}{\sqrt{\varepsilon^{2}+\lvert \xi_{1}\rvert^{2}}}-\frac{\sqrt{\varepsilon^{2}+\lvert \xi_{1}\rvert^{2}}-\sqrt{\varepsilon^{2}+\lvert \xi_{0}\rvert^{2}}}{\sqrt{\varepsilon^{2}+\lvert \xi_{1}\rvert^{2}}\cdot \sqrt{\varepsilon^{2}+\lvert \xi_{0}}\rvert^{2}}\xi_{0}  \mright\rvert\le 2\frac{\lvert \xi_{1}-\xi_{0}\rvert}{\sqrt{\varepsilon^{2}+\lvert\xi_{1}\rvert^{2}}},\]
from which (\ref{Eq (Section 2) Growth estimate for A-s-epsilon with s=1}) follows. Here we have used one-Lipschitz continuity of the smooth function $\sqrt{\varepsilon^{2}+t^{2}}\,(t\in{\mathbb R})$.
For monotonicity of $A_{1,\,\varepsilon}$, from (\ref{Eq (Section 2) Estimates on B-1-epsilon}) we obtain
\begin{equation}\label{Eq (Section 2) Monotonicity of A-1-epsilon}
\mleft\langle A_{1,\,\varepsilon}(\xi_{1})-A_{1,\,\varepsilon}(\xi_{0})\mathrel{}\middle|\mathrel{}\xi_{1}-\xi_{0} \mright\rangle\ge 0
\end{equation}
for all $\xi_{0},\,\xi_{1}\in{\mathbb R}^{Nn}$. When it comes to a monotonicity estimate that is independent of $\varepsilon$, better estimates than (\ref{Eq (Section 2) Monotonicity of A-1-epsilon}) seem to be no longer expectable, since ellipticity of $\Delta_{1}u$ degenerates in the direction of $Du$.

In Lemma \ref{Lemma: Error estimates} below, we briefly deduce good estimates that will be used in Section \ref{Subsect: freezing coefficient method}.
\begin{lemma}\label{Lemma: Error estimates}
Let positive constants $\delta$ and $\varepsilon$ satisfy (\ref{Eq (Section 2) delta-epsilon}). Under the assumptions (\ref{Eq (Section 1) Growth g-p-prime})--(\ref{Eq (Section 1) Growth g-p-ppprime}), the mappings $A_{\varepsilon}$ and ${\mathcal B}_{\varepsilon}$ defined by (\ref{Eq (Section 2) Def of A-epsilon})--(\ref{Eq (Section 2) Def of B-epsilon}) satisfy the following:
\begin{enumerate}
\item \label{Item 1/2 (Section 2) Monotonicity and Growth estimates}
Let $M\in(\delta,\,\infty)$ be a fixed constant. Then there exists constants $C_{1},\,C_{2}\in(0,\,\infty)$, depending at most on $b$, $p$, $\gamma$, $\Gamma$, $\delta$ and $M$, such that we have
\begin{equation}\label{Eq (Section 2) Monotonicity outside}
\mleft\langle A_{\varepsilon}(\xi_{1})-A_{\varepsilon}(\xi_{0})\mathrel{}\middle|\mathrel{}\xi_{1}-\xi_{0}\mright\rangle\ge C_{1}\lvert \xi_{1}-\xi_{0}\rvert^{2},
\end{equation}
and
\begin{equation}\label{Eq (Section 2) Growth outside}
\lvert A_{\varepsilon}(\xi_{1})-A_{\varepsilon}(\xi_{0})\rvert\le C_{2}\lvert \xi_{1}-\xi_{0}\rvert,
\end{equation}
for all $\xi_{0},\,\xi_{1}\in{\mathbb R}^{Nn}$ satisfying
\[\delta\le\lvert \xi_{0}\rvert\le M,\quad \textrm{and}\quad \lvert \xi_{1}\rvert\le M.\]
\item \label{Item 2/2 (Section 2) Hessian Errors}
For all $\xi_{0},\,\xi_{1}\in{\mathbb R}^{Nn}$ enjoying
\begin{equation}\label{Eq (Section 2) Variable conditions in error estimates}
\delta+\frac{\mu}{4}\le\lvert\xi_{0}\rvert\le \delta+\mu\quad \textrm{and}\quad \lvert\xi_{1}\rvert\le\delta+\mu\quad \textrm{with}\quad \delta<\mu<\infty,
\end{equation}
we have
\begin{equation}\label{Eq (Section 2) Hessian errors}
\mleft\lvert {\mathcal B}_{\varepsilon}(\xi_{0})(\xi_{1}-\xi_{0})-\mleft(A_{\varepsilon}(\xi_{1})-A_{\varepsilon}(\xi_{0})\mright)\mright\rvert\le C(b,\,p,\,\beta_{0},\,\Gamma,\,\delta)\mu^{p-2-\beta_{0}}\lvert\xi_{1}-\xi_{0}\rvert^{1+\beta_{0}}.
\end{equation}
\end{enumerate}
\end{lemma}
Estimates (\ref{Eq (Section 2) Monotonicity outside})--(\ref{Eq (Section 2) Growth outside}) are easy to deduce from (\ref{Eq (Section 2) Monotonicity estimate for A-p-epsilon})--(\ref{Eq (Section 2) Monotonicity of A-1-epsilon}) and $\varepsilon<\delta<M$.
We would like to show (\ref{Eq (Section 2) Hessian errors}), which plays an important role in our proof of regularity estimates. Although an estimate like (\ref{Eq (Section 2) Hessian errors}) is shown in the author's recent work \cite[Lemmata 2.2 \& 2.6]{T-scalar}, most of our computations herein become rather direct and simple, since the density $E_{\varepsilon}$ is assumed to be spherically symmetric.
\begin{proof}
Let positive constants $\delta,\,\varepsilon$ and vectors $\xi_{0},\,\xi_{1}\in{\mathbb R}^{Nn}$ satisfy respectively (\ref{Eq (Section 2) delta-epsilon}) and (\ref{Eq (Section 2) Variable conditions in error estimates}). We set $\xi_{t}\coloneqq \xi_{0}+t(\xi_{1}-\xi_{0})\in{\mathbb R}^{k}$ for $t\in\lbrack0,\,1\rbrack$. To show (\ref{Eq (Section 2) Hessian errors}), we claim that
\begin{equation}\label{Eq (Section 2) Hessian error p}
\mleft\lvert {\mathcal B}_{p,\,\varepsilon}(\xi_{0})(\xi_{1}-\xi_{0})-\mleft(A_{p,\,\varepsilon}(\xi_{1})-A_{p,\,\varepsilon}(\xi_{0})\mright)\mright\rvert\le C(p,\,\beta_{0})\Gamma\mu^{p-2-\beta_{0}}\lvert\xi_{1}-\xi_{0}\rvert^{1+\beta_{0}},
\end{equation}
and
\begin{equation}\label{Eq (Section 2) Hessian error 1}
\mleft\lvert {\mathcal B}_{1,\,\varepsilon}(\xi_{0})(\xi_{1}-\xi_{0})-\mleft(A_{1,\,\varepsilon}(\xi_{1})-A_{1,\,\varepsilon}(\xi_{0})\mright)\mright\rvert\le Cb\mu^{-2}\lvert\xi_{1}-\xi_{0}\rvert^{2}.
\end{equation}
The desired estimate (\ref{Eq (Section 2) Hessian errors}) immediately follows from (\ref{Eq (Section 2) Hessian error p})--(\ref{Eq (Section 2) Hessian error 1}). Here it should be noted that the inequality $\mu^{-2}\lvert \xi_{1}-\xi_{0}\rvert^{2}\le 4^{1-\beta_{0}}\delta^{1-p}\mu^{p-2-\beta_{0}}\lvert \xi_{1}-\xi_{0}\rvert^{1+\beta_{0}}$ holds by $0<\delta<\mu<\infty$ and $\lvert \xi_{1}-\xi_{0}\rvert\le 4\mu$.

We would like to give the proof of (\ref{Eq (Section 2) Hessian error p}).
We first consider the case $\lvert \xi_{1}-\xi_{0}\rvert\le \mu/2$. Then, by the triangle inequality, it is easy to check that
\[\frac{\mu}{2}\le\lvert\xi_{0}\rvert-t\lvert \xi_{1}-\xi_{0}\rvert\le\lvert \xi_{t}\rvert\le\lvert\xi_{0}\rvert+t\lvert\xi_{1}-\xi_{0}\rvert\le \frac{5\mu}{2}\]
for all $t\in\lbrack0,\,1\rbrack$. In particular, we have
\[\frac{\mu^{2}}{4}\le a_{t}\coloneqq \min\mleft\{\,\varepsilon^{2}+\lvert\xi_{t}\rvert^{2},\,\varepsilon^{2}+\lvert\xi_{0}\rvert^{2}\,\mright\}\le \max\mleft\{\,\varepsilon^{2}+\lvert\xi_{t}\rvert^{2},\,\varepsilon^{2}+\lvert\xi_{0}\rvert^{2}\,\mright\}\eqqcolon b_{t}\le \frac{13\mu^{2}}{2},\]
where we have used $\varepsilon<\delta<\mu$ to get the last inequality.
Also, it is easy to compute
\[\lVert \xi_{t}\otimes\xi_{t}-\xi_{0}\otimes \xi_{0}\rVert\le \mleft(2t\lvert\xi_{0}\rvert+t^{2}\lvert\xi_{1}-\xi_{0}\rvert\mright)\lvert\xi_{1}-\xi_{0}\rvert\le\frac{9\mu}{2}\lvert\xi_{1}-\xi_{0}\rvert,\]
and
\[b_{t}-a_{t}=\mleft\lvert \lvert\xi_{t}\rvert^{2}-\lvert\xi_{0}\rvert^{2} \mright\rvert=\mleft\lvert 2t\langle \xi_{0}\mid \xi_{1}-\xi_{0}\rangle+t^{2}\lvert \xi_{1}-\xi_{0}\rvert^{2}\mright\rvert \le \frac{9\mu}{2}\lvert \xi_{1}-\xi_{0}\rvert\]
for all $t\in\lbrack0,\,1\rbrack$, and $\lVert \xi_{0}\otimes\xi_{0}\rVert\le\lvert\xi_{0}\rvert^{2}\le (2\mu)^{2}$.
Combining them with (\ref{Eq (Section 1) Growth g-p-pprime}), and (\ref{Eq (Section 1) Growth g-p-ppprime}), we are able to check that the operator norm of
\begin{align*}
{\mathcal B}_{p,\,\varepsilon}(\xi_{t})-{\mathcal B}_{p,\,\varepsilon}(\xi_{0})&=2g_{p}^{\prime\prime}(\varepsilon^{2}+\lvert \xi_{t}\rvert^{2})\mleft(\xi_{t}\otimes \xi_{t}-\xi_{0}\otimes \xi_{0} \mright)\\ &\quad+2\mleft[g_{p}^{\prime\prime}(\varepsilon^{2}+\lvert\xi_{t}\rvert ^{2})-g_{p}^{\prime\prime}(\varepsilon^{2}+\lvert\xi_{0}\rvert ^{2}) \mright](\xi_{0}\otimes \xi_{0})\\&\quad +\mleft[g_{p}^{\prime}(\varepsilon^{2}+\lvert\xi_{t}\rvert^{2})- g_{p}^{\prime}(\varepsilon^{2}+\lvert\xi_{0}\rvert^{2})\mright]\mathrm{id}_{Nn}
\end{align*}
is bounded by $C\Gamma\mu^{p-2-\beta_{0}}\lvert\xi_{1}-\xi_{0}\rvert^{\beta_{0}}$ for some constant $C=C(p,\,\beta_{0})\in(0,\,\infty)$.
As a result, we obtain
\begin{align*}
\mleft\lvert {\mathcal B}_{p,\,\varepsilon}(\xi_{0})(\xi_{1}-\xi_{0})-\mleft(A_{p,\,\varepsilon}(\xi_{1})-A_{p,\,\varepsilon}(\xi_{0})\mright)\mright\rvert&=\mleft\lvert \int_{0}^{1}\mleft({\mathcal B}_{p,\,\varepsilon}(\xi_{0})-{\mathcal B}_{p,\,\varepsilon}(\xi_{t})\mright)\cdot(\xi_{1}-\xi_{0})\,{\mathrm{d}}t \mright\rvert\\&\le \lvert\xi_{1}-\xi_{0}\rvert\int_{0}^{1}\mleft\lVert {\mathcal B}_{p,\,\varepsilon}(\xi_{t})-{\mathcal B}_{p,\,\varepsilon}(\xi_{0}) \mright\rVert\,{\mathrm{d}}t\\&\le C(p,\,\beta_{0})\Gamma\mu^{p-2-\beta_{0}}\lvert\xi_{1}-\xi_{0}\rvert^{1+\beta_{0}}.
\end{align*}
In the remaining case $\lvert \xi_{1}-\xi_{0}\rvert>\mu/2$, it is easy to compute
\begin{align*}
&\mleft\lvert {\mathcal B}_{p,\,\varepsilon}(\xi_{0})(\xi_{1}-\xi_{0})-\mleft(A_{p,\,\varepsilon}(\xi_{1})-A_{p,\,\varepsilon}(\xi_{0})\mright)\mright\rvert\\&\le\mleft(\mleft\lVert {\mathcal B}_{p,\,\varepsilon}(\xi_{0})\mright\rVert\lvert \xi_{1}-\xi_{0}\rvert+\mleft\lvert \mleft(A_{p,\,\varepsilon}(\xi_{1})-A_{p,\,\varepsilon}(\xi_{0})\mright)\mright\rvert\mright)\cdot\mleft(\frac{2\lvert \xi_{1}-\xi_{0}\rvert}{\mu}\mright)^{\beta_{0}}\\&\le C(p,\,\beta_{0})\Gamma\mu^{p-2-\beta_{0}}\lvert \xi_{1}-\xi_{0}\rvert^{1+\beta_{0}}
\end{align*}
by (\ref{Eq (Section 2) Estimates on B-p-epsilon}), (\ref{Eq (Section 2) Growth estimate for A-p-epsilon})--(\ref{Eq (Section 2) Growth estimate for A-s-epsilon with 1<s<2}), (\ref{Eq (Section 2) Variable conditions in error estimates}) and $\varepsilon<\delta<\mu$. This completes the proof of (\ref{Eq (Section 2) Hessian error p}). By similar computations, we are able to conclude (\ref{Eq (Section 2) Hessian error 1}) from (\ref{Eq (Section 2) Growth g-1-pprime}), (\ref{Eq (Section 2) Growth g-1-ppprime}), (\ref{Eq (Section 2) Estimates on B-1-epsilon}) and (\ref{Eq (Section 2) Growth estimate for A-s-epsilon with s=1}).
\end{proof}
We conclude this section by mentioning that similar estimates hold for another mapping $G_{p,\,\varepsilon}\colon{\mathbb R}^{Nn}\rightarrow {\mathbb R}^{Nn}$, defined by
\begin{equation}\label{Eq (Section 2) G-p-epsilon}
G_{p,\,\varepsilon}(\xi)\coloneqq h_{p}^{\prime}(\varepsilon^{2}+\lvert \xi\rvert^{2})\xi\quad \textrm{for }\xi\in{\mathbb R}^{Nn}
\end{equation}
with
\begin{equation}\label{Eq (Section 2) hp}
h_{p}(\sigma)\coloneqq \frac{2\sigma^{(p+1)/2}}{p+1}\quad \textrm{for }\sigma\in\lbrack0,\,\infty).
\end{equation}
It should be mentioned that this $h_{p}$ is the same as $g_{p+1}$ with $g_{p}$ given by (\ref{Eq (Section 1) gp example}).
Hence, similarly to (\ref{Eq (Section 2) Monotonicity estimate for A-p-epsilon}), we have
\[\langle G_{p,\,\varepsilon}(\xi_{1})-G_{p,\,\varepsilon}(\xi_{0})\mid \xi_{1}-\xi_{0}\rangle\ge c\lvert \xi_{1}-\xi_{0}\rvert^{p+1}\quad \textrm{for all } \xi_{0},\,\xi_{1}\in {\mathbb R}^{Nn}\]
with $c=c(p)\in(0,\,\infty)$. Moreover, since the mapping $G_{p,\,\varepsilon}$ is bijective and enjoys $G_{p,\,\varepsilon}(0)=0$ by the definition, the inverse mapping $G_{p,\,\varepsilon}^{-1}$ satisfies
\begin{equation}\label{Eq (Section 2) Estimate on inverse mapping}
\mleft\lvert G_{p,\,\varepsilon}^{-1}(\xi)\mright\rvert\le C(p)\lvert \xi\rvert^{1/p}\quad \textrm{for all }\xi\in{\mathbb R}^{Nn}
\end{equation}
with $C=c^{-1}\in(0,\,\infty)$. Also, similarly to (\ref{Eq (Section 2) Growth estimate for A-s-epsilon with s>2}), it is possible to get
\begin{equation}\label{Eq (Section 2) G-p-epsilon local ellipticity}
\mleft\lvert G_{p,\,\varepsilon}(\xi_{1})-G_{p,\,\varepsilon}(\xi_{0})\mright\rvert\ge C(p)\max\mleft\{\,\lvert\xi_{1}\rvert^{p-1},\,\lvert \xi_{2}\rvert^{p-1} \,\mright\}\lvert \xi_{1}-\xi_{0}\rvert
\end{equation}
for all $\xi_{0},\,\xi_{1}\in{\mathbb R}^{Nn}$. The estimates (\ref{Eq (Section 2) Estimate on inverse mapping})--(\ref{Eq (Section 2) G-p-epsilon local ellipticity}) will be used in Section \ref{Subsect: Energy estimates}.

\subsection{Justifications of convergence of solutions}\label{Subsect: Convergence}
The aim of Section \ref{Subsect: Convergence} is to give an approximating system for (\ref{Eq (Section 1) main system}), and justify convergence of weak solutions. We only deal with the special case where the energy density $E$ is spherically symmetric.
In the scalar case, more general approximation problems are discussed in \cite[\S 2.4--2.5]{T-scalar}, including variational inequality problems and generalization of the total variation energy.

Only in this section, we assume that the exponent $q$ satisfies
\begin{equation}\label{Eq (Section 2) exponent q}
\mleft\{\begin{array}{rc} \displaystyle\frac{np}{np-n+p}<q\le \infty& (1<p<n),\\ 1<q\le\infty & (p=n),\\ 1\le q\le \infty & (n<p<\infty),\end{array} \mright.
\end{equation}
so that we can use the compact embedding \(W^{1,\,p}(\Omega;\,{\mathbb R}^{N})\hookrightarrow L^{q^{\prime}}(\Omega;\,{\mathbb R}^{N})\) (see e.g., \cite[Chapters 4 \& 6]{MR2424078}).
Under this setting, we give the definitions of a weak solution to the system (\ref{Eq (Section 1) main system}), and of the Dirichlet boundary value problem
\begin{equation}\label{Eq (Section 2) Dirichlet boundary problem}
\mleft\{\begin{array}{ccccc}
{\mathcal L}u& = & f & \textrm{in} &\Omega,\\ u & = & u_{\star} & \textrm{on} & \partial\Omega.
\end{array}  \mright.
\end{equation}
\begin{definition}
Let the functions $u_{\star}\in W^{1,\,p}(\Omega;\,{\mathbb R}^{N})$, $f\in L^{q}(\Omega;\,{\mathbb R}^{N})$ be given with $p\in(1,\,\infty)$ and $q\in\lbrack 1,\,\infty\rbrack$ satisfying (\ref{Eq (Section 2) exponent q}). A function $u\in W^{1,\,p}(\Omega;\,{\mathbb R}^{N})$ is said to be a \textit{weak} solution to (\ref{Eq (Section 1) main system}) in $\Omega$ when there exists $Z\in L^{\infty}(\Omega;\,{\mathbb R}^{Nn})$ such that there hold
\begin{equation}\label{Eq (Section 2) Z is subgradient}
Z(x)\in\partial \lvert\,\cdot\,\rvert(Du(x))\quad \textrm{for a.e. }x\in\Omega,
\end{equation}
and 
\begin{equation}\label{Eq (Section 2) Weak formulation of very singular problems}
\int_{\Omega}\langle Z\mid D\phi\rangle\,{\mathrm{d}}x+\int_{\Omega}\mleft\langle A_{p}(Du)\mathrel{}\middle|\mathrel{}D\phi\mright\rangle\,{\mathrm{d}}x=\int_{\Omega}\langle f\mid \phi \rangle\,{\mathrm{d}}x\quad \textrm{for all }\phi\in W_{0}^{1,\,p}(\Omega;\,{\mathbb R}^{N}).
\end{equation}
Here $A_{p}\in C({\mathbb R}^{n};\,{\mathbb R}^{Nn})$ and $\partial\lvert\,\cdot\,\rvert(\xi)\subset{\mathbb R}^{Nn}\,(\xi\in{\mathbb R}^{Nn})$ are given by (\ref{Eq (Section 2) Def of A-p})--(\ref{Eq: sub}). When a function $u\in u_{\star}+W_{0}^{1,\,p}(\Omega;\,{\mathbb R}^{N})$ is a weak solution to (\ref{Eq (Section 1) main system}) in $\Omega$, $u$ is called a \textit{weak} solution of the Dirichlet problem (\ref{Eq (Section 2) Dirichlet boundary problem}).
\end{definition}

It should be recalled that the problem (\ref{Eq (Section 1) main system}) is derived from a minimizing problem of the energy functional
\begin{equation}\label{Eq (Section 2) energy functional: original}
{\mathcal F}_{0}(v)\coloneqq \int_{\Omega}\mleft(E(Dv)-\langle f\mid v\rangle\mright)\,{\mathrm{d}}x\quad \textrm{for }v\in W^{1,\,p}(\Omega;\,{\mathbb R}^{N})
\end{equation}
under a suitable boundary condition.
We approximate this functional by
\begin{equation}\label{Eq (Section 2) energy functional: relaxed}
{\mathcal F}_{\varepsilon}(v)\coloneqq \int_{\Omega}\mleft(E_{\varepsilon}(Dv)-\langle f_{\varepsilon}\mid v\rangle\mright)\,{\mathrm{d}}x\quad \textrm{for }v\in W^{1,\,p}(\Omega;\,{\mathbb R}^{N}),\,\varepsilon\in(0,\,1).
\end{equation}
Here the net $\{f_{\varepsilon}\}_{0<\varepsilon<1}\subset L^{q}(\Omega;\,{\mathbb R}^{N})$ satisfies
\begin{equation}\label{Eq (Section 2) Weak convergence of f}
f_{\varepsilon}\rightharpoonup f\quad \textrm{in }\sigma\mleft(L^{q}(\Omega;\,{\mathbb R}^{N}),\,L^{q^{\prime}}(\Omega;\,{\mathbb R}^{N})\mright).
\end{equation}
In other words, we only let $f_{\varepsilon}\in L^{q}(\Omega;\,{\mathbb R}^{N})$ weakly converge to $f$ when $q$ is finite, and otherwise weak$^{\ast}$ convergence is assumed. In particular, we may let $\lVert f_{\varepsilon}\rVert_{L^{q}(\Omega)}$ be uniformly bounded with respect to $\varepsilon$. In Proposition \ref{Prop: Convergence on weak system}, we justify convergence of minimizers of relaxed energy functionals.
\begin{proposition}[A convergence result on approximation problems]\label{Prop: Convergence on weak system}
Let $p\in(1,\,\infty),\,q\in\lbrack 1,\,\infty\rbrack$ satisfy (\ref{Eq (Section 2) exponent q}), and consider functionals ${\mathcal F}_{0}$, ${\mathcal F}_{\varepsilon}\,(0<\varepsilon<1)$ given by (\ref{Eq (Section 2) energy functional: original})--(\ref{Eq (Section 2) energy functional: relaxed}), where $\{f_{\varepsilon}\}_{0<\varepsilon<1}\subset L^{q}(\Omega;\,{\mathbb R}^{N})$ satisfies (\ref{Eq (Section 2) Weak convergence of f}). 
For a given function $u_{\star}\in W^{1,\,p}(\Omega;\,{\mathbb R}^{N})$, we define
\begin{equation}\label{Eq (Section 2) minimizer of approximated functionals}
u_{\varepsilon}\coloneqq \argmin\mleft\{v\in u_{\star}+W_{0}^{1,\,p}(\Omega;\,{\mathbb R}^{N})\mathrel{}\middle|\mathrel{} {\mathcal F}_{\varepsilon}(v)\mright\}\quad \textrm{for each }\varepsilon\in(0,\,1).
\end{equation}
Then, there exists a unique function $u\in u_{\star}+W_{0}^{1,\,p}(\Omega;\,{\mathbb R}^{N})$ such that $u_{\varepsilon}\rightarrow u$ in $W^{1,\,p}(\Omega;\,{\mathbb R}^{N})$ up to a subsequence. Moreover, the limit function $u$ is a weak solution of the Dirichlet boundary value problem (\ref{Eq (Section 2) Dirichlet boundary problem}), and there holds
\begin{equation}\label{Eq (Section 2) minimizer of original functionals}
u=\argmin\mleft\{v\in u_{\star}+W_{0}^{1,\,p}(\Omega;\,{\mathbb R}^{N})\mathrel{}\middle|\mathrel{} {\mathcal F}_{0}(v)\mright\}.
\end{equation}
\end{proposition}
Before showing Proposition \ref{Prop: Convergence on weak system}, we note that (\ref{Eq (Section 2) minimizer of approximated functionals}) is characterized by 
\begin{equation}\label{Eq (Section 2) Regularized weak formulation}
\int_{\Omega}\mleft\langle A_{\varepsilon}(Du_{\varepsilon})\mid D\phi\mright\rangle\,{\mathrm{d}}x=\int_{\Omega}\langle f_{\varepsilon}\mid \phi\rangle\,{\mathrm{d}}x\quad \textrm{for all }\phi\in W_{0}^{1,\,p}(\Omega;\,{\mathbb R}^{N}).
\end{equation}
In other words, for a function $u_{\varepsilon}\in u_{\star}+W_{0}^{1,\,p}(\Omega;\,{\mathbb R}^{N})$, $u_{\varepsilon}$ verifies (\ref{Eq (Section 2) minimizer of approximated functionals}) if and only if $u_{\varepsilon}$ satisfies (\ref{Eq (Section 2) Regularized weak formulation}). In the limiting case $\varepsilon=0$, however, it is a non-trivial question whether a function $u_{0}$ that verifies (\ref{Eq (Section 2) minimizer of original functionals}) is a weak solution of (\ref{Eq (Section 2) Dirichlet boundary problem}), since the total variation energy is neither G\^{a}teaux differentiable nor Fr\'{e}chet differentiable. This is substantially due to non-smoothness of the absolute value function at the origin. 
To overcome this problem, we appeal to an approximation method, and construct a vector field $Z$ that satisfies (\ref{Eq (Section 2) Z is subgradient})--(\ref{Eq (Section 2) Weak formulation of very singular problems}). 
\begin{proof}
We first mention that the right hand sides of (\ref{Eq (Section 2) minimizer of approximated functionals})--(\ref{Eq (Section 2) minimizer of original functionals}) are well-defined. 
In fact, to construct minimizers by direct methods (\cite[Chapter 8]{MR1625845}, \cite[Chapter 4]{MR1962933}), it suffices to prove boundedness and coerciveness of the functional ${\mathcal F}_{\varepsilon}$ over the space $u_{\star}+W_{0}^{1,\,p}(\Omega;\,{\mathbb R}^{N})$.
To check boundedness, we let $g_{p}(0)=0$ without loss of generality. Then the growth estimate $g_{p}(\sigma)\le C(p,\,\Gamma)\sigma^{p/2}$ easily follows from (\ref{Eq (Section 1) Growth g-p-prime}).
In particular, by the continuous embedding $L^{q}(\Omega;\,{\mathbb R}^{N})\hookrightarrow W^{1,\,p}(\Omega;\,{\mathbb R}^{N})$, we have
\begin{equation}\label{Eq (Section 2) Growth of F-epsilon}
{\mathcal F}_{\varepsilon}(v)\le K_{0}\mleft( 1+\lVert Dv\rVert_{L^{p}(\Omega)}^{p}+\lVert f_{\varepsilon}\rVert_{L^{q}(\Omega)}\lVert v\rVert_{W^{1,\,p}(\Omega)} \mright)
\end{equation}
for all $\varepsilon\in\lbrack 0,\,1)$ and $v\in W^{1,\,p}(\Omega;\,{\mathbb R}^{N})$.
For a coercive estimate, we have
\begin{equation}\label{Eq (Section 2) Coercive of F-epsilon}
{\mathcal F}_{\varepsilon}(v)\ge K_{1}\lVert Dv\rVert_{L^{p}(\Omega)}-K_{2}\mleft( 1+\lVert f_{\varepsilon}\rVert_{L^{q}(\Omega)}\lVert u_{\star}\rVert_{W^{1,\,p}(\Omega)}+\lVert f_{\varepsilon}\rVert_{L^{q}(\Omega)}^{p^{\prime}}\mright)
\end{equation}
for all $\varepsilon\in\lbrack 0,\,1)$ and $v\in u_{\star}+W_{0}^{1,\,p}(\Omega;\,{\mathbb R}^{N})$.
Here the constants $K_{1}\in\lbrack 0,\,1)$ and $K_{2}\in(1,\,\infty)$ depend at most on $n$, $p$, $q$, $\gamma$ and $\Omega$, but are independent of $\varepsilon\in\lbrack 0,\,1)$. It is possible to deduce (\ref{Eq (Section 2) Coercive of F-epsilon}) by applying the continuous embedding $W^{1,\,p}(\Omega;\,{\mathbb R}^{N})\hookrightarrow L^{q^{\prime}}(\Omega;\,{\mathbb R}^{N})$ and the Poincar\'{e} inequality
\begin{equation}\label{Eq (Section 2) Poincare ineq}
\lVert v-u_{\star}\rVert_{W^{1,\,p}(\Omega)}\le C(n,\,p,\,\Omega)\lVert Dv-Du_{\star}\rVert_{L^{p}(\Omega)}\quad \textrm{for all }v\in u_{\star}+W_{0}^{1,\,p}(\Omega;\,{\mathbb R}^{N}).
\end{equation}
The detailed computations to show (\ref{Eq (Section 2) Coercive of F-epsilon}) are substantially similar to \cite[Propostion 2.11]{T-scalar} (see also \cite[\S 3]{MR4201656}).
Uniqueness of minimizers are guaranteed by strict convexity of $E_{\varepsilon}$ and (\ref{Eq (Section 2) Poincare ineq}). Therefore, we are able to define $u$ and $u_{\varepsilon}$ satisfying (\ref{Eq (Section 2) minimizer of approximated functionals}) and (\ref{Eq (Section 2) minimizer of original functionals}) respectively.

We also mention that $\{Du_{\varepsilon}\}_{0<\varepsilon<1}\subset L^{p}(\Omega;\,{\mathbb R}^{Nn})$ is bounded, and so $\{u_{\varepsilon}\}_{0<\varepsilon<1}\subset W^{1,\,p}(\Omega;\,{\mathbb R}^{N})$ is by (\ref{Eq (Section 2) Poincare ineq}). This is easy to deduce by applying (\ref{Eq (Section 2) Growth of F-epsilon})--(\ref{Eq (Section 2) Coercive of F-epsilon}) to the inequality ${\mathcal F}_{\varepsilon}(v_{\varepsilon})\le {\mathcal F}_{\varepsilon}(u_{\star})$ following from (\ref{Eq (Section 2) minimizer of approximated functionals}).
Hence, by the weak compactness theorem, we may choose a sequence $\{\varepsilon_{j}\}_{j=1}^{\infty}\subset(0,\,1)$ such that $\varepsilon_{j}\to 0$ and
\begin{equation}\label{Eq (Section 2) Weak limit u0}
u_{\varepsilon_{j}}\rightharpoonup u_{0}\quad \textrm{in }W^{1,\,p}(\Omega;\,{\mathbb R}^{N})
\end{equation}
for some function $u_{0}\in u_{\star}+W_{0}^{1,\,p}(\Omega;\,{\mathbb R}^{N})$.
Moreover, the condition (\ref{Eq (Section 2) exponent q}) enables us to apply the compact embedding, so that we may let
\begin{equation}\label{Eq (Section 2) Strong convergence from compact embedding}
u_{\varepsilon_{j}}\rightarrow u_{0}\quad \textrm{in }L^{q^{\prime}}(\Omega;\,{\mathbb R}^{N})
\end{equation}
by taking a subsequence if necessary. We claim that
\begin{equation}\label{Eq (Section 2) Monotonicity term convergence}
I(\varepsilon_{j})\coloneqq \int_{\Omega}\langle A_{\varepsilon_{j}}(Du_{\varepsilon_{j}})-A_{\varepsilon_{j}}(Du_{0})\mid Du_{\varepsilon_{j}}-Du_{0}\rangle\,{\mathrm{d}}x\to 0\quad\textrm{as }j\to\infty.
\end{equation}
Before showing (\ref{Eq (Section 2) Monotonicity term convergence}), we mention that this result yields
\begin{equation}\label{Eq (Section 2) strong convergence}
Du_{\varepsilon_{j}}\to Du_{0}\quad \textrm{in }L^{p}(\Omega;\,{\mathbb R}^{Nn}).
\end{equation}
In fact, when $1<p<2$, we apply (\ref{Eq (Section 2) Monotonicity estimate for A-p-epsilon}), (\ref{Eq (Section 2) Monotonicity of A-1-epsilon}), and H\"{o}lder's inequality to obtain
\begin{align*}
\lVert Du_{\varepsilon_{j}}-Du_{0}\rVert_{L^{p}(\Omega)}^{p}&\le \mleft(\int_{\Omega}\mleft(\varepsilon_{j}^{2}+\lvert Du_{\varepsilon_{j}}\rvert^{2}+\lvert Du_{0}\rvert^{2}\mright)^{p/2-1}\lvert Du_{\varepsilon_{j}}-Du_{0}\rvert^{2}\,{\mathrm{d}}x\mright)^{p/2} \\& \quad \cdot\underbrace{\sup_{j}\mleft(\int_{\Omega}\mleft(\varepsilon_{j}^{2}+\lvert Du_{\varepsilon_{j}}\rvert^{2}+\lvert Du_{0}\rvert^{2}\mright)^{p/2}\,{\mathrm{d}}x \mright)^{1-p/2}}_{\eqqcolon C}\\&\le C\lambda^{-p/2}\cdot I(\varepsilon_{j})^{p/2}\to 0\quad \textrm{as }j\to\infty.
\end{align*}
Here it is noted that $C$ is finite by (\ref{Eq (Section 2) Weak limit u0}).
In the remaining case $2\le p<\infty$, we simply use (\ref{Eq (Section 2) Monotonicity estimate for A-p-epsilon}) and (\ref{Eq (Section 2) Monotonicity of A-1-epsilon}) to get
\[\lVert Du_{\varepsilon_{j}}-Du_{0}\rVert_{L^{p}(\Omega)}^{p}\le \frac{I(\varepsilon_{j})}{\lambda C(p)}\to 0\quad \textrm{as }j\to\infty.\]

For the proof of (\ref{Eq (Section 2) Monotonicity term convergence}), we decompose $I=I_{1}-I_{2}$ with
\[I_{1}(\varepsilon_{j})\coloneqq\int_{\Omega}\langle A_{\varepsilon_{j}}(Du_{\varepsilon_{j}})\mid Du_{\varepsilon_{j}}-Du_{0}\rangle\,{\mathrm{d}}x,\quad I_{2}(\varepsilon_{j})\coloneqq\int_{\Omega}\langle A_{\varepsilon_{j}}(Du)\mid Du_{\varepsilon_{j}}-Du_{0}\rangle\,{\mathrm{d}}x.\]
For $I_{1}$, we test $\phi\coloneqq u_{\varepsilon_{j}}-u_{0}\in W_{0}^{1,\,p}(\Omega;\,{\mathbb R}^{N})$ into (\ref{Eq (Section 2) Regularized weak formulation}). Then, we obtain
\[I_{1}(\varepsilon_{j})=\int_{\Omega}\langle f_{\varepsilon_{j}}\mid u_{\varepsilon_{j}}-u_{0}\rangle\,{\mathrm d}x\to 0\quad \textrm{as }j\to\infty\]
by (\ref{Eq (Section 2) Weak convergence of f}) and (\ref{Eq (Section 2) Strong convergence from compact embedding}).
For $I_{2}$, we note that for every $\xi\in{\mathbb R}^{Nn}$, $A_{\varepsilon}(\xi)$ converges to
\[A_{0}(\xi)\coloneqq \mleft\{\begin{array}{cc}
g^{\prime}(\lvert \xi\rvert^{2})\xi & (\xi\neq 0),\\
0 & (\xi=0),
\end{array}\quad \textrm{as }\varepsilon\to 0, \mright.\]
since $A_{\varepsilon}(\xi)=b\mleft(\varepsilon^{2}+\lvert \xi\rvert^{2}\mright)^{-1/2}\xi+g_{p}^{\prime}(\varepsilon^{2}+\lvert \xi\rvert^{2})\xi$ satisfies $A_{\varepsilon}(0)=0$. Also, for every $\varepsilon\in(0,\,1)$, it is easy to check that
\[\lvert A_{\varepsilon}(Du_{0})\rvert\le b+\Gamma\mleft(\varepsilon^{2}+\lvert Du_{0}\rvert^{2} \mright)^{\frac{p-1}{2}}\le C(b,\,p,\,\Lambda)(1+\lvert Du_{0}\rvert^{p-1})\quad \textrm{a.e. in }\Omega.\]
Hence, by applying Lebesgue's dominated convergence theorem, we have $A_{\varepsilon}(Du_{0})\rightarrow A_{0}(Du_{0})$ in $L^{p^{\prime}}(\Omega;\,{\mathbb R}^{Nn})$. It should be recalled that the weak convergence $Du_{\varepsilon}\rightharpoonup Du_{0}$ in $L^{p}(\Omega;\,{\mathbb R}^{Nn})$ is already known by (\ref{Eq (Section 2) Weak limit u0}). These convergence results imply $I_{2}(\varepsilon_{j})\to 0$ as $j\to\infty$. Thus, the claim (\ref{Eq (Section 2) Monotonicity term convergence}) is verified. 
 
We would like to prove that the limit $u_{0}$ is a weak solution of the problem (\ref{Eq (Section 2) Dirichlet boundary problem}).
First, by (\ref{Eq (Section 2) strong convergence}) and \cite[Theorem 4.9]{MR2759829}, we may let 
\begin{equation}\label{Eq (Section 2) a.e. convergence of Du}
Du_{\varepsilon_{j}}\rightarrow Du_{0}\quad \textrm{a.e. in }\Omega,
\end{equation}
and
\begin{equation}\label{Eq (Section 2) existence of dominant function}
\lvert Du_{\varepsilon_{j}}\rvert\le v\quad \textrm{a.e. in }\Omega,
\end{equation}
by taking a subsequence if necessary. Here a non-negative function $v\in L^{p}(\Omega)$ is independent of the subscript $j\in{\mathbb N}$. Then,
\begin{equation}\label{Eq (Section 2) A-p-epsilon strong convergence}
A_{p,\,\varepsilon_{j}}(Du_{\varepsilon_{j}})\rightarrow A_{p}(Du_{0})\quad \textrm{in }L^{p^{\prime}}(\Omega;\,{\mathbb R}^{Nn})
\end{equation}
follows from (\ref{Eq (Section 2) a.e. convergence of Du})--(\ref{Eq (Section 2) existence of dominant function}).
To verify this, we should note that by (\ref{Eq (Section 2) Growth estimate for A-p-epsilon}), the mapping $A_{p,\,\varepsilon}$ locally uniformly converges to $A_{p}$ in ${\mathbb R}^{Nn}$.
Combining with (\ref{Eq (Section 2) a.e. convergence of Du}), we can check $A_{p,\,\varepsilon_{j}}(Du_{\varepsilon_{j}})\rightarrow A_{p}(Du_{0})$ a.e. in $\Omega$. 
Also, by (\ref{Eq (Section 1) Growth g-p-prime}) and (\ref{Eq (Section 2) existence of dominant function}), we get 
\[\lvert A_{p,\,\varepsilon_{j}}(Du_{\varepsilon_{j}})\rvert\le \Gamma\mleft(\varepsilon_{j}^{2}+\lvert Du_{\varepsilon_{j}}\rvert^{2}\mright)^{p/2-1}\lvert Du_{\varepsilon_{j}}\rvert\le \Gamma\mleft(1+v^{2}\mright)^{\frac{p-1}{2}}\in L^{p^{\prime}}(\Omega)\]
a.e. in $\Omega$.
Thus, (\ref{Eq (Section 2) A-p-epsilon strong convergence}) can be deduced by Lebesgue's dominated convergence theorem.
Secondly, since the mapping $Z_{j}\coloneqq A_{1,\,\varepsilon_{j}}(Du_{\varepsilon_{j}})$ satisfies $\lVert Z_{j}\rVert_{L^{\infty}(\Omega)}\le 1$, up to a subsequence, we may let
\begin{equation}\label{Eq (Section 2) weak-star limit Z}
Z_{j}\overset{\ast}{\rightharpoonup} Z\quad \textrm{in }L^{\infty}(\Omega;\,{\mathbb R}^{Nn})
\end{equation}
for some $Z\in L^{\infty}(\Omega;\,{\mathbb R}^{Nn})$ \cite[Corollary 3.30]{MR2759829}. This limit clearly satisfies $\lVert Z\rVert_{L^{\infty}(\Omega)}\le 1$. Therefore, to check (\ref{Eq (Section 2) Z is subgradient}), it suffices to prove 
\[Z=\frac{Du_{0}}{\lvert Du_{0}\rvert}\quad \textrm{a.e. in }D\coloneqq \{x\in\Omega\mid Du_{0}(x)\neq 0\}.\]
This claim is easy to deduce by (\ref{Eq (Section 2) a.e. convergence of Du}). In fact, (\ref{Eq (Section 2) a.e. convergence of Du}) yields $Z_{\varepsilon_{j}}\to Du_{0}/\lvert Du_{0}\rvert$ a.e. in $D$, and hence we are able to conclude $Z_{\varepsilon_{j}}\overset{\ast}{\rightharpoonup}Du_{0}/\lvert Du_{0}\rvert$ in $L^{\infty}(D;\,{\mathbb R}^{Nn})$ by Lebesgue's dominated convergence theorem. Finally, using (\ref{Eq (Section 2) Weak convergence of f}) and (\ref{Eq (Section 2) A-p-epsilon strong convergence})--(\ref{Eq (Section 2) weak-star limit Z}), we are able to deduce (\ref{Eq (Section 2) Weak formulation of very singular problems}) by letting $\varepsilon=\varepsilon_{j}$ and $j\to\infty$ in the weak formulation (\ref{Eq (Section 2) Regularized weak formulation}).

We mention that $u_{0}$, a weak solution of (\ref{Eq (Section 2) Dirichlet boundary problem}), coincides with $u$ satisfying (\ref{Eq (Section 2) minimizer of original functionals}).
In fact, for arbitrary $\phi\in W_{0}^{1,\,p}(\Omega;\,{\mathbb R}^{N})$, we have
\[\lvert D(u+\phi)\rvert\ge \lvert Du\rvert+\langle Z\mid D\phi\rangle\quad \textrm{a.e. in }\Omega,\]
since $Z$ satisfies (\ref{Eq (Section 2) Z is subgradient}). Similarly, we can easily get
\[g_{p}(\lvert D(u+\phi)\rvert^{2})\ge g_{p}(\lvert Du\rvert^{2})+\mleft\langle A_{p}(Du)\mathrel{}\middle|\mathrel{}D\phi\mright\rangle\quad \textrm{a.e. in }\Omega.\]
By testing $\phi\in W_{0}^{1,\,p}(\Omega;\,{\mathbb R}^{N})$ into (\ref{Eq (Section 2) Weak formulation of very singular problems}), we can easily notice ${\mathcal F}_{0}(u)\le {\mathcal F}_{0}(u+\phi)$ for all $\phi\in W_{0}^{1,\,p}(\Omega;\,{\mathbb R}^{N})$.
In other words, the limit function $u_{0}\in u_{\star}+W_{0}^{1,\,p}(\Omega;\,{\mathbb R}^{N})$ satisfies (\ref{Eq (Section 2) minimizer of original functionals}). We recall that a function verifying (\ref{Eq (Section 2) minimizer of original functionals}) uniquely exists, and thus $u=u_{0}$. This completes the proof.
\end{proof}

\subsection{H\"{o}lder continuity estimates}\label{Subsect: A priori Hoelder}
In Section \ref{Subsect: A priori Hoelder}, we would like to prove our main theorem (Theorem \ref{Theorem: C1-regularity}) by an approximation argument.
In Proposition \ref{Prop: Convergence on weak system}, we have already justified that a weak solution to (\ref{Eq (Section 1) main system}) can be approximated by a weak solution to
\begin{equation}\label{Eq (Section 2) Approximated system}
-\divx\mleft(g_{\varepsilon}^{\prime}(\lvert Du_{\varepsilon}\rvert^{2})Du_{\varepsilon}\mright)=f_{\varepsilon},
\end{equation}
under a suitable Dirichlet boundary value condition. We should note that the function $u_{\varepsilon}$ defined by (\ref{Eq (Section 2) minimizer of approximated functionals}) solves (\ref{Eq (Section 2) Approximated system}) in the distributional sense, under a boundary condition $u_{\varepsilon}=u_{\star}$ on $\partial\Omega$.
Since we have already justified convergence for approximated solutions $u_{\varepsilon}$ in Proposition \ref{Prop: Convergence on weak system}, it suffices to obtain a priori regularity estimates on weak solutions to a regularized system (\ref{Eq (Section 2) Approximated system}). 
The key estimates are given by Theorem \ref{Theorem: A priori Hoelder estimate}, where the continuity estimates are independent of the approximation parameter $\varepsilon$, so that the Arzel\`{a}--Ascoli theorem can be applied.
\begin{theorem}[A priori H\"{o}lder estimates on truncated Jacobian matrices]\label{Theorem: A priori Hoelder estimate}
Let positive numbers $\delta,\,\varepsilon$ satisfy (\ref{Eq (Section 2) delta-epsilon}), and $u_{\varepsilon}$ be a weak solution to a regularized system (\ref{Eq (Section 2) Approximated system}) in \(\Omega\) with
\begin{equation}\label{Eq (Section 2) control of Du}
\lVert Du_{\varepsilon}\rVert_{L^{p}(\Omega)}\le L,
\end{equation}
and
\begin{equation}\label{Eq (Section 2) control of f}
\lVert f_{\varepsilon}\rVert_{L^{q}(\Omega)}\le F
\end{equation}
for some constants $F,\,L\in(0,\,\infty)$. Then, for each fixed $x_{\ast}\in\Omega$, there exist a sufficiently small open ball $B_{\rho_{0}}(x_{\ast})\Subset\Omega$ and an exponent $\alpha\in(0,\,1)$, depending at most on $b,\,n,\,N,\,p,\,q,\,\gamma,\,\Gamma,\, F,\,L,\,d_{\ast}=\dist(x_{\ast},\,\partial\Omega)$, and $\delta$, such that ${\mathcal G}_{2\delta,\,\varepsilon}(Du_{\varepsilon})$ is in $C^{\alpha}(B_{\rho_{0}/2}(x_{0});\,{\mathbb R}^{Nn})$. Moreover, there exists a constant $\mu_{0}\in(0,\,\infty)$, depending at most on $b$, $n$, $N$, $p$, $q$, $\gamma$, $\Gamma$, $F$, $L$, and $d_{\ast}$, such that
\begin{equation}\label{Eq (Section 2) Bound of G-2delta-epsilon}
\sup_{B_{\rho_{0}}(x_{\ast})}\mleft\lvert{\mathcal G}_{2\delta,\,\varepsilon}(Du_{\varepsilon})\mright\rvert\le \mu_{0},
\end{equation}
and
\begin{equation}\label{Eq (Section 2) Holder continuity of G-2delta-epsilon}
\mleft\lvert{\mathcal G}_{2\delta,\,\varepsilon}(Du_{\varepsilon}(x_{1}))-{\mathcal G}_{2\delta,\,\varepsilon}(Du_{\varepsilon}(x_{2}))\mright\rvert\le \frac{2^{n/2+2\alpha+2}\mu_{0}}{\rho_{0}^{\alpha}}\lvert x_{1}-x_{2}\rvert^{\alpha}
\end{equation}
for all $x_{1},\,x_{2}\in B_{\rho_{0}/2}(x_{\ast})$.
\end{theorem}
To prove Theorem \ref{Theorem: A priori Hoelder estimate}, we use three basic propositions, whose proofs are given later in Sections \ref{Section: Weak formulations}--\ref{Section: Appendix}. 
Here it is noted that when stating Propositions \ref{Prop: Lipschitz bounds}--\ref{Prop: Schauder estimate}, we have to use another modulus \(V_{\varepsilon}\coloneqq \sqrt{\varepsilon^{2}+\lvert Du_{\varepsilon}\rvert^{2}}\), instead of $\lvert Du_{\varepsilon}\rvert$, since we have relaxed the principal divergence operator ${\mathcal L}$ itself. Along with this, we have already introduced another truncation mapping ${\mathcal G}_{2\delta,\,\varepsilon}$, instead of ${\mathcal G}_{2\delta}$.

The first proposition states local a priori Lipschitz bounds for regularized solutions. 
\begin{proposition}[Local Lipschitz bounds]\label{Prop: Lipschitz bounds}
Let $\varepsilon\in(0,\,1)$ and $u_{\varepsilon}$ be a weak solution to a regularized system (\ref{Eq (Section 2) Approximated system}) in \(\Omega\). Fix an open ball $B_{\rho}(x_{0})\Subset \Omega$ with $\rho\in(0,\,1\rbrack$. Then, there exists a constant $C\in(0,\,\infty)$ depending at most on $b$, $n$, $p$, $q$, $\gamma$, and $\Gamma$, such that
\[\esssup_{B_{\theta\rho}(x_{0})}\,V_{\varepsilon}\le C\mleft[1+\lVert f_{\varepsilon}\rVert_{L^{q}(B_{\rho}(x_{0}))}^{1/(p-1)}+\frac{\lVert V_{\varepsilon}\rVert_{L^{p}(B_{\rho}(x_{0}))}}{[(1-\theta)\rho]^{d}}\mright]\]
for all $\theta\in(0,\,1)$. Here the exponent $d\in\lbrack n/p,\,\infty)$ depends at most on $n$, $p$, and $q$.
\end{proposition}
These estimates can be deduced by carefully choosing test functions whose supports are separated from facets of approximated solutions.
Also, it should be emphasized that local Lipschitz bounds in Proposition \ref{Prop: Lipschitz bounds} are expectable, since the density $E_{\varepsilon}$ has a $p$-Laplace-type structure when it is sufficiently far from the origin. Lipschitz estimates from a viewpoint of asymptotic bahaviour of density functions are found in the existing literature \cite{MR4078712} (see also \cite{MR852362} as a classical work). For the reader's convenience, we would like to provide the proof of Proposition \ref{Prop: Lipschitz bounds} in the appendix (Section \ref{Section: Appendix}).

Hereinafter we assume local uniform boundedness of $V_{\varepsilon}$, which is guaranteed by Proposition \ref{Prop: Lipschitz bounds}. In particular, a scalar function $\lvert {\mathcal G}_{\delta,\,\varepsilon}(Du_{\varepsilon})\rvert$ is uniformly bounded in each fixed subdomain of $\Omega$. For an open ball $B_{\rho}(x_{0})\Subset \Omega$ and positive numbers $\mu\in(0,\,\infty)$, $\nu\in(0,\,1)$, we introduce a superlevel set
\[S_{\rho,\,\mu,\,\nu}(x_{0})\coloneqq \{x\in B_{\rho}(x_{0})\mid V_{\varepsilon}-\delta>(1-\nu)\mu\}.\]
The second and third propositions (Propositions \ref{Prop: De Giorgi's truncation}--\ref{Prop: Schauder estimate}) are useful for estimating oscillation of ${\mathcal G}_{2\delta,\,\varepsilon}(Du_{\varepsilon})$. Our analysis herein depends on whether $V_{\varepsilon}$ vanishes near a point $x_{0}$ or not, which is judged by measuring the size of the superlevel set $S_{\rho,\,\mu,\,\nu}(x_{0})$. Before stating Propositions \ref{Prop: De Giorgi's truncation}--\ref{Prop: Schauder estimate}, we introduce a constant $\beta\in(0,\,1)$ by
\[\beta\coloneqq\mleft\{\begin{array}{cc}
1-\displaystyle\frac{n}{q} & (n<q<\infty),\\ {\hat\beta}_{0} & (q=\infty),
\end{array} \mright.\]
where ${\hat\beta}_{0}\in(0,\,1)$ is an arbitrary number. The exponent $\beta$ appears when one considers regularity of solutions to the Poisson system $-\Delta v=f$ with $f$ in $L^{q}$. In fact, from the classical Schauder theory, it is well-known that this weak solution $v$ admits local $C^{1,\,\beta}$-regularity. 

Propositions \ref{Prop: De Giorgi's truncation}--\ref{Prop: Schauder estimate} below will be shown in Sections \ref{Section: Weak formulations}--\ref{Section: Campanato estimates}. 
\begin{proposition}[An oscillation lemma]\label{Prop: De Giorgi's truncation}
Let $u_{\varepsilon}$ be a weak solution to a regularized system (\ref{Eq (Section 2) Approximated system}) in $\Omega$. Assume that positive numbers $\delta,\,\varepsilon,\,\mu,\,F,\,M$ satisfy (\ref{Eq (Section 2) delta-epsilon}), (\ref{Eq (Section 2) control of f}) and
\begin{equation}\label{Eq (Section 2) G-delta-epsilon bound}
\esssup_{B_{\rho}(x_{0})}\,\mleft\lvert {\mathcal G}_{\delta,\,\varepsilon}(Du_{\varepsilon})\mright\rvert\le \mu<\mu+\delta\le M
\end{equation} 
for some open ball $B_{\rho}(x_{0})\Subset \Omega$ with $\rho\in(0,\,1\rbrack$.
Assume that there holds
\begin{equation}\label{Eq (Section 3) Measure condition 1}
\lvert S_{\rho/2,\,\mu,\,\nu}(x_{0}) \rvert\le (1-\nu)\lvert B_{\rho/2}(x_{0})\rvert
\end{equation}
for some constant $\nu\in(0,\,1)$.
Then, we have either
\begin{equation}
\mu^{2}<C_{\star}\rho^{\beta},
\end{equation}
or
\begin{equation}
\esssup_{B_{\rho/4}(x_{0})}\,\mleft\lvert {\mathcal G}_{\delta,\,\varepsilon}(Du_{\varepsilon})\mright\rvert\le \kappa\mu.
\end{equation}
Here the constants $C_{\star}\in(0,\,\infty)$ and $\kappa\in(2^{-\beta},\,1)$ depend at most on $b$, $n$, $N$, $p$, $q$, $\gamma$, $\Gamma$, $F$, $M$, $\delta$ and $\nu$.
\end{proposition}
\begin{proposition}[Campanato-type growth estimates]\label{Prop: Schauder estimate}
Let $u_{\varepsilon}$ be a weak solution to a regularized system (\ref{Eq (Section 2) Approximated system}) in $\Omega$. Assume that positive numbers $\delta$, $\varepsilon$, $\mu$, $F$, $M$ satisfy (\ref{Eq (Section 2) delta-epsilon}), (\ref{Eq (Section 2) control of f}),
\begin{equation}\label{Eq (Section 2) esssup V-epsilon}
\esssup_{B_{\rho}(x_{0})}\,V_{\varepsilon}\le \mu+\delta\le M
\end{equation}
for some open ball $B_{\rho}(x_{0})\Subset \Omega$, and
\begin{equation}\label{Eq (Section 2) delta<mu}
0<\delta<\mu.
\end{equation}
Then, there exist numbers $\nu\in(0,\,1/4),\,\rho_{\star}\in(0,\,1)$, depending at most on $b$, $n$, $N$, $p$, $q$, $\gamma$, $\Gamma$, $F$, $M$, and $\delta$, such that the following statement holds true. If there hold $0<\rho<\rho_{\star}$ and 
\begin{equation}\label{Eq (Section 2) Measure condition 2}
\lvert S_{\rho,\,\mu,\,\nu}(x_{0})\rvert>(1-\nu)\lvert B_{\rho}(x_{0})\rvert,
\end{equation}
then the limit
\begin{equation}
\Gamma_{2\delta,\,\varepsilon}(x_{0})\coloneqq \lim_{r\to 0}({\mathcal G}_{2\delta,\,\varepsilon}(Du_{\varepsilon}))_{x_{0},\,r}\in{\mathbb R}^{Nn}
\end{equation}
exists. Moreover, there hold
\begin{equation}\label{Eq (Section 2) Bound of Gamma-2delta-epsilon}
\lvert \Gamma_{2\delta,\,\varepsilon}(x_{0})\rvert\le \mu,
\end{equation}
and
\begin{equation}\label{Eq (Section 2) Campanato-type growth from Schauder}
\fint_{B_{r}(x_{0})}\mleft\lvert {\mathcal G}_{2\delta,\,\varepsilon}(Du_{\varepsilon})-\Gamma_{2\delta,\,\varepsilon}(x_{0})\mright\rvert^{2}\,{\mathrm{d}}x\le \mleft(\frac{r}{\rho} \mright)^{2\beta} \mu^{2}\quad \textrm{for all }r\in(0,\,\rho\rbrack.
\end{equation}
\end{proposition}
\begin{Remark}\label{Rmk: Equivalence on boundedness}\upshape
Let $\delta$ and $\varepsilon$ satisfy $0<\varepsilon<\delta$.
Then, for $\xi\in{\mathbb R}^{Nn}$, $\mleft\lvert {\mathcal G}_{\delta,\,\varepsilon}(\xi)\mright\rvert\le \mu$ holds if and only if $\xi$ satisfies $\sqrt{\varepsilon^{2}+\lvert\xi\rvert^{2}}\le\mu+\delta$. In particular, the conditions (\ref{Eq (Section 2) G-delta-epsilon bound}) and (\ref{Eq (Section 2) esssup V-epsilon}) are equivalent. Also, it should be noted that the mapping ${\mathcal G}_{2\delta,\,\varepsilon}$ satisfies
\begin{equation}\label{Eq (Section 2) Control of G-2delta-epsilon by G-delta-epsilon}
\mleft\lvert{\mathcal G}_{2\delta,\,\varepsilon}(\xi)\mright\rvert\le \mleft(\mleft\lvert{\mathcal G}_{\delta,\,\varepsilon}(\xi)\mright\rvert -\delta\mright)_{+}\le \mleft\lvert {\mathcal G}_{\delta,\,\varepsilon}(\xi)\mright\rvert\quad \textrm{for all }\xi\in{\mathbb R}^{Nn},
\end{equation}
which is used in the proof of Theorem \ref{Theorem: A priori Hoelder estimate}.
\end{Remark}
By applying Propositions \ref{Prop: Lipschitz bounds}--\ref{Prop: Schauder estimate}, we would like to prove Theorem \ref{Theorem: A priori Hoelder estimate}. 
\begin{proof}
For each fixed \(x_{\ast}\in\Omega\), we first choose
\[R\coloneqq \min\mleft\{\,\frac{1}{2},\,\frac{1}{3}\mathop{\mathrm{dist}}\,(x_{\ast},\,\partial\Omega)\mright\}>0,\]
so that \(B_{2R}(x_{\ast})\Subset \Omega\) holds. By Proposition \ref{Prop: Lipschitz bounds}, we may take a finite constant \(\mu_{0}\in(0,\,\infty)\), depending at most on $b$, $n$, $N$, $p$, $q$, $\gamma$, $\Gamma$, $F$, and $R$, such that
\begin{equation}\label{Eq (Section 2) Lipschitz bound}
\esssup_{B_{R}(x_{\ast})}\,V_{\varepsilon}\le \mu_{0}.
\end{equation}
We set \(M\coloneqq 1+\mu_{0}\), so that \(\mu_{0}+\delta\le M\) clearly holds. 

We choose and fix the numbers \(\nu\in(0,\,1/4),\,\rho_{\star}\in(0,\,1)\) as in Proposition \ref{Prop: Schauder estimate}, which depend at most on $b$, $n$, $N$, $p$, $q$, $\gamma$, $\Gamma$, $F$, $M$, and $\delta$.
Corresponding to this \(\nu\), we choose finite constants \(\kappa\in(2^{-\beta},\,1)\), \(C_{\star}\in\lbrack1,\,\infty)\) as in Proposition \ref{Prop: De Giorgi's truncation}.
We define the desired H\"{o}lder exponent \(\alpha\in(0,\,\beta/2)\) by \(\alpha\coloneqq -\log\kappa/\log 4\), so that the identity \(4^{-\alpha}=\kappa\) holds.
We also put the radius \(\rho_{0}\) such that
\begin{equation}\label{Eq (Section 2) determination of radius}
0<\rho_{0}\le \min\mleft\{\,\frac{R}{2},\,\rho_{\star}\,\mright\}<1\quad\textrm{and}\quad C_{\star}\rho_{0}^{\beta}\le \kappa^{2}\mu_{0}^{2},
\end{equation}
which depends at most on $b$, $n$, $p$, $q$, $\gamma$, $\Gamma$, $F$, $L$, $M$, and $\delta$.
We define non-negative decreasing sequences \(\{\rho_{k}\}_{k=1}^{\infty},\,\{\mu_{k}\}_{k=1}^{\infty}\) by setting \(\rho_{k}\coloneqq 4^{-k}\rho_{0},\,\mu_{k}\coloneqq \kappa^{k}\mu_{0}\) for \(k\in{\mathbb N}\).
By \(2^{-\beta}<\kappa=4^{-\alpha}<1\) and (\ref{Eq (Section 2) determination of radius}), it is easy to check that
\begin{equation}\label{Eq (Section 2) An estimate for induction}
\sqrt{ C_{\star}\rho_{k}^{\beta}}\le 2^{-\beta k}\kappa\mu_{0}\le\kappa^{k+1}\mu_{0}= \mu_{k+1},
\end{equation}
and
\begin{equation}\label{Eq (Section 2) Estimate on radius ratio}
\mu_{k}=4^{-\alpha k}\mu_{0}=\mleft(\frac{\rho_{k}}{\rho_{0}}\mright)^{\alpha}\mu_{0}
\end{equation}
for every \(k\in{\mathbb Z}_{\ge 0}\).

We claim that for every \(x_{0}\in B_{\rho_{0}}(x_{\ast})\), the limit
\[\Gamma_{2\delta,\,\varepsilon}(x_{0})\coloneqq\lim_{r\to 0}\mleft({\mathcal G}_{2\delta,\,\varepsilon}(Du_{\varepsilon})\mright)_{x_{0},\,r}\in{\mathbb R}^{n}\]
exists, and this limit satisfies
\begin{equation}\label{Eq (Section 2) Campanato-type alpha-growth estimate}
\fint_{B_{r}(x_{0})}\mleft\lvert {\mathcal G}_{2\delta,\,\varepsilon}(Du_{\varepsilon})-\Gamma_{2\delta,\,\varepsilon}(x_{0}) \mright\rvert^{2}\,{\mathrm d}x\le 4^{2\alpha+1}\mleft(\frac{r}{\rho_{0}}\mright)^{2\alpha}\mu_{0}^{2}\quad \textrm{for all }r\in(0,\,\rho_{0}\rbrack.
\end{equation}
In the proof of (\ref{Eq (Section 2) Campanato-type alpha-growth estimate}), we introduce a set
\[{\mathcal N}\coloneqq \mleft\{k\in{\mathbb Z}_{\ge 0}\mathrel{}\middle|\mathrel{} \lvert S_{\rho_{k}/2,\,\mu_{k},\,\nu}(x_{0})\rvert>(1-\nu) \lvert B_{\rho_{k}/2}(x_{0})\rvert\mright\},\]
and define \(k_{\star}\in{\mathbb Z}_{\ge 0}\) to be the minimum number of $\mathcal N$ when it is non-empty.
We consider the three possible cases: \({\mathcal N}\neq \emptyset\) and \(\mu_{k_{\star}}> \delta\); \({\mathcal N}\neq \emptyset\) and \(\mu_{k_{\star}}\le \delta\); and \({\mathcal N}=\emptyset\).
Before dealing with each case, we mention that if \({\mathcal N}\neq \emptyset\), then it is possible to apply Proposition \ref{Prop: De Giorgi's truncation} with \((\rho,\,\mu)=(\rho_{k},\,\mu_{k})\) for \(k\in\{\,0,\,1,\,\dots\,,\,k_{\star}-1\,\}\), by the definition of $k_{\star}$. With (\ref{Eq (Section 2) An estimate for induction}) in mind, we are able to obtain 
\begin{equation}\label{Eq (Section 2) Oscillation growth from De Giorgi}
\esssup_{B_{\rho_{k}}(x_{0})}\,\mleft\lvert{\mathcal G}_{\delta,\,\varepsilon}(\nabla u_{\varepsilon})\mright\rvert\le \mu_{k}\quad \textrm{for every }k\in\{\,0,\,1,\,\dots\,,\,k_{\star}\,\}.\end{equation}

In the first case, the conditions \(k_{\star}\in {\mathcal N}\) and \(\mu_{k_{\star}}>\delta\) enable us to apply Proposition \ref{Prop: Schauder estimate} for the open ball \(B_{\rho_{k_{\star}}/2}(x_{0})\) with \(\mu=\mu_{k_{\star}}\). In particular, the limit \(\Gamma_{2\delta,\,\varepsilon}(x_{0})\) exists and this limit satisfies
\begin{equation}\label{Eq (Section 2) Bound of mu-k in main theorem}
\mleft\lvert \Gamma_{2\delta,\,\varepsilon}(x_{0})\mright\rvert\le \mu_{k_{\star}}.
\end{equation}
and
\begin{equation}\label{Eq (Section 2) Campanato-type beta-growth in main theorem}
\fint_{B_{r}(x_{0})}\mleft\lvert{\mathcal G}_{2\delta,\,\varepsilon}(Du_{\varepsilon})-\Gamma_{2\delta,\,\varepsilon}(x_{0})\mright\rvert^{2}\,{\mathrm d}x\le \mleft(\frac{2r}{\rho_{k_{\star}}}\mright)^{2\beta}\mu_{k_{\star}}^{2}\quad \textrm{for all }r\in\mleft( 0,\,\frac{\rho_{k_{\star}}}{2}\mright],
\end{equation}When \(0<r\le \rho_{k_{\star}}/2\), we use (\ref{Eq (Section 2) Estimate on radius ratio}), (\ref{Eq (Section 2) Campanato-type beta-growth in main theorem}), and \(\alpha<\beta\) to get
\[\fint_{B_{r}(x_{0})}\mleft\lvert{\mathcal G}_{2\delta,\,\varepsilon}(Du_{\varepsilon})-\Gamma_{2\delta,\,\varepsilon}(x_{0})\mright\rvert^{2}\,{\mathrm d}x\le \mleft(\frac{2r}{\rho_{k_{\star}}}\mright)^{2\alpha}\mleft(\frac{\rho_{k_{\star}}}{\rho_{0}}\mright)^{2\alpha}\mu_{0}^{2}= 4^{\alpha}\mleft(\frac{r}{\rho_{0}}\mright)^{2\alpha}\mu_{0}^{2}.\]
When \(\rho_{k_{\star}}/2<r\le \rho_{0}\), there corresponds a unique integer \(k\in\{\,0,\,\dots\,,\,k_{\star}\,\}\) such that \(\rho_{k+1}<r\le \rho_{k}\). By (\ref{Eq (Section 2) Oscillation growth from De Giorgi})--(\ref{Eq (Section 2) Bound of mu-k in main theorem}), we compute
\begin{align*}
\fint_{B_{r}(x_{0})}\mleft\lvert{\mathcal G}_{2\delta,\,\varepsilon}(Du_{\varepsilon})-\Gamma_{2\delta,\,\varepsilon}(x_{0})\mright\rvert^{2}\,{\mathrm d}x&\le 2\mleft(\esssup_{B_{\rho_{k}}(x_{0})}\,\mleft\lvert {\mathcal G}_{\delta,\,\varepsilon}(Du_{\varepsilon})\mright\rvert^{2}+\mleft\lvert\Gamma_{2\delta,\,\varepsilon}(x_{0}) \mright\rvert^{2}\mright)\\&\le 4\mu_{k}^{2}\le 4 \mleft(\frac{\rho_{k}}{\rho_{0}}\mright)^{2\alpha}\mu_{0}^{2}\le 4 \mleft(\frac{4r}{\rho_{0}}\mright)^{2\alpha}\mu_{0}^{2}.
\end{align*}

In the second case, we recall (\ref{Eq (Section 2) Control of G-2delta-epsilon by G-delta-epsilon}) in Remark \ref{Rmk: Equivalence on boundedness} to notice \({\mathcal G}_{2\delta,\,\varepsilon}(Du_{\varepsilon})=0\) a.e. in \(B_{k_{\star}}(x_{0})\). Combining with (\ref{Eq (Section 2) Oscillation growth from De Giorgi}), we can easily check
\begin{equation}\label{Eq (Section 2) G-2delta-epsilon decay in digital}
\esssup_{B_{\rho_{k}}(x_{0})}\,\mleft\lvert{\mathcal G}_{2\delta,\,\varepsilon}(Du_{\varepsilon})\mright\rvert\le \mu_{k}\quad \textrm{for every }k\in{\mathbb Z}_{\ge 0},
\end{equation}
which clearly yields \(\Gamma_{2\delta,\,\varepsilon}(x_{0})=0\). For every \(r\in (0,\,\rho_{0}\rbrack\), there corresponds a unique integer \(k\in{\mathbb Z}_{\ge 0}\) such that \(\rho_{k+1}<r\le \rho_{k}\). By (\ref{Eq (Section 2) G-2delta-epsilon decay in digital}) and \(\kappa=4^{-\alpha}\), we have
\begin{align*}
\fint_{B_{r}(x_{0})}\mleft\lvert {\mathcal G}_{2\delta,\,\varepsilon}(Du_{\varepsilon})-\Gamma_{2\delta,\,\varepsilon}(x_{0})\mright\rvert^{2}\,{\mathrm d}x&=\fint_{B_{r}(x_{0})}\mleft\lvert {\mathcal G}_{2\delta,\,\varepsilon}(Du_{\varepsilon})\mright\rvert^{2}\,{\mathrm d}x\\&\le \esssup_{B_{\rho_{k}}(x_{0})}\,\mleft\lvert{\mathcal G}_{2\delta,\,\varepsilon}(Du_{\varepsilon})\mright\rvert^{2}\le \mu_{k}^{2}=4^{-2\alpha k}\mu_{0}^{2}\\&\le \mleft[4\cdot\mleft(\frac{r}{\rho_{0}}\mright)\mright]^{2\alpha}\mu_{0}^{2}=16^{\alpha}\mleft(\frac{r}{\rho_{0}}\mright)^{2\alpha}\mu_{0}^{2}.
\end{align*}

In the remaining case \({\mathcal N}=\emptyset\), it is clear that there holds \(\lvert E_{\rho_{k}/2,\,\mu_{k},\,\nu_{k}}\rvert \le (1-\nu)\lvert B_{\rho_{k}/2}(x_{0})\rvert\) for every \(k\in{\mathbb Z}_{\ge 0}\). Applying (\ref{Eq (Section 2) An estimate for induction}) and Proposition \ref{Prop: De Giorgi's truncation} with \((\rho,\,\mu)=(\rho_{k},\,\mu_{k})\,(k\in{\mathbb Z}_{\ge 0})\) repeatedly, we can easily check that 
\[\esssup_{B_{\rho_{k}}(x_{0})}\,\mleft\lvert{\mathcal G}_{\delta,\,\varepsilon}(Du_{\varepsilon})\mright\rvert\le \mu_{k}\quad \textrm{for every }k\in{\mathbb Z}_{\ge 0}.\]
In particular, (\ref{Eq (Section 2) G-2delta-epsilon decay in digital}) clearly follows from this and (\ref{Eq (Section 2) Control of G-2delta-epsilon by G-delta-epsilon}). Thus, the proof of (\ref{Eq (Section 2) Campanato-type alpha-growth estimate}) can be accomplished, similarly to the second case. 
In any possible cases, we conclude that $\Gamma_{2\delta,\,\varepsilon}(x_{0})$ exists and satisfies (\ref{Eq (Section 2) Campanato-type alpha-growth estimate}), as well as
\begin{equation}\label{Eq (Section 2) Boundedness on Gamma-2delta-epsilon}
\mleft\lvert \Gamma_{2\delta,\,\varepsilon}(x_{0})\mright\rvert\le \mu_{0}\quad \textrm{for all }x_{0}\in B_{\rho_{0}}(x_{\ast}),
\end{equation}
which immediately follows from (\ref{Eq (Section 2) Lipschitz bound}).

From (\ref{Eq (Section 2) Campanato-type alpha-growth estimate}) and (\ref{Eq (Section 2) Boundedness on Gamma-2delta-epsilon}), we would like to show
\begin{equation}\label{Eq (Section 2) Hoelder estimate on Gamma-2delta-epsilon}
\mleft\lvert \Gamma_{2\delta,\,\varepsilon}(x_{1})-\Gamma_{2\delta,\,\varepsilon}(x_{2})\mright\rvert\le \mleft(\frac{2^{2\alpha+2+n/2}}{\rho_{0}^{\alpha}}\mu_{0}\mright)\lvert x_{1}-x_{2}\rvert^{\alpha}
\end{equation}
for all \(x_{1},\,x_{2}\in B_{\rho_{0}/2}(x_{\ast})\). We prove (\ref{Eq (Section 2) Hoelder estimate on Gamma-2delta-epsilon}) by dividing into the two possible cases. When \(r\coloneqq \lvert x_{1}-x_{2}\rvert\le \rho_{0}/2\), we set a point \(x_{3}\coloneqq (x_{1}+x_{2})/2\in B_{\rho_{0}}(x_{\ast})\). By \(B_{r/2}(x_{3})\subset B_{r}(x_{j})\subset B_{\rho_{0}/2}(x_{j})\subset B_{\rho_{0}}(x_{\ast})\) for each \(j\in\{\,1,\,2\,\}\), we use (\ref{Eq (Section 2) Campanato-type alpha-growth estimate}) to obtain
\begin{align*}
&\mleft\lvert\Gamma_{2\delta,\,\varepsilon}(x_{1})-\Gamma_{2\delta,\,\varepsilon}(x_{2}) \mright\rvert^{2}\\&=\fint_{B_{r/2}(x_{3})}\mleft\lvert\Gamma_{2\delta,\,\varepsilon}(x_{1})-\Gamma_{2\delta,\,\varepsilon}(x_{2}) \mright\rvert^{2}\,{\mathrm d}x\\&\le 2\mleft(\fint_{B_{r/2}(x_{3})}\mleft\lvert{\mathcal G}_{2\delta,\,\varepsilon}(Du_{\varepsilon})-\Gamma_{2\delta,\,\varepsilon}(x_{1})\mright\rvert^{2}\,{\mathrm d}x +\fint_{B_{r/2}(x_{3})}\mleft\lvert{\mathcal G}_{2\delta,\,\varepsilon}(Du_{\varepsilon})-\Gamma_{2\delta,\,\varepsilon}(x_{2})\mright\rvert^{2}\,{\mathrm d}x\mright) \\& \le 2^{n+1}\mleft(\fint_{B_{r}(x_{1})}\mleft\lvert{\mathcal G}_{2\delta,\,\varepsilon}(Du_{\varepsilon})-\Gamma_{2\delta,\,\varepsilon}(x_{1})\mright\rvert^{2}\,{\mathrm d}x +\fint_{B_{r}(x_{2})}\mleft\lvert{\mathcal G}_{2\delta,\,\varepsilon}(Du_{\varepsilon})-\Gamma_{2\delta,\,\varepsilon}(x_{2})\mright\rvert^{2}\,{\mathrm d}x\mright)\\&\le 2^{n+4}\cdot 16^{\alpha}\mleft(\frac{r}{\rho_{0}}\mright)^{2\alpha}\mu_{0}^{2}=\mleft(\frac{2^{2\alpha+2+n/2}}{\rho_{0}^{\alpha}}\mu_{0}\mright)^{2}\lvert x_{1}-x_{2}\rvert^{2\alpha}, 
\end{align*}
which yields (\ref{Eq (Section 2) Hoelder estimate on Gamma-2delta-epsilon}). In the remaining case \(\lvert x_{1}-x_{2}\rvert>\rho_{0}/2\), we simply use (\ref{Eq (Section 2) Boundedness on Gamma-2delta-epsilon}) to get
\[\mleft\lvert\Gamma_{2\delta,\,\varepsilon}(x_{1})-\Gamma_{2\delta,\,\varepsilon}(x_{2})\mright\rvert\le 2\mu_{0}\le 2\cdot \frac{2^{\alpha}\lvert x_{1}-x_{2}\rvert^{\alpha}}{\rho_{0}^{\alpha}}\mu_{0},\]
which completes the proof of (\ref{Eq (Section 2) Hoelder estimate on Gamma-2delta-epsilon}).
Finally, since the mapping \(\Gamma_{2\delta,\,\varepsilon}\) is a Lebesgue representative of \({\mathcal G}_{2\delta,\,\varepsilon}(Du_{\varepsilon})\in L^{p}(\Omega;\,{\mathbb R}^{Nn})\), the claims (\ref{Eq (Section 2) Bound of G-2delta-epsilon})--(\ref{Eq (Section 2) Holder continuity of G-2delta-epsilon}) immediately follow from (\ref{Eq (Section 2) Boundedness on Gamma-2delta-epsilon})--(\ref{Eq (Section 2) Hoelder estimate on Gamma-2delta-epsilon}).
\end{proof}
We would like to conclude Section \ref{Section: Approximation} by giving the proof of our main Theorem \ref{Theorem: C1-regularity}.
\begin{proof}
Let $u\in W^{1,\,p}(\Omega;\,{\mathbb R}^{N})$ be a weak solution to (\ref{Eq (Section 1) main system}). We choose and fix $\{f_{\varepsilon}\}_{0<\varepsilon<1}\subset L^{q}(\Omega;\,{\mathbb R}^{Nn})$ enjoying (\ref{Eq (Section 2) Weak convergence of f}), and for each fixed $\varepsilon\in(0,\,1)$, we define the functional ${\mathcal F}_{\varepsilon}$ by (\ref{Eq (Section 2) energy functional: relaxed}). We set a function
\[u_{\varepsilon}\coloneqq \argmin\mleft\{v\in u+W_{0}^{1,\,p}(\Omega;\,{\mathbb R}^{N})\mathrel{}\middle|\mathrel{} {\mathcal F}_{\varepsilon}(v)\mright\}\quad \textrm{for every }\varepsilon\in(0,\,1),\]
which is a unique solution of the Dirichlet problem
\[\mleft\{\begin{array}{ccccc}
{\mathcal L}^{\varepsilon}u_{\varepsilon}& = & f_{\varepsilon} & \textrm{in} &\Omega,\\ u_{\varepsilon} & = & u & \textrm{on} & \partial\Omega.
\end{array} \mright.\]
By Proposition \ref{Prop: Convergence on weak system}, there exists a decreasing sequence $\{\varepsilon_{j}\}_{j=1}^{\infty}\subset(0,\,1)$ such that $\varepsilon_{j}\to 0$ and $u_{\varepsilon_{j}}\to u$ in $W^{1,\,p}(\Omega;\,{\mathbb R}^{N})$.
In particular, we may let (\ref{Eq (Section 2) a.e. convergence of Du}) hold by taking a subsequence if necessary. Also, we may choose finite constants $L>\lVert Du_{\varepsilon}\rVert_{L^{p}(\Omega)}$ and $F>\lVert f_{\varepsilon}\rVert_{L^{q}(\Omega)}$ that satisfy (\ref{Eq (Section 2) control of Du})--(\ref{Eq (Section 2) control of f}).

Fix $x_{\ast}\in\Omega$ and $\delta\in(0,\,1)$ arbitrarily. Then Theorem \ref{Theorem: A priori Hoelder estimate} enables us to apply the Arzel\`{a}--Ascoli theorem to the net $\{{\mathcal G}_{2\delta,\,\varepsilon}(Du_{\varepsilon})\}_{0<\varepsilon<\delta/4}\subset C^{\alpha}(B_{\rho_{0}/2}(x_{\ast});\,{\mathbb R}^{Nn})$. In particular, by taking a subsequence if necessary, we are able to conclude that ${\mathcal G}_{2\delta,\,\varepsilon_{j}}(Du_{\varepsilon_{j}})\in C^{\alpha}(B_{\rho_{0}/2}(x_{\ast});\,{\mathbb R}^{Nn})$ uniformly converges to a continuous mapping $v_{2\delta}\in C^{\alpha}(B_{\rho_{0}/2}(x_{\ast});\,{\mathbb R}^{Nn})$.
On the other hand, by (\ref{Eq (Section 2) a.e. convergence of Du}) we have already known that ${\mathcal G}_{2\delta,\,\varepsilon_{j}}(Du_{\varepsilon_{j}})\rightarrow {\mathcal G}_{2\delta}(Du)$ a.e. in $\Omega$. Hence, it follows that $v_{2\delta}={\mathcal G}_{2\delta}(Du)$ a.e. in $B_{\rho_{0}/2}(x_{\ast})$. Moreover, from (\ref{Eq (Section 2) Bound of G-2delta-epsilon})--(\ref{Eq (Section 2) Holder continuity of G-2delta-epsilon}), we conclude that ${\mathcal G}_{2\delta}(Du)$ satisfies (\ref{Eq (Section 1) local bounds of Jacobian multiplied with G-delta})--(\ref{Eq (Section 1) local continuity of Jacobian multiplied with G-delta}).

We have already proved ${\mathcal G}_{\delta}(Du)\in C^{0}(\Omega;\,{\mathbb R}^{Nn})$ for each fixed $\delta\in(0,\,1)$, from which we show continuity of $Du$. By the definition of ${\mathcal G}_{\delta}$, we can easily check that
\[\sup_{\Omega}\,\mleft\lvert{\mathcal G}_{\delta_{1}}(Du)-{\mathcal G}_{\delta_{2}}(Du)\mright\rvert\le \lvert \delta_{1}-\delta_{2}\rvert\quad \textrm{for all }\delta_{1},\,\delta_{2}\in(0,\,1).\]
In particular, the net $\{{\mathcal G}_{\delta}(Du)\}_{0<\delta<1}\subset C^{0}(\Omega;\,{\mathbb R}^{Nn})$ uniformly converges to a continuous mapping $v_{0}\in C^{0}(\Omega;\,{\mathbb R}^{Nn})$. On the other hand, it is clear that ${\mathcal G}_{\delta}(Du)\rightarrow Du$ a.e. in $\Omega$. Thus, we realize $v_{0}=Du$ a.e. in $\Omega$, and this completes the proof.
\end{proof}
\section{Weak formulations}\label{Section: Weak formulations}
\subsection{A basic weak formulation}\label{Subsect: Basic weak formulations}
We note that for each fixed $\varepsilon\in(0,\,1)$, the regularized system (\ref{Eq (Section 2) Approximated system}) is uniformly elliptic. To be precise, by direct computations, we can easily check that the relaxed density $E_{\varepsilon}(\xi)=g_{\varepsilon}(\lvert \xi\rvert^{2})$ admits a constant $C_{\varepsilon}\in(\gamma,\,\infty)$, depending on $\varepsilon$, such that
\begin{equation}\label{Eq (Section 3) Uniform ellipticity on approximated density}
\gamma\mleft(\varepsilon^{2}+\lvert \xi\rvert^{2}\mright)^{p/2-1}{\mathrm{id}_{Nn}}\leqslant {\mathcal B}_{\varepsilon}(\xi)=D^{2}E_{\varepsilon}(\xi) \leqslant C_{\varepsilon}\mleft(\varepsilon^{2}+\lvert \xi\rvert^{2}\mright)^{p/2-1}{\mathrm{id}_{Nn}}\quad \textrm{for all }\xi\in{\mathbb R}^{Nn}
\end{equation}
Hence, it is not restrictive to let $u_{\varepsilon}\in W_{\mathrm{loc}}^{1,\,\infty}(\Omega;\,{\mathbb R}^{N})\cap W_{\mathrm{loc}}^{2,\,2}(\Omega;\,{\mathbb R}^{N})$ (see also Remark \ref{Eq Higher W-2-2 and W-1-infty regularity} in Section \ref{Section: Appendix}). 
In particular, for each fixed $B_{\rho}(x_{0})\Subset\Omega$, we are able to deduce
\begin{equation}\label{Eq (Section 3) Weak formulation local}
\int_{B_{\rho}(x_{0})}\langle A_{\varepsilon}(Du_{\varepsilon})\mid D\phi\rangle\,{\mathrm{d}}x=\int_{B_{\rho}(x_{0})}\langle f_{\varepsilon}\mid \phi\rangle\,{\mathrm{d}}x
\end{equation}
for all $\phi\in W_{0}^{1,\,1}(B_{\rho}(x_{0});\,{\mathbb R}^{Nn})$.
Also, differentiating (\ref{Eq (Section 2) Approximated system}) by $x_{\alpha}$ and integrating by parts, we can deduce an weak formulation
\begin{equation}\label{Eq (Section 3) Weak formulation differentiated}
\int_{B_{\rho}(x_{0})}\mleft\langle D_{\alpha}\mleft(A_{\varepsilon}\mleft(Du_{\varepsilon}\mright) \mright)\mathrel{} \middle|\mathrel{} D\phi\mright\rangle\,{\mathrm{d}}x=-\int_{B_{\rho}(x_{0})}\langle f_{\varepsilon}\mid D_{\alpha}\phi\rangle\,{\mathrm{d}}x
\end{equation}
for all $\phi\in W_{0}^{1,\,2}(B_{\rho}(x_{0});\,{\mathbb R}^{N})$, $\alpha\in\{\,1,\,\dots\,,\,n\,\}$. Following \cite[Lemma 4.1]{BDGPdN} (see also \cite[Lemma 3.5]{T-scalar}), we would like to give a basic weak formulation that is fully used in our computations.
\begin{lemma}\label{Lemma: Weak formulation for V-epsilon}
Let $u_{\varepsilon}$ be a weak solution to the regularized system (\ref{Eq (Section 2) Approximated system}) with $0<\varepsilon<1$. Assume that a non-decreasing and non-negative function $\psi\in W_{\mathrm{loc}}^{1,\,\infty}(\lbrack0,\,\infty))$ is differentiable except at finitely many points.
For any non-negative function $\zeta\in C_{c}^{1}(B_{\rho}(x_{0}))$, we set 
\[\mleft\{\begin{array}{ccl}
J_{1} & \coloneqq &  \displaystyle\int_{B_{\rho}(x_{0})}\mleft\langle {\mathcal C}_{\varepsilon}(Du_{\varepsilon})\nabla V_{\varepsilon}\mathrel{} \middle|\mathrel{} \nabla\zeta\mright\rangle\psi(V_{\varepsilon})V_{\varepsilon}\,{\mathrm{d}}x,\\ J_{2} &\coloneqq & \displaystyle\int_{B_{\rho}(x_{0})}\mleft\langle {\mathcal C}_{\varepsilon}(Du_{\varepsilon})\nabla V_{\varepsilon}\mathrel{} \middle|\mathrel{} \nabla V_{\varepsilon}\mright\rangle\zeta \psi^{\prime}(V_{\varepsilon}) V_{\varepsilon}\,{\mathrm{d}}x,\\ J_{3} & \coloneqq &\displaystyle\int_{B_{\rho}(x_{0})}V_{\varepsilon}^{p-2}\mleft\lvert D^{2}u_{\varepsilon}\mright\rvert^{2}\zeta\psi(V_{\varepsilon})\,{\mathrm{d}}x,\\ J_{4}& \coloneqq & \displaystyle\int_{B_{\rho}(x_{0})}\lvert f_{\varepsilon}\rvert^{2}\psi(V_{\varepsilon})V_{\varepsilon}^{2-p}\zeta\,{\mathrm{d}}x,\\ J_{5}&\coloneqq &\displaystyle\int_{B_{\rho}(x_{0})}\lvert f_{\varepsilon}\rvert^{2} \psi^{\prime}(V_{\varepsilon})V_{\varepsilon}^{3-p}\zeta\,{\mathrm{d}}x,\\  J_{6}&\coloneqq & \displaystyle\int_{B_{\rho}(x_{0})}\lvert f_{\varepsilon}\rvert \lvert\nabla\zeta\rvert\psi(V_{\varepsilon}) V_{\varepsilon}\,{\mathrm{d}}x,
\end{array} \mright.\]
with
\[{\mathcal C}_{\varepsilon}(Du_{\varepsilon})\coloneqq g_{\varepsilon}^{\prime}(\lvert Du_{\varepsilon}\rvert^{2}){\mathrm{id}_{n}}+2g_{\varepsilon}^{\prime\prime}(\lvert Du_{\varepsilon}\rvert^{2})\sum_{i=1}^{N}\nabla u_{\varepsilon}^{i}\otimes\nabla u_{\varepsilon}^{i}.\]
Then we have
\begin{equation}\label{Eq (Section 3) Resulting weak formulation}
2J_{1}+J_{2}+\gamma J_{3}\le \frac{1}{\gamma}(nJ_{4}+J_{5})+2J_{6}.
\end{equation}
\end{lemma}
Before the proof, we mention that the matrix-valued functions ${\mathcal C}_{\varepsilon}(Du_{\varepsilon})$ satisfy
\begin{equation}\label{Eq (Section 3) Ellipticity of coefficients}
\gamma V_{\varepsilon}^{p-2} {\mathrm{id}_{n}} \leqslant {\mathcal C}_{\varepsilon}(Du_{\varepsilon})\leqslant \mleft(bV_{\varepsilon}^{-1}+3\Gamma V_{\varepsilon}^{p-2}\mright){\mathrm{id}_{n}}\quad \textrm{a.e. in }\Omega,
\end{equation}
which follows from (\ref{Eq (Section 1) Growth g-p-pprime})--(\ref{Eq (Section 1) Ellipticity g_p}) and (\ref{Eq (Section 2) Growth g-1-pprime})--(\ref{Eq (Section 2) Degenerate ellipticity of g-1}).
\begin{proof}
For notational simplicity we write $B\coloneqq B_{\rho}(x_{0})$.
For a fixed index \(\alpha\in\{\,1,\,\dots\,,\,n\,\}\), we test $\phi\coloneqq \zeta\psi(V_{\varepsilon})D_{{\alpha}}u_{\varepsilon}\in W_{0}^{1,\,2}(B;\,{\mathbb R}^{N})$ into (\ref{Eq (Section 3) Weak formulation differentiated}). Summing over $\alpha\in\{\,1,\,\dots\,,\,n\,\}$, we have
\begin{align*}
&J_{0}+J_{1}+J_{2}\\&=-\int_{B}\psi(V_{\varepsilon})\langle f_{\varepsilon}\mid Du_{\varepsilon}\nabla\zeta \rangle\,{\mathrm{d}}x-\int_{B}\zeta\psi^{\prime}(V_{\varepsilon})\langle f_{\varepsilon}\mid Du_{\varepsilon}\nabla V_{\varepsilon}\rangle\,{\mathrm{d}}x-\int_{B}\zeta\psi(V_{\varepsilon}) \sum_{i=1}^{N}\sum_{\alpha=1}^{n}f_{\varepsilon}^{i}\partial_{x_{\alpha}}^{2}u_{\varepsilon}^{i}\,{\mathrm{d}}x \\&\eqqcolon -(J_{7}+J_{8}+J_{9})
\end{align*}
with
\[J_{0}\coloneqq \int_{B}\mleft[g_{\varepsilon}^{\prime}(\lvert Du_{\varepsilon}\rvert^{2})\lvert D^{2}u_{\varepsilon}\rvert^{2}+\frac{1}{2}g_{\varepsilon}^{\prime\prime}(\lvert Du_{\varepsilon}\rvert^{2})\mleft\lvert \nabla \lvert Du_{\varepsilon}\rvert^{2}\mright\rvert^{2}\mright]\zeta\psi(V_{\varepsilon})\,{\mathrm{d}}x.\]
To estimate $J_{0}$, we recall $\mleft\lvert \nabla \lvert Du_{\varepsilon}\rvert^{2}\mright\rvert^{2}\le 4\lvert D^{2}u_{\varepsilon}\rvert^{2}\lvert Du_{\varepsilon}\rvert^{2}$,
which follows from the Cauchy--Schwarz inequality. By this and (\ref{Eq (Section 1) Ellipticity g_p}), we are able to deduce
\[g_{p,\,\varepsilon}^{\prime}(\lvert Du_{\varepsilon}\rvert^{2})\lvert D^{2}u_{\varepsilon}\rvert^{2}+\frac{1}{2}g_{p,\,\varepsilon}^{\prime\prime}(\lvert Du_{\varepsilon}\rvert^{2})\mleft\lvert \nabla \lvert Du_{\varepsilon}\rvert^{2}\mright\rvert^{2}\ge \gamma V_{\varepsilon}^{p-2}\mleft\lvert D^{2}u_{\varepsilon}\mright\rvert^{2}\]
a.e. in $\Omega$.
Similarly, we use (\ref{Eq (Section 2) Degenerate ellipticity of g-1}) to get
\[g_{1,\,\varepsilon}^{\prime}(\lvert Du_{\varepsilon}\rvert^{2})\lvert D^{2}u_{\varepsilon}\rvert^{2}+\frac{1}{2}g_{1,\,\varepsilon}^{\prime\prime}(\lvert Du_{\varepsilon}\rvert^{2})\mleft\lvert \nabla \lvert Du_{\varepsilon}\rvert^{2}\mright\rvert^{2}\ge 0\]
a.e. in $\Omega$.
These inequalities imply $J_{0}\ge \gamma J_{3}$.
Clearly, $\lvert J_{7}\rvert\le J_{6}$ holds.
To compute $J_{8},\,J_{9}$, we recall (\ref{Eq (Section 3) Ellipticity of coefficients}). Then by Young's inequality we can compute
\[\lvert J_{8}\rvert\le \frac{\gamma}{2}\int_{B}V_{\varepsilon}^{p-1}\lvert\nabla V_{\varepsilon}\rvert^{2}\zeta\psi^{\prime}(V_{\varepsilon})\,{\mathrm{d}}x+\frac{1}{2\gamma}\int_{B}\lvert f_{\varepsilon}\rvert^{2} V_{\varepsilon}^{3-p}\zeta\psi^{\prime}(V_{\varepsilon})\,{\mathrm{d}}x\le \frac{1}{2}J_{2}+\frac{1}{2\gamma}J_{5},\]
and 
\begin{align*}
\lvert J_{9}\rvert&\le \sqrt{n}\int_{B}\lvert f_{\varepsilon}\rvert\zeta\psi(V_{\varepsilon})\mleft\lvert D^{2}u_{\varepsilon} \mright\rvert\,{\mathrm{d}}x\\&\le \frac{\gamma}{2}\int_{B}V_{\varepsilon}^{p-1}\mleft\lvert D^{2}u_{\varepsilon} \mright\rvert^{2}\zeta\psi(V_{\varepsilon})\,{\mathrm{d}}x+\frac{n}{2\gamma}\int_{B}\lvert f_{\varepsilon}\rvert^{2}\psi(V_{\varepsilon})V_{\varepsilon}^{2-p}\zeta\,{\mathrm{d}}x\\&=\frac{\gamma}{2}J_{3}+\frac{n}{2\gamma}J_{4}.
\end{align*}
As a result, we get
\[J_{1}+J_{2}+\gamma J_{3}\le \frac{1}{2}J_{2}+\frac{\gamma}{2}J_{3}+\frac{n}{2\gamma}J_{4}+\frac{1}{2\gamma}J_{5}+J_{6},\]
from which (\ref{Eq (Section 3) Resulting weak formulation}) immediately follows.
\end{proof}

\subsection{De Giorgi's truncation and Caccioppoli-type estimates}\label{Subsect: De Giorgi's truncation}
In the resulting weak formulation (\ref{Eq (Section 3) Resulting weak formulation}), we may discard two non-negative integrals $J_{2},\,J_{3}$. Then, (\ref{Eq (Section 3) Resulting weak formulation}) implies that the scalar function $V_{\varepsilon}$ is a subsolution to an elliptic problem.
By appealing to convex composition of $V_{\varepsilon}$, we would like to prove that another scalar function $U_{\delta,\,\varepsilon}$ defined by
\begin{equation}\label{Eq (Section 3) U-delta-epsilon}
U_{\delta,\,\varepsilon}\coloneqq (V_{\varepsilon}-\delta)_{+}^{2}\in L_{\mathrm{loc}}^{\infty}(\Omega)\cap W_{\mathrm{loc}}^{1,\,2}(\Omega)
\end{equation}
is also a subsolution to a certain \textit{uniformly} elliptic problem. This is possible since this function is supported in a place $\{\,V_{\varepsilon}>\delta\,\}$, where the system (\ref{Eq (Section 2) Approximated system}) becomes uniformly elliptic. Thus, as in Lemma \ref{Lemma: Subsolution} below, we are able to obtain Caccioppoli-type estimates for $U_{\delta,\,\varepsilon}$.
\begin{lemma}\label{Lemma: Subsolution}
Let $u_{\varepsilon}$ be a weak solution to (\ref{Eq (Section 2) Approximated system}) with $0<\varepsilon<1$. Assume that positive numbers $\delta$, $\mu$, $M$, and an open ball $B_{\rho}(x_{0})$ satisfy (\ref{Eq (Section 2) G-delta-epsilon bound}). Then, the scalar function $U_{\delta,\,\varepsilon}$ defined by (\ref{Eq (Section 3) U-delta-epsilon}) satisfies
\begin{equation}\label{Eq (Section 3) U-delta-epsilon is subsol}
\int_{B_{\rho}(x_{0})}\mleft\langle {\mathcal A}_{\delta,\,\varepsilon}(Du_{\varepsilon})\nabla U_{\delta,\,\varepsilon}\mathrel{}\middle|\mathrel{}\nabla\zeta \mright\rangle\,{\mathrm{d}}x\le C_{0}\mleft[\int_{B_{\rho}(x_{0})}\lvert f_{\varepsilon}\rvert^{2}\zeta\,{\mathrm{d}}x+\int_{B_{\rho}(x_{0})}\lvert f_{\varepsilon}\rvert\lvert\nabla \zeta\rvert\,{\mathrm{d}}x\mright]
\end{equation}
for all non-negative function $\zeta\in C_{c}^{1}(B_{\rho}(x_{0}))$, where ${\mathcal A}_{\delta,\,\varepsilon}(Du_{\varepsilon})$ is an $n\times n$ matrix-valued function satisfying
\begin{equation}\label{Eq (Section 3) Uniform ellipticity of A-delta-epsilon}
\gamma_{\ast}{\mathrm{id}_{n}} \leqslant {\mathcal A}_{\delta,\,\varepsilon}(Du_{\varepsilon})\leqslant \Gamma_{\ast}{\mathrm{id}_{n}}\quad \textrm{a.e. in }B_{\rho}(x_{0}).
\end{equation}
The constants $C_{0}\in(0,\,\infty)$ and $0<\gamma_{\ast}\le \Gamma_{\ast}<\infty$ depend at most on $b$, $n$, $p$, $\gamma$, $\Gamma$, $M$, and $\delta$. 
In particular, we have
\begin{equation}\label{Eq (Section 3) Caccioppoli estimate}
\int_{B_{\rho}(x_{0})}\lvert \nabla[\eta (U_{\delta,\,\varepsilon}-k)_{+}]\rvert^{2}\,{\mathrm{d}}x\le C\mleft[\int_{B_{\rho}(x_{0})}\lvert \nabla \eta\rvert^{2}(U_{\delta,\,\varepsilon}-k)_{+}^{2}\,{\mathrm{d}}x+\mu^{2}\int_{A_{k,\,\rho}(x_{0})}\lvert f_{\varepsilon}\rvert^{2}\eta^{2}\,{\mathrm{d}}x \mright]
\end{equation}
for all $k\in(0,\,\infty)$ and for any non-negarive function $\eta\in C_{c}^{1}(B_{\rho}(x_{0}))$. Here $A_{k,\,\rho}(x_{0})\coloneqq \{ x\in B_{\rho}(x_{0})\mid U_{\delta,\,\varepsilon}(x)>k\}$, and the constant $C\in(0,\,\infty)$ depends on $\gamma_{\ast}$, $\Gamma_{\ast}$, and $C_{0}$.
\end{lemma}
\begin{proof}
We choose $\psi(t)\coloneqq (t-\delta)_{+}$ for $t\in\lbrack 0,\,\infty)$, so that $\psi(V_{\varepsilon})=\lvert{\mathcal G}_{\delta,\,\varepsilon}(Du_{\varepsilon})\rvert$ holds. From (\ref{Eq (Section 3) Resulting weak formulation}), we will deduce (\ref{Eq (Section 3) U-delta-epsilon is subsol}).
To compute $J_{1}$, we note that $\psi(V_{\varepsilon})$ vanishes when $V_{\varepsilon}\le \delta$, and that the identity $\nabla U_{\delta,\,\varepsilon}=2\psi(V_{\varepsilon})\nabla V_{\varepsilon}$ holds. From these we obtain
\[J_{1}=\int_{B_{\rho}(x_{0})}\mleft\langle V_{\varepsilon} {\mathcal C}_{\varepsilon}(Du_{\varepsilon})\psi(V_{\varepsilon})\nabla V_{\varepsilon}\mathrel{} \middle|\mathrel{} \nabla\zeta\mright\rangle\,{\mathrm{d}}x=\frac{1}{2}\int_{B_{\rho}(x_{0})}\mleft\langle {\mathcal A}_{\delta,\,\varepsilon}(Du_{\varepsilon})\nabla U_{\delta,\,\varepsilon}\mathrel{}\middle|\mathrel{}\nabla\zeta \mright\rangle\,{\mathrm{d}}x\]
with
\[{\mathcal A}_{\delta,\,\varepsilon}(Du_{\varepsilon})\coloneqq\mleft\{\begin{array}{cc}
V_{\varepsilon}{\mathcal C}_{\delta,\,\varepsilon}(Du_{\varepsilon}) & (\textrm{if }V_{\varepsilon}>\delta),\\ \mathrm{id}_{n} & (\textrm{otherwise}). \end{array} \mright.\]
By (\ref{Eq (Section 3) Ellipticity of coefficients}), the coefficient matrix ${\mathcal A}_{\delta,\,\varepsilon}(Du_{\varepsilon})$ satisfies (\ref{Eq (Section 3) Uniform ellipticity of A-delta-epsilon}) with \[\gamma_{\ast}=\min\mleft\{\,1,\,\gamma\delta^{p-1}\mright\},\quad\Gamma_{\ast}=\max\mleft\{\,1,\,b+3\Gamma M^{p-1}\,\mright\}.\]
By discarding positive integrals $J_{2}$, $J_{3}$, we compute
\begin{align*}
J_{1}&\le \frac{1}{2\gamma}\int_{B_{\rho}(x_{0})}\lvert f_{\varepsilon}\rvert^{2}\frac{n\psi(V_{\varepsilon})V_{\varepsilon}+\psi^{\prime}(V_{\varepsilon})V_{\varepsilon}^{2}}{V_{\varepsilon}^{p-1}}\zeta\,{\mathrm{d}}x+ \int_{B_{\rho}(x_{0})}\lvert f_{\varepsilon}\rvert \lvert\nabla\zeta\rvert\psi(V_{\varepsilon}) V_{\varepsilon}\,{\mathrm{d}}x\\&\le \frac{1}{2\gamma}\int_{B_{\rho}(x_{0})}\lvert f_{\varepsilon}\rvert^{2}\frac{n\mu M+M^{2}}{\delta^{p-1}}\zeta\,{\mathrm{d}}x+\mu M\int_{B_{\rho}(x_{0})}\lvert f_{\varepsilon}\rvert\lvert \nabla\zeta \rvert\,{\mathrm{d}}x,
\end{align*}
from which (\ref{Eq (Section 3) U-delta-epsilon is subsol}) immediately follows.

The estimate (\ref{Eq (Section 3) Caccioppoli estimate}) is easy to deduce by choosing $\zeta\coloneqq \eta^{2}(U_{\delta,\,\varepsilon}-k)_{+}\in W_{0}^{1,\,2}(B_{\rho}(x_{0}))$ into (\ref{Eq (Section 3) U-delta-epsilon is subsol}), and making standard absorbing arguments (see \cite[Lemma 3.15]{T-scalar} for detailed computations). We should mention that this test function $\zeta$ is admissible by approximation, since it is compactly supported in $B_{\rho}(x_{0})$.
\end{proof}
The Caccioppoli-type inequality (\ref{Eq (Section 3) Caccioppoli estimate}) implies that the scalar function $U_{\delta,\,\varepsilon}$ is in a certain De Giorgi class. From this fact, we can conclude two oscillation lemmata for the scalar function $U_{\delta,\,\varepsilon}=\lvert {\mathcal G}_{\delta,\,\varepsilon}(\nabla u_{\varepsilon})\rvert^{2}$ (Lemmata \ref{Lemma: Oscillation lemma for U-epsilon-delta 1}--\ref{Lemma: Oscillation lemma for U-epsilon-delta 2}).
\begin{lemma}\label{Lemma: Oscillation lemma for U-epsilon-delta 1}
Let $u_{\varepsilon}$ be a weak solution to (\ref{Eq (Section 2) Approximated system}). Assume that positive numbers $\delta$, $\mu$, $F$, $M$, and an open ball $B_{\rho}(x_{0})$ satisfy (\ref{Eq (Section 2) control of f}) and (\ref{Eq (Section 2) G-delta-epsilon bound}). Then there exists a number ${\hat \nu}\in(0,\,1)$, depending at most on $b$, $n$, $N$, $p$, $q$, $\gamma$, $\Gamma$, $F$, $M$, and $\delta$, such that if there holds
\begin{equation}\label{Eq (Section 3) measure condition on De Giorgi lemma}
\frac{\lvert \{x\in B_{\rho/2}(x_{0})\mid U_{\delta,\,\varepsilon}(x)>(1-\theta)\mu^{2}\}\rvert}{\lvert B_{\rho/2}(x_{0})\rvert}\le {\hat \nu}
\end{equation}
for some $\theta\in(0,\,1)$, then we have either
\[\mu^{2}<\frac{\rho^{\beta}}{\theta}\quad \textrm{or}\quad \esssup_{B_{\rho/4}(x_{0})}\,U_{\delta,\,\varepsilon}\le \mleft(1-\frac{\theta}{2}\mright)\mu^{2}.\]
\end{lemma}
\begin{lemma}\label{Lemma: Oscillation lemma for U-epsilon-delta 2}
Under the assumptions of Proposition \ref{Prop: De Giorgi's truncation}, for every $i_{\star}\in {\mathbb N}$, we have either 
\[\mu^{2}<\frac{2^{i_{\star}}\rho^{\beta}}{\nu}\quad \textrm{or}\quad \frac{\mleft\lvert \mleft\{ x\in B_{\rho/2}(x_{0})\mathrel{}\middle|\mathrel{}U_{\delta,\,\varepsilon}(x)> \mleft(1-2^{-i_{\star}}\nu\mright)\mu^{2}\mright\} \mright\rvert}{\lvert B_{\rho/2}(x_{0})\rvert}<\frac{C_{\dagger}}{\nu\sqrt{i_{\star}}}.\]
Here the constant $C_{\dagger}\in(0,\,\infty)$ depends at most on $b$, $n$, $N$, $p$, $q$, $\gamma$, $\Gamma$, $F$, $M$, and $\delta$.
\end{lemma}
We omit the proofs of Lemmata \ref{Lemma: Oscillation lemma for U-epsilon-delta 1}--\ref{Lemma: Oscillation lemma for U-epsilon-delta 2}, because they can be deduced by routine calculations given in \cite[Chapter 10, \S 4--5]{MR2566733}. For detailed computations, see \cite[\S 7]{BDGPdN} or \cite[\S 4]{T-scalar}. By Lemmata \ref{Lemma: Oscillation lemma for U-epsilon-delta 1}--\ref{Lemma: Oscillation lemma for U-epsilon-delta 2}, we are able to find the desired constants $C_{\star}\in(1,\,\infty)$ and $\kappa\in(2^{-\beta},\,1)$ in Proposition \ref{Prop: De Giorgi's truncation}.
To be precise, we choose a sufficiently large number $i_{\star}=i_{\star}(C_{\dagger},\,\nu,\,{\hat\nu})\in{\mathbb N}$ such that
\[\frac{C_{\dagger}}{\nu\sqrt{i_{\star}}}\le {\hat\nu}\quad \textrm{and}\quad 0<2^{-(i_{\star}+1)}<1-2^{-2\beta}.\]
Then, by applying Lemmata \ref{Lemma: Oscillation lemma for U-epsilon-delta 1}--\ref{Lemma: Oscillation lemma for U-epsilon-delta 2} with $\theta=2^{-i_{\star}}\nu$, we can easily check that the constants $C_{\star}\coloneqq \theta^{-1}=2^{i_{\star}}\nu^{-1}\in(1,\,\infty)$, $\kappa\coloneqq \sqrt{1-2^{-(i_{\star}+1)}}\in(2^{-\beta},\,1)$ satisfy Proposition \ref{Prop: De Giorgi's truncation} (see also \cite[Proposition 3.5]{BDGPdN}, \cite[Proposition 3.3]{T-scalar}). 

\section{Campanato-type decay estimates}\label{Section: Campanato estimates}
Section \ref{Section: Campanato estimates} provides the proof of Proposition \ref{Prop: Schauder estimate}. The basic idea is that the measure assumption (\ref{Eq (Section 2) Measure condition 2}) implies that the scalar function $V_{\varepsilon}$ will not degenerate near the point $x_{0}$ if the ratio $\nu$ is sufficiently close to $0$. With this in mind, we appeal to freezing coefficient arguments and shrinking methods to obtain Campanato-type integral growth estimates (\ref{Eq (Section 2) Campanato-type growth from Schauder}). To prove this, however, we also have to check that an average $(Du_{\varepsilon})_{x_{0},\,r}$ never degenerates even when the radius $r$ tends to $0$. To overcome this problem, we will provide a variety of energy estimates from the assumptions (\ref{Eq (Section 2) delta<mu}) and (\ref{Eq (Section 2) Measure condition 2}). Most of our computations concerning energy estimates differ from \cite[\S 4--5]{BDGPdN}, since the density function we discuss is structurally different from theirs.

In Section \ref{Section: Campanato estimates}, we consider an $L^{2}$-mean oscillation given by
\[\Phi(x_{0},\,r)\coloneqq \fint_{B_{r}(x_{0})}\mleft\lvert Du_{\varepsilon}-(Du_{\varepsilon})_{x_{0},\,r}\mright\rvert^{2}\,{\mathrm{d}}x\quad \textrm{for }r\in(0,\,\rho\rbrack\]
in an open ball $B_{\rho}(x_{0})\Subset\Omega$. To estimate $\Phi$, we often use a well-known fact that there holds
\begin{equation}\label{Eq (Section 4) minimizing property on L2-average}
\fint_{B_{r}(x_{0})}\mleft\lvert v-(v)_{x_{0},\,r}\mright\rvert^{2}\,{\mathrm{d}}x=\min_{\xi\in{\mathbb R}^{Nn}}\fint_{B_{r}(x_{0})}\lvert v-\xi\rvert^{2}\,{\mathrm{d}}x
\end{equation}
for every $v\in L^{2}(B_{r}(x_{0});\,{\mathbb R}^{Nn})$. Also, we recall the Poincar\'{e}--Sobolev inequality \cite[Chapter IX, Theorem 10.1]{MR1897317}:
\begin{equation}\label{Eq: Poincare--Sobolev inequality}
\fint_{B_{r}(x_{0})}\lvert v\rvert^{2}\,{\mathrm{d}}x\le C(n)r^{2}\mleft(\fint_{B_{r}(x_{0})}\lvert Dv\rvert^{\frac{2n}{n+2}}\,{\mathrm{d}}x\mright)^{\frac{n+2}{n}}
\end{equation}
for all $v\in W^{1,\,\frac{2n}{n+2}}(B_{r}(x_{0});\,{\mathbb R}^{Nn})$ satisfying 
\[\fint_{B_{r}(x_{0})}v\,{\mathrm{d}}x=0,\]
which is used in Sections \ref{Subsect: Energy estimates}--\ref{Subsect: freezing coefficient method}.
\subsection{Energy estimates}\label{Subsect: Energy estimates}
First in Section \ref{Subsect: Energy estimates}, we go back to the weak formulation (\ref{Eq (Section 3) Resulting weak formulation}), and deduce some energy estimates under suitable assumptions as in Proposition \ref{Prop: Schauder estimate}. Here we would like to show local $L^{2}$-bounds for the Jacobian matrices of $G_{p,\,\varepsilon}(Du_{\varepsilon})=V_{\varepsilon}^{p-1}Du_{\varepsilon}$, where the mapping $G_{p,\,\varepsilon}$ is defined by (\ref{Eq (Section 2) G-p-epsilon}).
\begin{lemma}\label{Lemma: Energy estimates}
Let $u_{\varepsilon}$ be a weak solution to (\ref{Eq (Section 2) Approximated system}) in \(\Omega\). Assume that positive numbers $\delta$, $\mu$, $F$, $M$, and an open ball $B_{\rho}(x_{0})\Subset \Omega$ satisfy (\ref{Eq (Section 2) control of f}), (\ref{Eq (Section 2) esssup V-epsilon})--(\ref{Eq (Section 2) delta<mu}) and $0<\rho<1$. Then, for all $\nu\in(0,\,1)$ and $\tau\in(0,\,1)$, the following estimates (\ref{Eq (Section 4) Energy estimate from weak form 1})--(\ref{Eq (Section 4) Energy estimate from weak form 2}) hold:
\begin{equation}\label{Eq (Section 4) Energy estimate from weak form 1}
\fint_{B_{\tau\rho}(x_{0})}\mleft\lvert D\mleft[G_{p,\,\varepsilon}(Du_{\varepsilon})\mright] \mright\rvert^{2}\,{\mathrm{d}}x\le \frac{C}{\tau^{n}}\mleft[\frac{1}{(1-\tau)^{2}\rho^{2}}+F^{2}\rho^{-\frac{2n}{q}}\mright]\mu^{2p}.
\end{equation}
\begin{equation}\label{Eq (Section 4) Energy estimate from weak form 2}
\frac{1}{\lvert B_{\tau\rho}(x_{0})\rvert}\int_{S_{\tau\rho,\,\mu,\,\nu}(x_{0})}\mleft\lvert D\mleft[G_{p,\,\varepsilon}(Du_{\varepsilon})\mright] \mright\rvert^{2}\,{\mathrm{d}}x\le \frac{C}{\tau^{n}}\mleft[\frac{\nu}{(1-\tau)^{2}\rho^{2}}+\frac{F^{2}\rho^{-\frac{2n}{q}}}{\nu}\mright]\mu^{2p}.
\end{equation}
Here the constants $C\in(0,\,\infty)$ given in (\ref{Eq (Section 4) Energy estimate from weak form 1})--(\ref{Eq (Section 4) Energy estimate from weak form 2}) depend at most on $b$, $n$, $p$, $q$, $\gamma$, $\Gamma$, and $\delta$.
\end{lemma}
Following computations given in \cite[Lemma 3.9]{T-scalar} for the scalar problems, we provide the proof of Lemma \ref{Lemma: Energy estimates}.
\begin{proof}
We first compute
\begin{align*}
\mleft\lvert D\mleft[G_{p,\,\varepsilon}(Du_{\varepsilon})\mright] \mright\rvert^{2}&\le \mleft\lvert h_{p}^{\prime}(V_{\varepsilon}^{2})D^{2}u_{\varepsilon}+2h_{p}^{\prime\prime}(V_{\varepsilon}^{2})Du_{\varepsilon}\otimes D^{2}u_{\varepsilon}Du_{\varepsilon} \mright\rvert^{2}\\&\le 2V_{\varepsilon}^{2(p-1)}\mleft\lvert D^{2}u_{\varepsilon}\mright\rvert^{2}+2(p-1)^{2}V_{\varepsilon}^{2(p-3)}\mleft\lvert D^{2}u_{\varepsilon}\mright\rvert^{2}\mleft\lvert Du_{\varepsilon}\mright\rvert^{4}\\&\le c_{p}V_{\varepsilon}^{2(p-1)}\mleft\lvert D^{2}u_{\varepsilon}\mright\rvert^{2}\quad \textrm{a.e. in }\Omega,
\end{align*}
where $c_{p}\coloneqq 2(p^{2}-2p+2)>0$.
With this in mind, we apply Lemma \ref{Lemma: Weak formulation for V-epsilon} with $\psi(t)\coloneqq t^{p}{\tilde \psi}(t)$ for $t\in\lbrack0,\,\infty)$. Here the function ${\tilde\psi}$ will later be chosen as either
\begin{equation}\label{Eq (Section 4) Energy estimate choice 1}
{\tilde\psi}(t)\equiv 1,
\end{equation}
or
\begin{equation}\label{Eq (Section 4) Energy estimate choice 2}
{\tilde\psi}(t)\coloneqq (t-\delta-k)_{+}^{2}
\end{equation}
for some constant $k>0$. 
We choose a cutoff function \(\eta\in C_{c}^{1}(B_{\rho}(x_{0}))\) such that
\[\eta\equiv 1\quad\textrm{on $B_{\tau \rho}(x_{0})$}\quad\textrm{and}\quad\lvert\nabla\eta\rvert\le \frac{2}{(1-\tau)\rho}\quad\textrm{in $B_{\rho}(x_{0})$},\]
and set \(\zeta\coloneqq \eta^{2}\).
Then, the weak formulation (\ref{Eq (Section 3) Resulting weak formulation}) becomes
\begin{align*}
&\gamma{\mathbf L}_{1}+{\mathbf L}_{2}\\
&
\coloneqq \gamma\int_{B}V_{\varepsilon}^{2(p-1)}\mleft\lvert D^{2}u_{\varepsilon}\mright\rvert^{2} {\tilde\psi}(V_{\varepsilon})\eta^{2}\,{\mathrm d}x\\&\quad
+
\int_{B}\mleft\langle{\mathcal C}_{\varepsilon}(D_{\varepsilon})\nabla V_{\varepsilon}\mathrel{}\middle|\mathrel{}\nabla V_{\varepsilon}\mright\rangle \psi^{\prime}(V_{\varepsilon})V_{\varepsilon}\eta^{2}\,{\mathrm d}x
\\&\le 4\int_{B}\mleft\lvert \mleft\langle{\mathcal C}_{\varepsilon}(D_{\varepsilon})\nabla V_{\varepsilon}\mathrel{}\middle|\mathrel{}\nabla\eta\mright\rangle\mright\rvert \psi(V_{\varepsilon})V_{\varepsilon}\eta\,{\mathrm d}x\\&\quad+
\frac{1}{\gamma}\int_{B}\lvert f_{\varepsilon}\rvert^{2}V_{\varepsilon}^{2-p}\mleft(n{\psi}(V_{\varepsilon})+V_{\varepsilon}\psi^{\prime}(V_{\varepsilon})\mright)\eta^{2}\,{\mathrm d}x
\\& \quad\quad+4\int_{B}\lvert f_{\varepsilon}\rvert\lvert \nabla \eta\rvert V_{\varepsilon}\psi(V_{\varepsilon})\eta\,{\mathrm d}x
\\&\eqqcolon {\mathbf R}_{1}+{\mathbf R}_{2}+{\mathbf R}_{3}.
\end{align*}
For \({\mathbf R}_{1}\), we use the Cauchy--Schwarz inequality
\[\mleft\lvert \mleft\langle {\mathcal C}_{\varepsilon}(D_{\varepsilon})\nabla V_{\varepsilon}\mathrel{}\middle|\mathrel{}\nabla\eta\mright\rangle\mright\rvert\le \sqrt{\mleft\langle{\mathcal C}_{\varepsilon}(D_{\varepsilon})\nabla V_{\varepsilon}\mathrel{}\middle|\mathrel{}\nabla V_{\varepsilon}\mright\rangle}\,\cdot\,\sqrt{\mleft\langle{\mathcal C}_{\varepsilon}(D_{\varepsilon})\nabla \eta\mathrel{}\middle|\mathrel{}\nabla\eta\mright\rangle},\]
which is possible by (\ref{Eq (Section 3) Ellipticity of coefficients}).
Hence, by Young's inequality, we have
\begin{align*}
{\mathbf R}_{1}&\le {\mathbf L}_{2}+4\int_{B}\mleft\langle {\mathcal C}_{\varepsilon}(Du_{\varepsilon})\nabla\eta\mathrel{}\middle|\mathrel{}\nabla\eta\mright\rangle \frac{V_{\varepsilon}\psi(V_{\varepsilon})^{2}}{\psi^{\prime}(V_{\varepsilon})}\,\mathrm{d}x.
\end{align*}
For \({\mathbf R}_{3}\), we apply Young's inequality to get
\[{\mathbf R}_{3}\le 2\int_{B}\lvert f_{\varepsilon}\rvert^{2}\psi^{\prime}(V_{\varepsilon}) V_{\varepsilon}^{3-p}\eta^{2}\,{\mathrm d}x+2 \int_{B}\lvert \nabla\eta\rvert^{2}\frac{V_{\varepsilon}\psi(V_{\varepsilon})^{2}}{\psi^{\prime}(V_{\varepsilon})}V_{\varepsilon}^{p-2}\,{\mathrm d}x.\]
Therefore, by (\ref{Eq (Section 3) Ellipticity of coefficients}), we obtain
\begin{align*}
{\mathbf L}_{1}&\le\frac{1}{\gamma}\int_{B}\lvert\nabla\eta\rvert^{2} \mleft[2V_{\varepsilon}^{p-2}+4\mleft(bV_{\varepsilon}^{-1}+3\Gamma V_{\varepsilon}^{p-2}\mright)\mright]\frac{V_{\varepsilon}\psi(V_{\varepsilon})^{2}}{\psi^{\prime}(V_{\varepsilon})}\,{\mathrm d}x\\&\quad +\frac{1}{\gamma}\mleft(2+\frac{1}{\gamma}\mright)\int_{B}\lvert f_{\varepsilon}\rvert^{2}\eta^{2}\frac{n\psi(V_{\varepsilon})+V_{\varepsilon}\psi^{\prime}(V_{\varepsilon})}{V_{\varepsilon}^{p-2}}\,{\mathrm d}x.
\end{align*}
Here we note that the assumptions (\ref{Eq (Section 2) esssup V-epsilon})--(\ref{Eq (Section 2) delta<mu}) yield \(V_{\varepsilon}\le \delta+\mu\le 2\mu\) a.e. in \(B\), and \(\mu^{l}=\mu^{l-2p}\cdot\mu^{2p}\le\delta^{l-2p}\mu^{2p}\) for every \(l\in(0,\,2p)\).
Hence it follows that
\begin{align*}
\mleft[2V_{\varepsilon}^{p-2}+4\mleft(bV_{\varepsilon}^{-1}+3\Gamma V_{\varepsilon}^{p-2}\mright)\mright]\frac{V_{\varepsilon}\psi(V_{\varepsilon})^{2}}{\psi^{\prime}(V_{\varepsilon})}&\le C(b,\,\Gamma) \mleft[V_{\varepsilon}^{p-2}+V_{\varepsilon}^{-1}\mright]\frac{V_{\varepsilon}^{p+2}{\tilde\psi}(V_{\varepsilon})^{2}}{p{\tilde \psi}(V_{\varepsilon})+V_{\varepsilon}{\tilde\psi}^{\prime}(V_{\varepsilon})}\\&\le C(b,\,\Gamma)\frac{\mleft(1+\delta^{1-p}\mright)\mu^{2p}{\tilde\psi}(V_{\varepsilon})^{2}}{p{\tilde \psi}(V_{\varepsilon})+V_{\varepsilon}{\tilde\psi}^{\prime}(V_{\varepsilon})}
\end{align*}
a.e. in $B$, and
\begin{align*}
\frac{n\psi(V_{\varepsilon})+V_{\varepsilon}\psi^{\prime}(V_{\varepsilon})}{V_{\varepsilon}^{p-2}}&=n(p+1)V_{\varepsilon}^{2}{\tilde\psi}(V_{\varepsilon})+V_{\varepsilon}^{3}{\tilde\psi}^{\prime}(V_{\varepsilon})\\&\le 4n(p+1)\delta^{2(1-p)}\mu^{2p}\mleft[{\tilde \psi}(V_{\varepsilon})+V_{\varepsilon}{\tilde\psi}^{\prime}(V_{\varepsilon})\mright] \end{align*}
a.e. in $B$.
As a result, we obtain
\begin{align}
&\int_{B}\mleft\lvert D\mleft[G_{p,\,\varepsilon}(\nabla u_{\varepsilon})\mright]\mright\rvert^{2}\eta^{2}{\tilde \psi}(V_{\varepsilon})\,{\mathrm d}x \le c_{p}{\mathbf L}_{1}\nonumber\\&\le C\mu^{2p}\mleft[\int_{B}\lvert\nabla\eta\rvert^{2}\frac{{\tilde\psi}(V_{\varepsilon})^{2}}{p{\tilde \psi}(V_{\varepsilon})+V_{\varepsilon}{\tilde\psi}^{\prime}(V_{\varepsilon})}\,{\mathrm d}x+\int_{B}\lvert f_{\varepsilon}\rvert^{2}\eta^{2} \mleft[{\tilde \psi}(V_{\varepsilon})+V_{\varepsilon}{\tilde\psi}^{\prime}(V_{\varepsilon})\mright]\,{\mathrm d}x\,\mright]\label{Eq (Section 4) Energy estimate mid fin}
\end{align}
with \(C=C(b,\,n,\,p,\,\gamma,\,\Gamma,\,\delta)\in(0,\,\infty)\). We will deduce (\ref{Eq (Section 4) Energy estimate from weak form 1})--(\ref{Eq (Section 4) Energy estimate from weak form 2}) from (\ref{Eq (Section 4) Energy estimate mid fin}) by choosing ${\tilde \psi}$ as (\ref{Eq (Section 4) Energy estimate choice 1}) or (\ref{Eq (Section 4) Energy estimate choice 2}).

Let ${\tilde \psi}$ satisfy (\ref{Eq (Section 4) Energy estimate choice 1}). Then by (\ref{Eq (Section 4) Energy estimate mid fin}) and H\"{o}lder's inequality, we have
\begin{align*}
&\int_{B_{\tau\rho}(x_{0})}\mleft\lvert D\mleft[G_{p,\,\varepsilon}(\nabla u_{\varepsilon})\mright]\mright\rvert^{2}\,{\mathrm d}x\\&\le  C(b,\,n,\,p,\,\gamma,\,\Gamma,\,\delta)\mu^{2p}\mleft[\frac{1}{p}\int_{B}\lvert \nabla\eta\rvert^{2}\,{\mathrm d}x+\int_{B}\lvert f_{\varepsilon}\rvert^{2}\eta^{2}{\mathrm d}x\mright]\\&\le C(b,\,n,\,p,\,q,\,\gamma,\,\Gamma,\,\delta)\mu^{2p}\mleft[\frac{1}{(1-\tau)^{2}\rho^{2}}+F^{2}\rho^{-\frac{2n}{q}}\mright]\lvert B_{\rho}(x_{0})\rvert.
\end{align*}
Next, we let \({\tilde\psi}\) satisfy (\ref{Eq (Section 4) Energy estimate choice 2}) with \[k\coloneqq (1-2\nu)\mu\ge\frac{\mu}{2}>0.\]
Then, from (\ref{Eq (Section 2) Lipschitz bound}), it follows that \((V_{\varepsilon}-\delta-k)_{+}\le \mu-k=2\nu\mu\) a.e. in \(B\). 
Hence, it is easy to check that
\[\mleft\{
\begin{array}{ccccccc}
\displaystyle\frac{{\tilde\psi}(V_{\varepsilon})^{2}}{p{\tilde \psi}(V_{\varepsilon})+V_{\varepsilon}{\tilde\psi}^{\prime}(V_{\varepsilon})}&=&\displaystyle\frac{(V_{\varepsilon}-\delta-k)_{+}^{3}}{p(V_{\varepsilon}-\delta-k)_{+}+2V_{\varepsilon}}&\le& \displaystyle\frac{(2\nu\mu)^{3}}{0+2(\delta+k)}&\le& 8\nu^{3}\mu^{2},\\
{\tilde\psi}(V_{\varepsilon})+V_{\varepsilon}{\tilde\psi}^{\prime}(V_{\varepsilon})&\le& 3V_{\varepsilon}(V_{\varepsilon}-\delta-k)_{+}&\le& 3\cdot 2\mu\cdot 2\nu\mu&=&12\nu\mu^{2},
\end{array}
\mright.\]
a.e. in $B$.
Over \(S_{\tau\rho,\,\mu,\,\nu}(x_{0})\), we have \(V_{\varepsilon}-\delta-k>(1-\nu)\mu-k=\nu\mu>0\), and hence ${\tilde\psi}(V_{\varepsilon})\ge (\nu\mu)^{2}$.
Combining these results with (\ref{Eq (Section 4) Energy estimate mid fin}), we are able to compute
\begin{align*}
&\nu^{2}\mu^{2}\int_{S_{\tau\rho,\,\mu,\,\nu}(x_{0})}\mleft\lvert D\mleft[G_{p,\,\varepsilon}(\nabla u_{\varepsilon})\mright]\mright\rvert^{2}\,{\mathrm d}x \\&\le \int_{B}\eta^{2}\mleft\lvert D\mleft[G_{p,\,\varepsilon}(\nabla u_{\varepsilon})\mright]\mright\rvert^{2}{\tilde\psi}(V_{\varepsilon})\,{\mathrm d}x\\&\le C\mu^{2p}\mleft[8\nu^{3}\mu^{2}\int_{B}\lvert\nabla\eta\rvert^{2}\,{\mathrm d}x+12\nu\mu^{2} \int_{B}\lvert f_{\varepsilon}\rvert^{2}\eta^{2}\,{\mathrm d}x\mright]\\&\le C\nu^{2}\mu^{2p+2}\mleft[\frac{\nu}{(1-\tau)^{2}\rho^{2}}+\frac{F^{2}\rho^{-\frac{2n}{q}}}{\nu} \mright]\lvert B_{\rho}(x_{0})\rvert
\end{align*}
for some constant \(C=C(b,\,n,\,p,\,q,\,\gamma,\,\Gamma,\,\delta)\in(0,\,\infty)\).
Recalling \(\lvert B_{\tau\rho}(x_{0})\rvert=\tau^{n}\lvert B_{\rho}(x_{0})\rvert\), we conclude (\ref{Eq (Section 4) Energy estimate from weak form 1})--(\ref{Eq (Section 4) Energy estimate from weak form 2}).
\end{proof}
\begin{lemma}\label{Lemma: Energy estimates 2}
Let $u_{\varepsilon}$ be a weak solution to (\ref{Eq (Section 2) Approximated system}) in \(\Omega\). Assume that positive numbers $\delta$, $\varepsilon$, $\mu$, $F$, $M$, and an open ball $B_{\rho}(x_{0})\Subset \Omega$ satisfy (\ref{Eq (Section 2) delta-epsilon}), (\ref{Eq (Section 2) control of f}), (\ref{Eq (Section 2) esssup V-epsilon})--(\ref{Eq (Section 2) delta<mu}) and $0<\rho<1$. If (\ref{Eq (Section 2) Measure condition 2}) is satisfied for some $\nu\in(0,\,1/4)$, then for all $\tau\in(0,\,1)$, we have
\begin{equation}\label{Eq (Section 4) Energy estimates from weak form 3}
\Phi(x_{0},\,\tau\rho)\le \frac{C_{\dagger}\mu^{2}}{\tau^{n}}\mleft[\frac{\nu^{2/n}}{(1-\tau)^{2}}+\frac{F^{2}}{\nu}\rho^{2\beta}\mright],
\end{equation}
where the constant $C_{\dagger}\in(0,\,\infty)$ depends at most on $b$, $n$, $p$, $q$, $\gamma$, $\Gamma$, and $\delta$.
\end{lemma}
For completeness, we provide the proof of Lemma \ref{Lemma: Energy estimates 2}, which can be accomplished similarly to \cite[Lemma 3.10]{T-scalar} for the scalar case.
\begin{proof}
To prove (\ref{Eq (Section 4) Energy estimates from weak form 3}), we consider integrals ${\mathbf I},\,{\mathbf{II}}$ given by
\[{\mathbf I}\coloneqq \frac{1}{\lvert B_{\tau\rho}(x_{0})\rvert}\int_{B_{\tau\rho}(x_{0})\setminus S_{\rho,\,\mu,\,\nu}(x_{0})}\lvert Du_{\varepsilon}-\xi\rvert^{2}\,{\mathrm{d}}x,\]
and
\[{\mathbf{II}}\coloneqq \frac{1}{\lvert B_{\tau\rho}(x_{0})\rvert}\int_{B_{\tau\rho}(x_{0})\cap S_{\rho,\,\mu,\,\nu}(x_{0})}\lvert Du_{\varepsilon}-\xi\rvert^{2}\,{\mathrm{d}}x,\]
where $\xi\coloneqq G_{p,\,\varepsilon}^{-1}\mleft((G_{p,\,\varepsilon}(Du_{\varepsilon}))_{x_{0},\,\tau\rho}\mright)\in{\mathbb R}^{Nn}$. 
For ${\mathbf I}$, we use (\ref{Eq (Section 2) Measure condition 2}) to get
\[{\mathbf I}\le \frac{\lvert B_{\rho}(x_{0})\setminus S_{\rho,\,\mu,\,\nu}(x_{0}) \rvert}{\tau^{n}\lvert B_{\rho}(x_{0})\rvert}\esssup_{B_{\rho}(x_{0})}\,\mleft(V_{\varepsilon}+\lvert \xi\rvert\mright)^{2}\le \frac{C(p)\nu\mu^{2}}{\tau^{n}}.\]
The last inequality is easy to check by (\ref{Eq (Section 2) esssup V-epsilon})--(\ref{Eq (Section 2) delta<mu}). In fact, we have
\[\mleft\lvert G_{p,\,\varepsilon}(Du_{\varepsilon})\mright\rvert=V_{\varepsilon}^{p-1}\lvert Du_{\varepsilon}\rvert\le V_{\varepsilon}^{p}\le (\delta+\mu)^{p}\le (2\mu)^{p}\quad \textrm{a.e. in }B_{\rho}(x_{0}).\]
Then, recalling (\ref{Eq (Section 2) Estimate on inverse mapping}), we obtain
\[\lvert \xi\rvert\le C(p)\mleft\lvert G_{p,\,\varepsilon}(\xi)\mright\rvert^{1/p}=C(p)\mleft\lvert(G_{p,\,\varepsilon}(Du_{\varepsilon}))_{x_{0},\,\tau\rho}\mright\rvert^{1/p}\le C(p)\esssup_{B_{\rho}(x_{0})}\,\mleft\lvert G_{p,\,\varepsilon}(Du_{\varepsilon})\mright\rvert^{1/p}\le C(p)\mu,\]
and hence it is possible to estimate ${\mathbf I}$ as above.
Before computing $\mathbf{II}$, we recall (\ref{Eq (Section 2) delta-epsilon}) and the definition of $S_{\rho,\,\mu,\,\nu}(x_{0})$. Then, we are able to compute
\begin{equation}\label{Eq (Section 4) Non-degenerate-gradient over superlevel set}
\frac{\delta}{4}+\lvert Du_{\varepsilon}\rvert\ge \varepsilon+\lvert Du_{\varepsilon}\rvert\ge V_{\varepsilon}\ge \delta+(1-\nu)\mu\quad \textrm{a.e. in }S_{\rho,\,\mu,\,\nu}(x_{0}).
\end{equation}
In particular, there holds \[\lvert Du_{\varepsilon}\rvert\ge \frac{3}{4}\mu\quad \textrm{a.e. in }S_{\rho,\,\mu,\,\nu}(x_{0})\]
by $0<\nu<1/4$.
With this in mind, we apply (\ref{Eq (Section 2) G-p-epsilon local ellipticity}) to obtain
\begin{align*}
\mathbf{II}&\le \frac{C(p)}{\mu^{2(p-1)}\lvert B_{\tau\rho}(x_{0})\rvert}\int_{B_{\tau\rho}(x_{0})\cap S_{\rho,\,\mu,\,\nu}(x_{0})}\mleft\lvert G_{p,\,\varepsilon}(Du_{\varepsilon})-G_{p,\,\varepsilon}(\xi)\mright\rvert^{2}\,{\mathrm d}x\\&\le \frac{C(p)}{\mu^{2(p-1)}}(\tau\rho)^{2}\mleft(\fint_{B_{\tau\rho}(x_{0})}\mleft\lvert D\mleft[G_{p,\,\varepsilon}(Du_{\varepsilon})\mright]  \mright\rvert^{\frac{2n}{n+2}} \,{\mathrm d}x\mright)^{\frac{n+2}{n}}\eqqcolon \frac{C(p)}{\mu^{2(p-1)}}(\tau\rho)^{2}\cdot \mathbf{III}^{1+2/n}.
\end{align*}
Here we have applied (\ref{Eq: Poincare--Sobolev inequality}) to the function $G_{p,\,\varepsilon}(Du_{\varepsilon})-G_{p,\,\varepsilon}(\xi)$. The integral $\mathbf{III}$ can be decomposed by ${\mathbf{III}}={\mathbf{III}}_{1}+{\mathbf{III}}_{2}$ with
\[{\mathbf{III}}_{1}\coloneqq \frac{1}{\lvert B_{\tau\rho}(x_{0})\rvert}\int_{B_{\tau\rho}(x_{0})\setminus S_{\rho,\,\mu,\,\nu}(x_{0})}\mleft\lvert D\mleft[G_{p,\,\varepsilon}(Du_{\varepsilon})\mright] \mright\rvert^{\frac{2n}{n+2}}\,{\mathrm d}x,\]
and
\[{\mathbf{III}}_{2}\coloneqq \frac{1}{\lvert B_{\tau\rho}(x_{0})\rvert}\int_{S_{\tau\rho,\,\mu,\,\nu}(x_{0})}\mleft\lvert D\mleft[G_{p,\,\varepsilon}(Du_{\varepsilon})\mright] \mright\rvert^{\frac{2n}{n+2}}\,{\mathrm d}x.\]
To control these integrals, we apply Lemma \ref{Lemma: Energy estimates}. For ${\mathbf{III}}_{1}$, we use H\"{o}lder's inequality and (\ref{Eq (Section 4) Energy estimate from weak form 1}) to obtain\begin{align*}
\mathbf{III}_{1}^{1+2/n}&\le \mleft[\frac{\lvert B_{\rho}(x_{0})\setminus S_{\rho,\,\mu,\,\nu}(x_{0})\rvert }{\lvert B_{\rho}(x_{0})\rvert }\mright]^{2/n} \fint_{B_{\tau\rho}(x_{0})}\mleft\lvert D\mleft[G_{p,\,\varepsilon}(Du_{\varepsilon})\mright] \mright\rvert^{2}\,{\mathrm d}x\\&\le C(b,\,n,\,p,\,q,\,\gamma,\,\Gamma,\,\delta)\frac{\nu^{2/n}\mu^{2p}}{\tau^{n}}\mleft[\frac{1}{(1-\tau)^{2}\rho^{2}}+F^{2}\rho^{-\frac{2n}{q}}\mright].
\end{align*}
where we should note $\lvert B_{\rho}(x_{0})\setminus S_{\rho,\,\mu,\,\nu}(x_{0})\rvert\le \nu\lvert B_{\rho}(x_{0})\rvert$ by (\ref{Eq (Section 2) Measure condition 2}).
Similarly for ${\mathbf{III}}_{2}$, we can compute
\begin{align*}
\mathbf{III}_{2}^{1+2/n}&\le \mleft[\frac{\lvert S_{\tau\rho,\,\mu,\,\nu}(x_{0})\rvert}{\lvert B_{\tau\rho}(x_{0})\rvert}\mright]^{2/n}\frac{1}{\lvert B_{\tau\rho}(x_{0})\rvert} \int_{S_{\tau\rho,\,\mu,\,\nu}(x_{0})}\mleft\lvert D\mleft[G_{p,\,\varepsilon}(Du_{\varepsilon})\mright] \mright\rvert^{2}\,{\mathrm d}x\\&\le C(b,\,n,\,p,\,q,\,\gamma,\,\Gamma,\,\delta)\frac{\mu^{2p}}{\tau^{n}}\mleft[\frac{\nu}{(1-\tau)^{2}\rho^{2}}+\frac{F^{2}\rho^{-\frac{2n}{q}}}{\nu}\mright].
\end{align*}
by H\"{o}lder's inequality and (\ref{Eq (Section 4) Energy estimate from weak form 2}).
Finally, we use (\ref{Eq (Section 4) minimizing property on L2-average}) to obtain
\begin{align*}
\Phi(x_{0},\,\tau\rho)&\le \fint_{B_{\tau\rho}(x_{0})}\lvert Du_{\varepsilon}-\xi\rvert^{2}\,{\mathrm{d}}x={\mathbf I}+{\mathbf{II}}\\&\le \frac{C(p)\nu\mu^{2}}{\tau^{n}}+\frac{C(n,\,p)}{\mu^{2p-2}}(\tau\rho)^{2}\mleft({\mathbf{III}}_{1}^{1+2/n}+{\mathbf{III}}_{2}^{1+2/n} \mright)\\&\le \frac{C\mu^{2}}{\tau^{n}}\mleft[\nu+\frac{\nu+\nu^{2/n}}{(1-\tau)^{2}}\tau^{2}+\frac{1+\nu^{1+2/n}}{\nu}\tau^{2}F^{2}\rho^{2\beta}\mright]
\end{align*}
with $C\in(0,\,\infty)$ depending at most on $b$, $n$, $p$, $q$, $\gamma$, $\Gamma$, and $\delta$. By $0<\tau<1$, $0<\nu<1/4$ and $n\ge 2$, we are able to find a constant $C_{\dagger}=C_{\dagger}(b,\,n,\,p,\,q,\,\gamma,\,\Gamma,\,\delta)\in(0,\,\infty)$ such that (\ref{Eq (Section 4) Energy estimates from weak form 3}) holds.
\end{proof}

\subsection{Higher integrability estimates and freezing coefficient arguments}\label{Subsect: freezing coefficient method}
In Section \ref{Subsect: freezing coefficient method}, we appeal to freezing coefficient methods when an average $(Du_{\varepsilon})_{x_{0},\,\rho}$ does not vanish. We introduce a harmonic mapping $v_{\varepsilon}$ near $x_{0}$, and obtain an error estimate for a comparison function $u_{\varepsilon}-v_{\varepsilon}$. 

When computing errors from a comparison function, we have to deduce higher integrability estimates on $\lvert Du_{\varepsilon}-(Du_{\varepsilon})_{x_{0},\,\rho}\rvert$, which can be justified by applying so called Gehring's lemma. 
\begin{lemma}[Gehring's Lemma]\label{Lemma: Gehring lemma}
Let \(B=B_{R}(x_{0})\subset{\mathbb R}^{n}\) be an open ball, and non-negative function \(g,\,h\) satisfy \(g\in L^{s}(B),\,h\in L^{{\tilde s}}(B)\) with \(1<s<{\tilde s}\le \infty\). Suppose that there holds
\[\fint_{B_{r}(z_{0})}g^{s}{\mathrm d}x\le {\hat C}\mleft[\mleft(\fint_{B_{2r}(z_{0})}g\,{\mathrm d}x \mright)^{s}+\fint_{B_{2r}(z_{0})}h^{s}\,{\mathrm d}x\mright]\]
for all \(B_{2r}(z_{0})\subset B\). Here \({\hat C}\in(0,\,\infty)\) is a constant independent of \(z_{0}\in B\) and \(r>0\). Then there exists a sufficiently small positive number \(\varsigma=\varsigma(b,\,s,\,{\tilde s},\,n)\) such that \(g\in L_{\mathrm{loc}}^{\sigma_{0}}(B)\) with \(\sigma_{0}\coloneqq s(1+\varsigma)\in(s,\,{\tilde s})\). Moreover, for each \(\sigma\in(s,\,\sigma_{0})\), we have
\[\mleft(\fint_{B_{R/2}(x_{0})}g^{\sigma}{\mathrm d}x\mright)^{1/\sigma}\le C\mleft[\mleft(\fint_{B_{R}(x_{0})}g^{s}\,{\mathrm d}x\mright)^{1/s}+\mleft(\fint_{B_{R}(x_{0})}h^{\sigma}\,{\mathrm d}x\mright)^{1/\sigma}\mright],\]
where the constant $C\in(0,\,\infty)$ depends at most on $\sigma,\,n,\,s,\,{\tilde s}$ and ${\hat C}$.
\end{lemma}
The proof of Gehring's lemma is found in \cite[Theorem 3.3]{MR2173373}, which is based on ball decompositions \cite[Lemma 3.1]{MR2173373} and generally works for a metric space with a doubling measure (see also \cite[\S 6.4]{MR1962933}).

By applying Lemma \ref{Lemma: Gehring lemma}, we prove Lemma \ref{Lemma: Higher integrability}. 
\begin{lemma}[Higher integrability lemma]\label{Lemma: Higher integrability}
Let $u_{\varepsilon}$ be a weak solution to (\ref{Eq (Section 2) Approximated system}). Assume that positive numbers $\delta,\,\varepsilon,\,\mu,\,M,\,F$ and an open ball $B_{\rho}(x_{0})\Subset\Omega$ satisfy (\ref{Eq (Section 2) delta-epsilon}), (\ref{Eq (Section 2) control of f}) and (\ref{Eq (Section 2) esssup V-epsilon})--(\ref{Eq (Section 2) delta<mu}).
Fix a matrix $\xi\in{\mathbb R}^{Nn}$ with
\begin{equation}\label{Eq (Section 4) xi}
\delta+\frac{\mu}{4}\le \lvert\xi\rvert\le \delta+\mu.
\end{equation}
Then, there exists a constant \(\vartheta=\vartheta(b,\,n,\,N,\,p,\,q,\,\gamma,\,\Gamma,\,M,\,\delta)\) such that
\begin{equation}\label{Eq (Section 4) integrability up}
0<\vartheta\le \min\mleft\{\,\frac{1}{2},\,\beta,\,\beta_{0}\,\mright\}<1,
\end{equation}
and 
\begin{equation}\label{Eq (Section 4) Higher integrability result}
\fint_{B_{\rho/2}(x_{0})}\lvert Du_{\varepsilon}-\xi\rvert^{2(1+\vartheta)}\,{\mathrm d}x\le C\mleft[\mleft(\fint_{B_{\rho}(x_{0})}\lvert Du_{\varepsilon}-\xi\rvert^{2}\,{\mathrm d}x \mright)^{1+\vartheta}+F^{2(1+\vartheta)}\rho^{2\beta(1+\vartheta)}\mright].
\end{equation}
Here the constant \(C\in(0,\,\infty)\) depends at most on \(b,\,n,\,N,\,p,\,q,\,\gamma,\,\Gamma,\,M,\,\delta\).
\end{lemma}
\begin{proof}
It suffices to prove  
\begin{equation}\label{Eq (Section 4) Claim for Gehring's lemma}
\fint_{B_{r/2}(z_{0})}\lvert Du_{\varepsilon}-\xi\rvert^{2}\,{\mathrm d}x\le {\hat C}\mleft[\mleft(\fint_{B_{r}(z_{0})}\lvert Du_{\varepsilon}-\xi\rvert^{\frac{2n}{n+2}}\,{\mathrm d}x\mright)^{\frac{n+2}{n}}+\fint_{B_{r}(z_{0})}\lvert \rho f_{\varepsilon}\rvert^{2}\,\mathrm{d}x\mright] 
\end{equation}
for any \(B_{r}(z_{0})\subset B\coloneqq B_{\rho}(x_{0})\), where \({\hat C}={\hat C}(b,\,n,\,p,\,\gamma,\,\Gamma,\,M,\,\delta)\in(0,\,\infty)\) is a constant.
In fact, this result enables us to apply Lemma \ref{Lemma: Gehring lemma} with \((s,\,{\tilde s})\coloneqq (1+2/n,\,q(n+2)/2n)\), \(g\coloneqq \lvert Du_{\varepsilon}-\xi\rvert^{\frac{2n}{n+2}}\in L^{s}(B),\,h\coloneqq \lvert \rho f_{\varepsilon}\rvert^{\frac{2n}{n+2}}\in L^{\tilde s}(B)\), we are able to find a small exponent \(\vartheta=\vartheta(n,\,q,\,C_{\ast})>0\) and a constant \(C=C(n,\,q,\,\vartheta,\,C_{\ast})\in(0,\,\infty)\) such that there hold (\ref{Eq (Section 4) integrability up}) and \[\fint_{B_{\rho/2}(x_{0})}\lvert Du_{\varepsilon}-\xi\rvert^{2(1+\vartheta)}\,{\mathrm d}x\le C\mleft[\mleft(\fint_{B_{\rho}(x_{0})}\lvert Du_{\varepsilon}-\xi\rvert^{2}\,{\mathrm d}x \mright)^{1+\vartheta}+\fint_{B_{\rho}(x_{0})}\lvert \rho f_{\varepsilon}\rvert^{2(1+\vartheta)}\,{\mathrm d}x \mright].\]
We note that $\vartheta\le \beta$ from (\ref{Eq (Section 4) integrability up}) yields $2(1+\vartheta)\le q$, and therefore $\lvert\rho f_{\varepsilon}\rvert^{2(1+\vartheta)}$ is integrable in $B_{\rho}(x_{0})$. Moreover, by H\"{o}lder's inequality, we have
\[\fint_{B_{\rho}(x_{0})}\lvert\rho f_{\varepsilon}\rvert^{2(1+\vartheta)}\,{\mathrm d}x\le C(n,\,\vartheta)F^{2(1+\vartheta)}\rho^{2\beta(1+\vartheta)},\]
from which (\ref{Eq (Section 4) Higher integrability result}) easily follows.

To prove (\ref{Eq (Section 4) Claim for Gehring's lemma}), for each fixed ball $B_{r}(z_{0})\subset B$, we set a function $w_{\varepsilon}\in W^{1,\,\infty}(B_{r}(z_{0});\,{\mathbb R}^{N})$ by
\[w_{\varepsilon}(x)\coloneqq u_{\varepsilon}(x)-(u_{\varepsilon})_{z_{0},\,r}-\xi\cdot (x-z_{0})\quad \textrm{for }x\in B_{r}(z_{0}).\]
Clearly $Dw_{\varepsilon}=Du_{\varepsilon}-\xi$ holds.
We choose a cutoff function $\eta\in C_{c}(B_{r}(z_{0}))$ satisfying
\[\eta\equiv 1\quad \textrm{on }B_{r/2}(z_{0}),\quad\textrm{and}\quad 0\le\eta\le 1,\quad  \lvert\nabla\eta\rvert\le \frac{4}{r}\quad \textrm{in }B_{r}(z_{0}),\]
and test $\phi\coloneqq \eta^{2}w_{\varepsilon}\in W_{0}^{1,\,1}(B_{\rho}(x_{0});\,{\mathbb R}^{N})$ into (\ref{Eq (Section 3) Weak formulation local}). Then we have
\begin{align*}
0&=\int_{B_{r}(z_{0})}\langle A_{\varepsilon}(Du_{\varepsilon})-A_{\varepsilon}(\xi)\mid D\phi\rangle\,{\mathrm{d}}x-\int_{B_{r}(z_{0})}\langle f_{\varepsilon}\mid \phi\rangle\,{\mathrm{d}}x\\&=\int_{B_{r}(z_{0})}\eta^{2}\langle A_{\varepsilon}(Du_{\varepsilon})-A_{\varepsilon}(\xi)\mid Du_{\varepsilon}-\xi\rangle\,{\mathrm{d}}x\\&\quad +2\int_{B_{r}(z_{0})}\eta\langle A_{\varepsilon}(Du_{\varepsilon})-A_{\varepsilon}(\xi)\mid w_{\varepsilon}\otimes\nabla\eta\rangle\,{\mathrm{d}}x-\int_{B_{r}(z_{0})}\eta^{2}\langle f_{\varepsilon}\mid w_{\varepsilon}\rangle\,{\mathrm{d}}x\\&\eqqcolon {\mathbf J}_{1}+{\mathbf J}_{2}+{\mathbf J}_{3}.
\end{align*}
Here it should be mentioned that since there clearly hold $\delta\le \lvert\xi\rvert\le M$, and $\lvert Du_{\varepsilon}\rvert\le M$ a.e. in $B$, we are able to apply (\ref{Eq (Section 2) Monotonicity outside}) and (\ref{Eq (Section 2) Growth outside}) in Lemma \ref{Lemma: Error estimates} to ${\mathbf J}_{1}$ and ${\mathbf J}_{2}$ respectively. Then, by applying Young's inequality to ${\mathbf J}_{2}$ and ${\mathbf J}_{3}$, and making a standard absorbing argument (see \cite[Lemma 3.7]{T-scalar} for detailed computations), we are able to obtain
\begin{align*}
\fint_{B_{r/2}(z_{0})}\mleft\lvert Dw_{\varepsilon}\mright\rvert^{2}\,{\mathrm d}x&\le 2^{n}\fint_{B_{r}(z_{0})}\mleft\lvert D(\eta w_{\varepsilon})\mright\rvert^{2}\,{\mathrm d}x\\&\le C\mleft[\fint_{B_{r}(z_{0})}\mleft(\lvert\nabla \eta\rvert^{2}+\frac{\eta^{2}}{r^{2}}\mright) \lvert w_{\varepsilon}\rvert^{2}\,{\mathrm d}x+r^{2}\fint_{B_{r}(z_{0})}\eta^{2}\lvert f_{\varepsilon}\rvert^{2}\,{\mathrm d}x \mright]\\&\le C\mleft[r^{-2}\fint_{B_{r}(z_{0})} \lvert w_{\varepsilon}\rvert^{2}\,{\mathrm d}x+\fint_{B_{r}(z_{0})}\lvert \rho f_{\varepsilon}\rvert^{2}\,{\mathrm d}x\mright]\\&\le {\hat C}^{2}\mleft[\mleft(\fint_{B_{r}(z_{0})}\lvert Dw_{\varepsilon}\rvert^{\frac{2n}{n+2}}\,{\mathrm{d}}x\mright)^{\frac{n+2}{2n}}+\mleft(\fint_{B_{r}(z_{0})}\lvert \rho f_{\varepsilon}\rvert^{2}\,{\mathrm d}x\mright)^{1/2}\mright]^{2}
\end{align*}
for some constant \({\hat C}\in(0,\,\infty)\) depending on $n$, $C_{1}$, and $C_{2}$. Here we have applied (\ref{Eq: Poincare--Sobolev inequality}) to the function $w_{\varepsilon}$ to obtain the last inequality. Recalling \(Dw_{\varepsilon}=Du_{\varepsilon}-\zeta\), we finally conclude that (\ref{Eq (Section 4) Claim for Gehring's lemma}) holds for any open ball \(B_{r}(z_{0})\subset B\).
\end{proof}
From Lemma \ref{Lemma: Higher integrability}, we would like to deduce a comparison estimate in Lemma \ref{Lemma: Perturbation result}.
\begin{lemma}\label{Lemma: Perturbation result}
Let $u_{\varepsilon}$ be a weak solution to (\ref{Eq (Section 2) Approximated system}). Assume that positive numbers $\delta$, $\varepsilon$, $\mu$, $M$, $F$, and an open ball $B_{\rho}(x_{0})\Subset\Omega$ satisfy (\ref{Eq (Section 2) delta-epsilon}), (\ref{Eq (Section 2) control of f}), (\ref{Eq (Section 2) esssup V-epsilon})--(\ref{Eq (Section 2) delta<mu}), and 
\begin{equation}\label{Eq (Section 4) average assumption}
\delta+\frac{\mu}{4}\le \mleft\lvert (Du_{\varepsilon})_{x_{0},\,\rho}\mright\rvert\le \delta+\mu.
\end{equation}
Consider the Dirichlet boundary value problem
\begin{equation}\label{Eq (Section 4) Dirichlet boundary problem harmonic system}
\mleft\{\begin{array}{rclcc}
-\divx\mleft({\mathcal B}_{\varepsilon}\mleft((Du_{\varepsilon})_{x_{0},\,\rho}\mright)Dv_{\varepsilon}\mright) & = & 0 &\textrm{in} &B_{\rho/2}(x_{0}),\\ 
v_{\varepsilon} & = & u_{\varepsilon} & \textrm{on} & \partial B_{\rho/2}(x_{0}).
\end{array} \mright.
\end{equation}
Then, there exists a unique function $v_{\varepsilon}\in u_{\varepsilon}+W_{0}^{1,\,2}(\Omega;\,{\mathbb R}^{N})$ that solves (\ref{Eq (Section 4) Dirichlet boundary problem harmonic system}). Moreover, we have
\begin{equation}\label{Eq (Section 4) Comparison estimate}
\fint_{B_{\rho/2}(x_{0})}\lvert Du_{\varepsilon}-Dv_{\varepsilon}\rvert^{2}\,{\mathrm{d}}x\le C\mleft\{ \mleft[\frac{\Phi(x_{0},\,\rho)}{\mu^{2}}\mright]^{\vartheta}\Phi(x_{0},\,\rho)+\mleft(F^{2}+F^{2(1+\vartheta)} \mright)\rho^{2\beta} \mright\},
\end{equation}
and
\begin{equation}\label{Eq (Section 4) Decay of harmonic mappings}
\fint_{B_{\tau\rho}(x_{0})}\mleft\lvert Dv_{\varepsilon}-(Dv_{\varepsilon})_{x_{0},\,\tau\rho}\mright\rvert^{2}\,{\mathrm{d}}x\le C\tau^{2}\fint_{B_{\rho/2}(x_{0})}\mleft\lvert Dv_{\varepsilon}-(Dv_{\varepsilon})_{x_{0},\,\rho}\mright\rvert^{2}\,{\mathrm{d}}x
\end{equation}
for all $\tau\in(0,\,1/2\rbrack$. Here the exponent $\vartheta$ is given by Lemma \ref{Lemma: Higher integrability}, and the constants $C\in(0,\,\infty)$ in (\ref{Eq (Section 4) Comparison estimate})--(\ref{Eq (Section 4) Decay of harmonic mappings}) depends at most on $b$, $n$, $N$, $p$, $q$, $\gamma$, $\Gamma$, $M$, and $\delta$.
\end{lemma}
Before showing Lemma \ref{Lemma: Perturbation result}, we mention that our analysis on perturbation arguments is based on the assumption (\ref{Eq (Section 4) average assumption}). It is easy to estimate $\mleft\lvert (Du_{\varepsilon})_{x_{0},\,\rho}\mright\rvert$ by above. In fact, by (\ref{Eq (Section 2) esssup V-epsilon}) we have
\[\mleft\lvert (Du_{\varepsilon})_{x_{0},\,\rho}\mright\rvert\le \fint_{B_{\rho}(x_{0})}\mleft\lvert Du_{\varepsilon}\mright\rvert\,{\mathrm{d}}x\le \fint_{B_{\rho}(x_{0})}V_{\varepsilon}\,{\mathrm{d}}x\le \delta+\mu.\]
To estimate the value $\mleft\lvert (Du_{\varepsilon})_{x_{0},\,\rho}\mright\rvert$ by below, however, we have to make careful computations, which are based on the measure assumption (\ref{Eq (Section 2) Measure condition 2}) and energy estimates as in Section \ref{Subsect: Energy estimates}. 
The condition $\mleft\lvert (Du_{\varepsilon})_{x_{0},\,\rho}\mright\rvert\ge \delta+\mu/4$ is to be justified later in Sections \ref{Subsect: Shrinking lemmata}--\ref{Subsect: Proof of Campanato-decay}.
\begin{proof}
We set $\xi\coloneqq (Du_{\varepsilon})_{x_{0},\,\rho}\in{\mathbb R}^{Nn}$. 
By (\ref{Eq (Section 2) delta-epsilon}), (\ref{Eq (Section 2) delta<mu}) and (\ref{Eq (Section 4) average assumption}), it is easy to check \(\delta/4\le \mu/4\le \sqrt{\varepsilon^{2}+\lvert \xi\rvert^{2}}\le \delta+(\delta+\mu)\le \delta+M\). Hence, the matrix ${\mathcal B}_{\varepsilon}(\xi)$ admits a constant $m\in(0,\,1)$, depending at most on $b$, $p$, $\gamma$, $\Gamma$, $M$, and $\delta$, such that there holds \(m{\mathrm{id}}_{Nn}\leqslant {\mathcal B}_{\varepsilon}(Du_{\varepsilon})\leqslant m^{-1}{\mathrm{id}}_{Nn}\). In particular, the matrix ${\mathcal B}_{\varepsilon}(\xi)$ satisfies the Legendre condition, and hence unique existence of the Dirichlet problem (\ref{Eq (Section 4) Dirichlet boundary problem harmonic system}) follows from \cite[Theorem 3.39]{MR3099262}. 
Moreover, since the coefficient matrix ${\mathcal B}_{\varepsilon}(\xi)$ is constant and satisfies the Legendre--Hadamard condition, it is easy to find a constant $C=C(n,\,N,\,m)\in(0,\,\infty)$ such that (\ref{Eq (Section 4) Decay of harmonic mappings}) holds (see e.g., \cite[Lemma 2.17]{MR3887613}, \cite[Proposition 5.8]{MR3099262}).

To prove (\ref{Eq (Section 4) Comparison estimate}), we first check $l_{0}\mu^{p-2}{\mathrm{id}}_{Nn} \leqslant{\mathcal B}_{\varepsilon}(\xi)$ for some constant $l_{0}=l_{0}(p)\in(0,\,1)$. This can be easily deduced by \(\mu/4\le \sqrt{\varepsilon^{2}+\lvert \xi\rvert^{2}}\le \delta+(\delta+\mu)\le 5\mu\). Since $v_{\varepsilon}$ satisfies a weak formulation
\[\int_{B}\mleft\langle {\mathcal B}_{\varepsilon}(\xi)Dv_{\varepsilon}\mathrel{}\middle|\mathrel{} D\phi \mright\rangle\,{\mathrm{d}}x=0\quad \textrm{for all }\phi\in W_{0}^{1,\,2}(B;\,{\mathbb R}^{N}),\]
where we write $B\coloneqq B_{\rho/2}(x_{0})$ for notational simplicity, combining with (\ref{Eq (Section 3) Weak formulation local}), we have
\begin{align*}
&\int_{B}\mleft\langle {\mathcal B}_{\varepsilon}(\xi)(Du_{\varepsilon}-Dv_{\varepsilon})\mathrel{}\middle|\mathrel{} D\phi \mright\rangle\,{\mathrm{d}}x\\&=\int_{B}\mleft\langle {\mathcal B}_{\varepsilon}(\xi)(Du_{\varepsilon}-\xi)-(A_{\varepsilon}(Du_{\varepsilon})-A_{\varepsilon}(\xi)) \mathrel{}\middle|\mathrel{}D\phi \mright\rangle\,{\mathrm{d}}x+\int_{B}\langle f_{\varepsilon}\mid \phi\rangle\,{\mathrm{d}}x
\end{align*}
for all $\phi\in W_{0}^{1,\,2}(B;\,{\mathbb R}^{N})$. The assumptions (\ref{Eq (Section 2) esssup V-epsilon})--(\ref{Eq (Section 2) delta<mu}) and (\ref{Eq (Section 4) average assumption}) enable us to use (\ref{Eq (Section 2) Hessian errors}) in Lemma \ref{Lemma: Error estimates}. As a result, we are able to find a constant $C\in(0,\,\infty)$, depending at most on $b$, $p$, $\beta_{0}$, $\gamma$, $\Gamma$, $M$, and $\delta$, such that 
\[\int_{B}\mleft\langle {\mathcal B}_{\varepsilon}(\xi)(Du_{\varepsilon}-Dv_{\varepsilon})\mathrel{}\middle|\mathrel{} D\phi \mright\rangle\,{\mathrm{d}}x\le C\mu^{p-2-\beta_{0}}\int_{B}\lvert Du_{\varepsilon}-\xi\rvert^{1+\beta_{0}}\lvert D\phi\rvert\,{\mathrm{d}}x+\int_{B}\lvert f_{\varepsilon}\rvert\lvert \phi\rvert\,{{\mathrm d}}x\]
for all $\phi\in W_{0}^{1,\,2}(B;\,{\mathbb R}^{N})$.
Testing $\phi\coloneqq u_{\varepsilon}-v_{\varepsilon}\in W_{0}^{1,\,2}(B;\,{\mathbb R}^{N})$ into this weak formulation and using the Cauchy--Schwarz inequality, we compute
\begin{align*}
&l_{0}\gamma\mu^{p-2}\int_{B}\lvert Du_{\varepsilon}-Dv_{\varepsilon}\rvert^{2}\,{\mathrm{d}}x\\&\le C\mu^{p-2-\beta_{0}}\mleft(\int_{B}\lvert Du_{\varepsilon}-\xi\rvert^{2(1+\beta_{0})}\,{\mathrm{d}}x \mright)^{1/2}\mleft(\int_{B}\lvert Du_{\varepsilon}-Dv_{\varepsilon} \rvert^{2}\,{\mathrm{d}}x \mright)^{1/2}\\&\quad +C(n,\,q)F\rho^{\beta+\frac{n}{2}}\mleft(\int_{B}\lvert Du_{\varepsilon}-Dv_{\varepsilon} \rvert^{2}\,{\mathrm{d}}x \mright)^{1/2}.
\end{align*}
Here we have also applied the Poincar\'{e} inequality to the function $u_{\varepsilon}-v_{\varepsilon}\in W_{0}^{1,\,2}(B;\,{\mathbb R}^{N})$. Thus, we obtain
\[\fint_{B}\lvert Du_{\varepsilon}-Dv_{\varepsilon}\rvert^{2}\,{\mathrm{d}}x\le C\mleft[\frac{1}{\mu^{2\beta_{0}}}\fint_{B}\lvert Du_{\varepsilon}-\xi\rvert^{2(1+\beta_{0})}\,{\mathrm{d}}x+\mu^{2(2-p)}F^{2}\rho^{2\beta} \mright]\]
for some constant $C\in(0,\,\infty)$ depending at most on $b$, $n$, $N$, $p$, $q$, $\beta_{0}$, $\gamma$, $\Gamma$, $M$, and $\delta$. Since $\xi$ clearly satisfies (\ref{Eq (Section 4) xi}), we are able to apply Lemma \ref{Lemma: Gehring lemma}.
Also, by (\ref{Eq (Section 2) esssup V-epsilon})--(\ref{Eq (Section 2) delta<mu}) and (\ref{Eq (Section 4) integrability up}), it is easy to check that $2(1+\vartheta)\le 2(1+\beta_{0})$ and 
\[\lvert Du_{\varepsilon}-\xi\rvert\le 2(\delta+\mu)\le 4\mu\quad \textrm{a.e. in }B_{\rho}(x_{0}).\]
With these results in mind, we use H\"{o}lder's inequality and (\ref{Eq (Section 4) Higher integrability result}) to obtain
\begin{align*}
\fint_{B}\lvert Du_{\varepsilon}-Dv_{\varepsilon}\rvert^{2}\,{\mathrm{d}}x&\le C\mleft[\frac{c(n,\,\vartheta)}{\mu^{2\vartheta}}\fint_{B}\lvert Du_{\varepsilon}-\xi\rvert^{2(1+\vartheta)}\,{\mathrm{d}}x+\mu^{2(2-p)}F^{2}\rho^{2\beta} \mright]\\&\le C\mleft[\frac{\Phi(x_{0},\,\rho)^{1+\vartheta}}{\mu^{2\vartheta}}+\frac{F^{2(1+\vartheta)}\rho^{2\beta(1+\vartheta)}}{\mu^{2\vartheta}}+\mu^{2(2-p)}F^{2}\rho^{2\beta} \mright]
\end{align*}
for some $C\in(0,\,\infty)$. Here it is mentioned that $0<\rho\le 1$ and $\delta<\mu<M$ hold by (\ref{Eq (Section 2) esssup V-epsilon})--(\ref{Eq (Section 2) delta<mu}), and therefore (\ref{Eq (Section 4) Comparison estimate}) is verified.
\end{proof}

\subsection{Key lemmata in shrinking methods}\label{Subsect: Shrinking lemmata}
In Section \ref{Subsect: Shrinking lemmata}, we provide two lemmata on shrinking methods. To prove these, we use results from Sections \ref{Subsect: Energy estimates}--\ref{Subsect: freezing coefficient method}. 

The first lemma (Lemma \ref{Lemma: Shrinking lemma 1}) states that an average $(Du_{\varepsilon})_{x_{0},\,\rho}\in{\mathbb R}^{Nn}$ does not vanish under suitable settings. This result makes sure that our freezing coefficient argument given in Section \ref{Subsect: freezing coefficient method} will work. 
\begin{lemma}\label{Lemma: Shrinking lemma 1}
Let $u_{\varepsilon}$ be a weak solution to (\ref{Eq (Section 2) Approximated system}) in $\Omega$. Assume that positive numbers $\delta$, $\varepsilon$, $\mu$, $F$, $M$, and an open ball $B_{\rho}(x_{0})\Subset\Omega$ satisfy (\ref{Eq (Section 2) delta-epsilon}), (\ref{Eq (Section 2) control of f}), and (\ref{Eq (Section 2) esssup V-epsilon})--(\ref{Eq (Section 2) delta<mu}). Then, for each fixed $\theta\in(0,\,1/16)$, there exist numbers $\nu\in(0,\,1/4),\,{\hat\rho}\in(0,\,1)$, depending at most on $b$, $n$, $N$, $p$, $q$, $\gamma$,  $\Gamma$, $F$, $M$, $\delta$, and $\theta$, such that the following statement is valid. If both $0<\rho<{\hat\rho}$ and (\ref{Eq (Section 2) Measure condition 2}) hold, then we have
\begin{equation}\label{Eq (Section 4) result 1 average integral non-vanishing}
\lvert (Du_{\varepsilon})_{x_{0},\,\rho}\rvert\ge \delta+\frac{\mu}{2},
\end{equation}
and
\begin{equation}\label{Eq (Section 4) result 2 oscillation control}
\Phi(x_{0},\,\rho)\le \theta\mu^{2}.
\end{equation}
\end{lemma}
Although the proof of Lemma \ref{Lemma: Shrinking lemma 1} is inspired by \cite[Lemma 5.5]{BDGPdN}, it should be emphasized that on our regularized problem (\ref{Eq (Section 2) Approximated system}), we have to deal with two different moduli $\lvert Du_{\varepsilon}\rvert$ and $V_{\varepsilon}=\sqrt{\varepsilon^{2}+\lvert Du_{\varepsilon}\rvert^{2}}$, which is substantially different from \cite{BDGPdN}. 
Therefore, as mentioned in Section \ref{Subsect: Models and comparisons}, in the proof of Lemma \ref{Lemma: Shrinking lemma 1}, we have to carefully utilize (\ref{Eq (Section 2) delta-epsilon}), so that $\varepsilon$ can be suitably dominated by $\delta$. Also, it should be mentioned that our proof is substantially the same with \cite[Lemma 3.12]{T-scalar}, although there are some differences on ranges of $\varepsilon$ or $\theta$.
\begin{proof}
We will later choose constants $\tau\in(0,\,1),\,\nu\in(0,\,1/4),\,{\hat\rho}\in(0,\,1)$. 
By (\ref{Eq (Section 4) minimizing property on L2-average}), we have
\[\Phi(x_{0},\,\rho)\le\fint_{B_{\rho}(x_{0})}\mleft\lvert Du_{\varepsilon}-(Du_{\varepsilon})_{x_{0},\,\tau\rho}\mright\rvert^{2}\,{\mathrm d}x=\mathbf{J}_{1}+\mathbf{J}_{2}\]
with
\[\mleft\{\begin{array}{rcl}
\mathbf{J}_{1}&\coloneqq&\displaystyle\frac{1}{\lvert B_{\rho}(x_{0})\rvert}\displaystyle\int_{B_{\tau\rho}(x_{0})}\mleft\lvert Du_{\varepsilon}-(Du_{\varepsilon})_{x_{0},\,\tau\rho}\mright\rvert^{2}\,{\mathrm d}x, \\ \mathbf{J}_{2}&\coloneqq & \displaystyle\frac{1}{\lvert B_{\rho}(x_{0})\rvert}\displaystyle\int_{B_{\rho}(x_{0})\setminus B_{\tau\rho}(x_{0})}\mleft\lvert Du_{\varepsilon}-(Du_{\varepsilon})_{x_{0},\,\tau\rho}\mright\rvert^{2}\,{\mathrm d}x.
\end{array} \mright.\]
For \(\mathbf{J}_{1}\), we apply Lemma \ref{Lemma: Energy estimates 2} to obtain
\[\mathbf{J}_{1}=\tau^{n}\Phi(x_{0},\,\tau\rho)\le C_{\dagger}\mu^{2}\mleft[\frac{\nu^{2/n}}{(1-\tau)^{2}}+\frac{F^{2}}{\nu} \rho^{2\beta}\mright]\]
with \(C_{\dagger}=C_{\dagger}(b,\,n,\,N,\,p,\,\gamma,\,\Gamma,\,M,\,\delta)\in(0,\,\infty)\).
For \(\mathbf{J}_{2}\), we use (\ref{Eq (Section 2) esssup V-epsilon})--(\ref{Eq (Section 2) delta<mu}) to get \(\lvert Du_{\varepsilon}\rvert\le V_{\varepsilon}\le \delta+\mu\le 2\mu\) a.e. in \(B_{\rho}(x_{0})\), and hence \(\lvert(Du_{\varepsilon})_{x_{0},\,\tau\rho}\rvert\le 2\mu\). These inequalities yield
\[\mathbf{J}_{2}\le 8\mu^{2}\cdot\frac{\lvert B_{\rho}(x_{0})\setminus B_{\tau\rho}(x_{0})\rvert}{\lvert B_{\rho}(x_{0})\rvert}=8\mu^{2}(1-\tau^{n})\le 8n\mu^{2}(1-\tau),\]
where we have used \(1-\tau^{n}=(1+\tau+\cdots+\tau^{n-1})(1-\tau)\le n(1-\tau)\).
Hence, we obtain
\[\Phi(x_{0},\,\rho)\le C_{\dagger}\mu^{2}\mleft[\frac{\nu^{2/n}}{(1-\tau)^{2}}+\frac{F^{2}}{\nu} {\hat\rho}^{2\beta}\mright]+8n(1-\tau)\mu^{2}.\]
We first fix \[\tau\coloneqq 1-\frac{\theta}{24n}\in(0,\,1),\quad\textrm{so that there holds}\quad 8n(1-\tau)=\frac{\theta}{3}.\] 
Next we choose \(\nu\in(0,\,1/4)\) sufficiently small that it satisfies
\[\nu\le\min\mleft\{\,\mleft(\frac{\theta(1-\tau)^{2}}{3C_{\dagger}}\mright)^{n/2},\,\frac{1-4\sqrt{\theta}}{11} \,\mright\},\]
so that we have
\[\frac{C_{\dagger}\nu^{2/n}}{(1-\tau)^{2}}\le \frac{\theta}{3},\quad \textrm{and}\quad \sqrt{\theta}\le \frac{1-11\nu}{4}.\]
Corresponding to this \(\nu\), we choose and fix sufficiently small \({\hat\rho}\in(0,\,1)\) satisfying 
\[{\hat\rho}^{2\beta}\le \frac{\nu\theta}{3C_{\dagger}(1+F^{2})},\]
which yields \(C_{\dagger}F^{2}{\hat\rho}^{2\beta}/\nu\le \theta/3\).
Our settings of \(\tau,\,\nu,\,{\hat\rho}\) clearly yield (\ref{Eq (Section 4) result 2 oscillation control}).

To prove (\ref{Eq (Section 4) result 1 average integral non-vanishing}), we recall (\ref{Eq (Section 4) Non-degenerate-gradient over superlevel set}) and use (\ref{Eq (Section 2) Measure condition 2}), Then we obtain
\begin{align*}
\fint_{B_{\rho}(x_{0})}\lvert Du_{\varepsilon}\rvert\,{\mathrm d}x&\ge \frac{\lvert S_{\rho,\,\mu,\,\nu}(x_{0})\rvert}{\lvert B_{\rho}(x_{0})\rvert}\cdot \essinf\limits_{S_{\rho,\,\mu,\,\nu}(x_{0})}\,\lvert Du_{\varepsilon}\rvert\\&\ge (1-\nu)\cdot\mleft[(1-\nu)\mu+\frac{3}{4}\delta\mright]>0.
\end{align*}
On the other hand, by the triangle inequality, the Cauchy--Schwarz inequality and (\ref{Eq (Section 4) result 2 oscillation control}), it is easy to get
\begin{align*}
\mleft\lvert\fint_{B_{\rho}(x_{0})}\lvert Du_{\varepsilon}\rvert\,{\mathrm d}x -\mleft\lvert(Du_{\varepsilon})_{x_{0},\,\rho}\mright\rvert\mright\rvert&=\mleft\lvert \fint_{B_{\rho}(x_{0})}\mleft[\lvert Du_{\varepsilon}\rvert-\mleft\lvert(Du_{\varepsilon})_{x_{0},\,\rho}\mright\rvert\mright]\,{\mathrm d}x\mright\rvert\\&\le \fint_{B_{\rho}(x_{0})}\mleft\lvert Du_{\varepsilon}-(Du_{\varepsilon})_{x_{0},\,\rho}\mright\rvert\,{\mathrm d}x\\&\le \sqrt{\Phi(x_{0},\,\rho)}\le \sqrt{\theta}\mu.
\end{align*}
Again by the triangle inequality, we obtain
\begin{align*}
\mleft\lvert(Du_{\varepsilon})_{x_{0},\,\rho}\mright\rvert&\ge \mleft\lvert\fint_{B_{\rho}(x_{0})}\lvert Du_{\varepsilon}\rvert\,{\mathrm d}x \mright\rvert-\mleft\lvert\fint_{B_{\rho}(x_{0})}\lvert Du_{\varepsilon}\rvert\,{\mathrm d}x-\mleft\lvert(Du_{\varepsilon})_{x_{0},\,\rho}\mright\rvert\mright\rvert\\&\ge \mleft((1-\nu)^{2}-\sqrt{\theta}\mright)\mu+\frac{3}{4}(1-\nu)\delta.
\end{align*}
By (\ref{Eq (Section 2) delta<mu}) and our choice of \(\nu\), we can check that
\begin{align*}
\mleft((1-\nu)^{2}-\sqrt{\theta}\mright)\mu+\frac{3}{4}(1-\nu)\delta-\mleft(\delta+\frac{\mu}{2}\mright)&=\mleft(\frac{1}{2}-2\nu+\nu^{2}-\sqrt{\theta}\mright)\mu-\mleft(\frac{1}{4}+\frac{3}{4}\nu\mright)\delta\\&\ge \mleft(\frac{1-11\nu}{4}-\sqrt{\theta} \mright)\mu\ge 0,
\end{align*}
which completes the proof of (\ref{Eq (Section 4) result 1 average integral non-vanishing}).
\end{proof}

The second lemma (Lemma \ref{Lemma: Shrinking lemma 2}) is a result from perturbation arguments given in Section \ref{Subsect: freezing coefficient method}. 
\begin{lemma}\label{Lemma: Shrinking lemma 2}
Let $u_{\varepsilon}$ be a weak solution to (\ref{Eq (Section 2) Approximated system}) in $\Omega$. Assume that positive numbers $\delta$, $\varepsilon$, $\mu$, $F$, $M$, and an open ball $B_{\rho}(x_{0})$ satisfy (\ref{Eq (Section 2) delta-epsilon}), (\ref{Eq (Section 2) control of f}), (\ref{Eq (Section 2) esssup V-epsilon})--(\ref{Eq (Section 2) delta<mu}), and $0<\rho<1$.
Let $\vartheta$ be the constant in Lemma \ref{Lemma: Gehring lemma}.
If (\ref{Eq (Section 4) average assumption}) and 
\begin{equation}\label{Eq (Section 4) energy decay setting}
\Phi(x_{0},\,\rho)\le \tau^{\frac{n+2}{\vartheta}}\mu^{2}
\end{equation}
hold for some $\tau\in(0,\,1/2)$, then we have
\begin{equation}\label{Eq (Section 4) energy decay result}
\Phi(x_{0},\,\tau\rho)\le C_{\ast}\mleft[\tau^{2}\Phi(x_{0},\,\rho)+\frac{\rho^{2\beta}}{\tau^{n}}\mu^{2}\mright].
\end{equation}
Here the constant $C_{\ast}$ depends at most on $b$, $n$, $N$, $p$, $q$, $\beta_{0}$, $\gamma$, $\Gamma$, $F$, $M$, and $\delta$.
\end{lemma}
\begin{proof}
Let \(v_{\varepsilon}\in u_{\varepsilon}+W_{0}^{1,\,2}(B_{\rho/2}(x_{0});\,{\mathbb R}^{N})\) be the unique solution of (\ref{Eq (Section 4) Dirichlet boundary problem harmonic system}). We use (\ref{Eq (Section 4) minimizing property on L2-average}) to get
\begin{align*}
\Phi(x_{0},\,\tau\rho)&\le\fint_{B_{\tau\rho}(x_{0})}\mleft\lvert Du_{\varepsilon}-(Dv_{\varepsilon})_{x_{0},\,\tau\rho}\mright\rvert^{2}\,{\mathrm d}x \\&\le \frac{2}{(2\tau)^{n}}\fint_{B_{\rho/2}(x_{0})}\lvert Du_{\varepsilon}-Dv_{\varepsilon}\rvert^{2}\,{\mathrm d}x +2\fint_{B_{\tau\rho}(x_{0})}\mleft\lvert Dv_{\varepsilon}-(Dv_{\varepsilon})_{x_{0},\,\tau\rho} \mright\rvert^{2}\,{\mathrm d}x,
\end{align*}
where we have used \(\lvert B_{\rho/2}(x_{0})\rvert=(2\tau)^{-n}\cdot \lvert B_{\tau\rho}(x_{0})\rvert\).
For the second average integral, (\ref{Eq (Section 4) minimizing property on L2-average}) and (\ref{Eq (Section 4) Decay of harmonic mappings}) yield
\[\fint_{B_{\tau\rho}(x_{0})}\mleft\lvert Dv_{\varepsilon}-(Dv_{\varepsilon})_{x_{0},\,\tau\rho} \mright\rvert^{2}\,{\mathrm d}x\le C\mleft[\tau^{2}\fint_{B_{\rho/2}(x_{0})}\lvert Du_{\varepsilon}-Dv_{\varepsilon}\rvert^{2}\,{\mathrm d}x+\tau^{2}\Phi(x_{0},\,\rho)\mright]
\]
with \(C=C(n,\,N,\,p,\,q,\,\gamma,\,\Gamma,\,M,\,\delta)\in(0,\,\infty)\).
By (\ref{Eq (Section 2) delta<mu}), (\ref{Eq (Section 4) Comparison estimate}) and (\ref{Eq (Section 4) energy decay setting}),  we are able to compute
\begin{align*}
\Phi(x_{0},\,\tau\rho)&\le C\mleft[\frac{1}{\tau^{n}}\fint_{B_{\rho/2}(x_{0})}\lvert Du_{\varepsilon}-Dv_{\varepsilon}\rvert^{2}\,{\mathrm d}x+\tau^{2}\Phi(x_{0},\,\rho)\mright]\\&\le C\mleft\{ \mleft[\frac{\Phi(x_{0},\,\rho)}{\mu^{2}}\mright]^{\vartheta}\cdot\frac{\Phi(x_{0},\,\rho)}{\tau^{n}}+\frac{F^{2}+F^{2(1+\vartheta)}}{\tau^{n}}\rho^{2\beta}+\tau^{2}\Phi(x_{0},\,\rho)\mright\}\\&\le C\mleft[\tau^{2}\Phi(x_{0},\,\rho)+\mleft(F^{2}+F^{2(1+\vartheta)}\mright)\cdot\frac{\rho^{2\beta}}{\tau^{n}}\cdot\mleft(\frac{\mu}{\delta}\mright)^{2}\mright]\\&\le C_{\ast}(b,\,n,\,N,\,p,\,q,\,\beta_{0},\,\gamma,\,\Gamma,\,F,\,M,\,\delta)\mleft[\tau^{2}\Phi(x_{0},\,\rho)+\frac{\rho^{2\beta}}{\tau^{n}}\mu^{2}\mright],
\end{align*}
which completes the proof.
\end{proof}
\subsection{Proof of Proposition \ref{Prop: Schauder estimate}}\label{Subsect: Proof of Campanato-decay}
We would like to prove Proposition \ref{Prop: Schauder estimate} by shrinking arguments. A key point of the proof, which is inspired by \cite[Proposition 3.4]{BDGPdN}, is to justify that an average $(Du_{\varepsilon})_{x_{0},\,r}\in{\mathbb R}^{Nn}$ never vanishes even when $r$ tends to $0$, so that Lemma \ref{Lemma: Shrinking lemma 2} can be applied. To verify this, we make careful computations, found in the proof of Lemma \ref{Lemma: Shrinking lemma 1}.
\begin{proof}
We set a constant $\vartheta$ as in Lemma \ref{Lemma: Higher integrability}. We will determine a sufficiently small constant $\tau\in(0,\,1/2)$, and corresponding to this \(\tau\), we will put the desired constants \(\rho_{\star}\in(0,\,1)\) and \(\nu\in(0,\,1/4)\). 

We first assume that 
\begin{equation}\label{Eq (Section 4) Determination of tau 1}
0<\tau<\max\mleft\{\,\tau^{\beta},\,\tau^{1-\beta}\,\mright\}<\frac{1}{16},\quad \textrm{and therefore}\quad
\theta\coloneqq \tau^{\frac{n+2}{\vartheta}}\in\mleft(0,\,\frac{1}{16}\mright)
\end{equation}
Throughout the proof, we let \(\nu\in(0,\,1/6)\) and \({\hat\rho}\in(0,\,1)\) be sufficiently small constants satisfying Lemma \ref{Lemma: Shrinking lemma 1} with \(\theta\) defined by (\ref{Eq (Section 4) Determination of tau 1}). We also assume that \(\rho_{\star}\) is so small that there holds
\begin{equation}\label{Eq (Section 4) Determination of rho-star 1}
0<\rho_{\star}\le {\hat\rho}<1.
\end{equation}
Assume that the open ball \(B_{\rho}(x_{0})\) satisfies \(0<\rho<\rho_{\star}\), and (\ref{Eq (Section 2) Measure condition 2}) holds for the constant \(\nu\in(0,\,1/6)\).
We set a non-negative decreasing sequence \(\{\rho_{k}\}_{k=0}^{\infty}\) by \(\rho_{k}\coloneqq \tau^{k}\rho\). We will choose suitable \(\tau\) and \(\rho_{\ast}\) such that there hold
\begin{equation}\label{Eq (Section 4) Induction claim 2}
\mleft\lvert (Du_{\varepsilon})_{x_{0},\,{\rho_{k}}} \mright\rvert\ge \delta+\mleft[\frac{1}{2}-\frac{1}{8}\sum_{j=0}^{k-1}2^{-j}\mright]\mu \ge \delta+\frac{\mu}{4}
\end{equation}
and
\begin{equation}\label{Eq (Section 4) Induction claim 1}
\Phi(x_{0},\,\rho_{k})\le \tau^{2\beta k}\tau^{\frac{n+2}{\vartheta}}\mu^{2},
\end{equation}
for all \(k\in{\mathbb Z}_{\ge 0}\), which will be proved by mathematical induction.
For \(k=0,\,1\), we apply Lemma \ref{Lemma: Shrinking lemma 1} to deduce (\ref{Eq (Section 4) result 1 average integral non-vanishing})--(\ref{Eq (Section 4) result 2 oscillation control}) with \(\theta=\tau^{\frac{n+2}{\vartheta}}\). In particular, we have
\begin{equation}\label{Eq (Section 4) Phi estimate for the first step}
\Phi(x_{0},\,\rho)\le \tau^{\frac{n+2}{\vartheta}}\mu^{2},
\end{equation}
and hence (\ref{Eq (Section 4) Induction claim 1}) is obvious when $k=0$. From (\ref{Eq (Section 4) result 1 average integral non-vanishing}), we have already known that (\ref{Eq (Section 4) Induction claim 2}) holds for \(k=0\). Also, (\ref{Eq (Section 4) result 1 average integral non-vanishing}) and (\ref{Eq (Section 4) Phi estimate for the first step}) enable us to apply Lemma \ref{Lemma: Shrinking lemma 2} to obtain
\begin{align*}
\Phi(x_{0},\,\rho_{1})&\le C_{\ast}\mleft[\tau^{2}\Phi(x_{0},\,\rho)+\frac{\rho^{2\beta}}{\tau^{n}}\mu^{2}\mright]\\&\le C_{\ast}\tau^{2(1-\beta)}\cdot\tau^{2\beta}\tau^{\frac{n+2}{\vartheta}}\mu^{2}+\frac{C_{\ast}\rho_{\star}^{2\beta}}{\tau^{n}}\mu^{2},
\end{align*}
where \(C_{\ast}\in(0,\,\infty)\) is a constant as in Lemma \ref{Lemma: Shrinking lemma 2}, depending at most on $b$, $n$, $N$, $p$, $q$, $\beta_{0}$, $\gamma$, $\Gamma$, $F$, $M$, and $\delta$.
Now we assume that \(\tau\) and \(\rho_{\star}\) satisfy
\begin{equation}\label{Eq (Section 4) Determination of tau 2}
C_{\ast}\tau^{2(1-\beta)}\le \frac{1}{2},
\end{equation}
and
\begin{equation}\label{Eq (Section 4) Determination of rho-star 2}
C_{\ast}\rho_{\star}^{2\beta}\le \frac{1}{2}\tau^{n+2\beta+\frac{n+2}{\vartheta}},
\end{equation}
so that (\ref{Eq (Section 4) Induction claim 1}) holds for \(k=1\).
In particular, by (\ref{Eq (Section 4) integrability up}), (\ref{Eq (Section 4) Determination of tau 1}), (\ref{Eq (Section 4) Phi estimate for the first step}) and the Cauchy--Schwarz inequality, we obtain
\begin{align*}
\mleft\lvert(Du_{\varepsilon})_{x_{0},\,\rho_{1}}-(Du_{\varepsilon})_{x_{0},\,\rho_{0}} \mright\rvert&\le \fint_{B_{\rho_{1}}(x_{0})}\mleft\lvert Du_{\varepsilon}-(Du_{\varepsilon})_{x_{0},\,\rho_{0}}\mright\rvert\,{\mathrm d}x\nonumber\\&\le \mleft(\fint_{B_{\rho_{1}}(x_{0})}\mleft\lvert Du_{\varepsilon}-(Du_{\varepsilon})_{x_{0},\,\rho_{0}}\mright\rvert^{2}\,{\mathrm d}x\mright)^{1/2}=\tau^{-\frac{n}{2}}\Phi(x_{0},\,\rho)^{1/2}\nonumber\\&\le \tau^{\frac{n+2}{2\vartheta}-\frac{n}{2}}\mu\le \tau\mu\le \frac{1}{8}\mu.
\end{align*}
Combining this result with (\ref{Eq (Section 4) result 1 average integral non-vanishing}), we use the triangle inequality to get
\[\mleft\lvert (Du_{\varepsilon})_{x_{0},\,\rho_{1}}\mright\rvert\ge \mleft\lvert (Du_{\varepsilon})_{x_{0},\,\rho_{0}}\mright\rvert-\mleft\lvert(Du_{\varepsilon})_{x_{0},\,\rho_{1}}-(Du_{\varepsilon})_{x_{0},\,\rho_{0}} \mright\rvert\ge \mleft(\delta+\frac{\mu}{2}\mright)-\frac{\mu}{8},\]
which means that (\ref{Eq (Section 4) Induction claim 2}) holds true for \(k=1\).
Next, we assume that the claims (\ref{Eq (Section 4) Induction claim 2})--(\ref{Eq (Section 4) Induction claim 1}) are valid for an integer \(k\ge 1\).
Then \(\Phi(x_{0},\,\rho_{k})\le \tau^{\frac{n+2}{\vartheta}}\mu^{2}\) clearly holds. Combining this result with the induction hypothesis (\ref{Eq (Section 4) Induction claim 2}), we have clarified that the solution \(u_{\varepsilon}\) satisfies assumptions of Lemma \ref{Lemma: Shrinking lemma 2} in a smaller ball \(B_{\rho_{k}}(x_{0})\subset B_{\rho}(x_{0})\). By Lemma \ref{Lemma: Shrinking lemma 2}, (\ref{Eq (Section 4) Determination of tau 2}), and the induction hypothesis (\ref{Eq (Section 4) Induction claim 1}), we compute
\begin{align*}
\Phi(x_{0},\,\rho_{k+1})&\le C_{\ast}\mleft[\tau^{2}\Phi(x_{0},\,\rho_{k})+\frac{\rho_{k}^{2\beta}}{\tau^{n}}\mu^{2}\mright]\\&\le C_{\ast}\tau^{2(1-\beta)}\cdot \tau^{2\beta (k+1)}\tau^{\frac{n+2}{\vartheta}}\mu^{2}+\frac{C_{\ast}\rho_{\star}^{2\beta}}{\tau^{n}}\cdot\tau^{2\beta k}\mu^{2}\\&\le \tau^{2\beta(k+1)}\cdot \tau^{\frac{n+2}{\vartheta}}\mu^{2},
\end{align*}
which means that (\ref{Eq (Section 4) Induction claim 1}) holds true for \(k+1\). Also, by the Cauchy--Schwarz inequality and the induction hypothesis (\ref{Eq (Section 4) Induction claim 1}), we have
\begin{align*}
\mleft\lvert(Du_{\varepsilon})_{x_{0},\,\rho_{k+1}}- (Du_{\varepsilon})_{x_{0},\,\rho_{k}}\mright\rvert&\le \fint_{B_{\rho_{k+1}}(x_{0})}\mleft\lvert Du_{\varepsilon}- (Du_{\varepsilon})_{x_{0},\,\rho_{k}}\mright\rvert\,{\mathrm d}x \\&\le \mleft(\fint_{B_{\rho_{k+1}}(x_{0})}\mleft\lvert Du_{\varepsilon}-(Du_{\varepsilon})_{x_{0},\,\rho_{k}}\mright\rvert^{2}\,{\mathrm d}x\mright)^{1/2}=\tau^{-n/2}\Phi(x_{0},\,\rho_{k})^{1/2}\\&\le \tau^{\beta k}\tau^{\frac{n+2}{\vartheta}-\frac{n}{2}}\mu\le 2^{-k}\cdot \frac{1}{8} \mu.
\end{align*}
Here we have also used (\ref{Eq (Section 4) Determination of tau 1}).
Therefore, by the induction hypothesis (\ref{Eq (Section 4) Induction claim 2}) and the triangle inequality, we get
\begin{align*}
\mleft\lvert(Du_{\varepsilon})_{x_{0},\,\rho_{k+1}}\mright\rvert&\ge\mleft\lvert(Du_{\varepsilon})_{x_{0},\,\rho_{k}}\mright\rvert-\mleft\lvert (Du_{\varepsilon})_{x_{0},\,\rho_{k+1}}-(Du_{\varepsilon})_{x_{0},\,\rho_{k}}\mright\rvert\\& \ge \delta+ \mleft[\frac{1}{2}-\frac{1}{8}\sum_{j=0}^{k-1}2^{-j}\mright]\mu-\frac{1}{8}\cdot 2^{-k}\mu,
\end{align*}
which implies that (\ref{Eq (Section 4) Induction claim 2}) is valid for \(k+1\). This completes the proof of (\ref{Eq (Section 4) Induction claim 2})--(\ref{Eq (Section 4) Induction claim 1}).

We define
\[\Psi_{2\delta,\,\varepsilon}(x_{0},\,r)\coloneqq \fint_{B_{r}(x_{0})}\mleft\lvert {\mathcal G}_{2\delta,\,\varepsilon}(Du_{\varepsilon})-\mleft({\mathcal G}_{2\delta,\,\varepsilon}(Du_{\varepsilon})\mright)_{x_{0},\,r}\mright\rvert^{2}\,{\mathrm d}x\]
for \(r\in(0,\,\rho\rbrack\), and we set a sequence of vectors \(\{\Gamma_{k}\}_{k=0}^{\infty}\subset {\mathbb R}^{Nn}\) by
\[\Gamma_{k}\coloneqq\mleft({\mathcal G}_{2\delta,\,\varepsilon}(Du_{\varepsilon})\mright)_{x_{0},\,\rho_{k}}\quad \textrm{for }k\in{\mathbb Z}_{\ge 0}.\]
Let \(c_{\dagger}>0\) be the constant satisfying (\ref{Eq (Section 2) Lipschitz bounds of the mapping G-2delta-epsilon}). For each \(k\in{\mathbb Z}_{\ge 0}\), we apply (\ref{Eq (Section 4) minimizing property on L2-average}), (\ref{Eq (Section 4) integrability up}) and (\ref{Eq (Section 4) Induction claim 1})--(\ref{Eq (Section 4) Phi estimate for the first step}) to get
\begin{align*}
\Psi_{2\delta,\,\varepsilon}(x_{0},\,\rho_{k})&\le \fint_{B_{\rho_{k}}(x_{0})}\mleft\lvert{\mathcal G}_{2\delta,\,\varepsilon}(Du_{\varepsilon})-{\mathcal G}_{2\delta,\,\varepsilon}\mleft((Du_{\varepsilon})_{x_{0},\,\rho_{k}} \mright)\mright\rvert^{2}\,{\mathrm d}x\\&\le c_{\dagger}^{2}\cdot \Phi(x_{0},\,\rho_{k})\\&\le c_{\dagger}^{2}\tau^{2\beta k}\tau^{\frac{n+2}{\vartheta}}\mu^{2}\le c_{\dagger}^{2}\tau^{2\beta k+2(n+2)}\mu^{2}.
\end{align*}
Here we let \(\tau\) satisfy
\begin{equation}\label{Eq (Section 4) Determination of tau 3}
\tau\le \frac{1}{\sqrt{2}c_{\dagger}}
\end{equation}
to get 
\begin{equation}\label{Eq (Section 4) Campanato-growth estimate for G-delta-epsilon}
\Psi_{2\delta,\,\varepsilon}\mleft(x_{0},\,\rho_{k}\mright)\le \tau^{2n+2}\tau^{2\beta k}\mu^{2}\quad \textrm{for all }k\in{\mathbb Z}_{\ge 0}.
\end{equation}
By (\ref{Eq (Section 4) Campanato-growth estimate for G-delta-epsilon}) and the Cauchy--Schwarz inequality, we have
\begin{align*}
\lvert \Gamma_{k+1}-\Gamma_{k}\rvert&\le \fint_{B_{k}}\mleft\lvert{\mathcal G}_{2\delta,\,\varepsilon}(Du_{\varepsilon})-\Gamma_{k}\mright\rvert\,{\mathrm d}x\\&\le \mleft(\fint_{B_{k}}\mleft\lvert{\mathcal G}_{2\delta,\,\varepsilon}(Du_{\varepsilon})-\Gamma_{k}\mright\rvert^{2}\,{\mathrm d}x\mright)^{1/2}=\tau^{-n/2}\Psi_{2\delta,\,\varepsilon}(x_{0},\,\rho_{k})^{1/2}\\&\le \tau^{n/2+1}\tau^{\beta k}\mu
\end{align*}
for all \(k\in{\mathbb Z}_{\ge 0}\). In particular, for all \(k,\,l\in{\mathbb Z}_{\ge 0}\) with \(k<l\), we have
\begin{align*}
\lvert \Gamma_{l}-\Gamma_{k}\rvert&\le \sum_{i=k}^{l-1}\lvert \Gamma_{i+1}-\Gamma_{i}\rvert\le \sum_{i=k}^{l-1}\tau^{n/2+1}\tau^{\beta i}\mu\\&\le \tau^{n/2+1}\mu\sum_{i=k}^{\infty}\tau^{\beta i}=\tau^{n/2+1}\frac{\tau^{\beta k}}{1-\tau^{\beta}}\mu.
\end{align*}
The setting (\ref{Eq (Section 4) Determination of tau 1}) clearly yields $\tau^{\beta}\le 1/2$. Hence, we have
\begin{equation}\label{Eq (Section 4) Cauchy-estimate}
\lvert \Gamma_{k}-\Gamma_{l}\rvert\le 2\tau^{n/2+1}\tau^{\beta k}\mu\quad \textrm{for all $k,l\in{\mathbb Z}_{\ge 0}$ with }k<l,
\end{equation}
which implies that \(\{\Gamma_{k}\}_{k=0}^{\infty}\) is a Cauchy sequence in \({\mathbb R}^{Nn}\). Therefore the limit \[\Gamma_{\infty}\coloneqq\lim_{k\to\infty}\Gamma_{k}\in{\mathbb R}^{Nn}\]
exists. Moreover, by letting \(l\to\infty\) in (\ref{Eq (Section 4) Cauchy-estimate}), we have
\[\lvert \Gamma_{k}-\Gamma_{\infty}\rvert\le 2\tau^{n/2+1}\tau^{\beta k}\mu\quad \textrm{for every }k\in{\mathbb Z}_{\ge 0}.\]
Combining this result with (\ref{Eq (Section 4) Campanato-growth estimate for G-delta-epsilon}), we obtain
\begin{align}\label{Eq (Section 4) Digital Campanato-Growth estimate in Perturbation}
\fint_{B_{\rho_{k}}(x_{0})}\mleft\lvert{\mathcal G}_{2\delta,\,\varepsilon}(Du_{\varepsilon})-\Gamma_{\infty} \mright\rvert^{2}\,{\mathrm d}x&\le 2\fint_{B_{\rho_{k}}(x_{0})}\mleft[\mleft\lvert {\mathcal G}_{2\delta,\,\varepsilon}(Du_{\varepsilon})-\Gamma_{k} \mright\rvert^{2}+\lvert\Gamma_{k}-\Gamma_{\infty} \rvert^{2} \mright] \,{\mathrm d}x\nonumber\\&\le 2\mleft(\tau^{2n+2}\tau^{2\beta k}\mu^{2}+4\tau^{n+2}\tau^{2\beta k}\mu^{2}\mright)\nonumber\\&\le  10\tau^{2(1-\beta)}\cdot \tau^{n+2\beta (k+1)}\mu^{2}\nonumber\\&\le \tau^{n+2\beta (k+1)}\mu^{2}.
\end{align}
Here we have used $\tau^{2(1-\beta)}\le 1/10$, which immediately follows from (\ref{Eq (Section 4) Determination of tau 1}).
For each \(r\in(0,\,\rho\rbrack\), there corresponds a unique \(k\in{\mathbb Z}_{\ge 0}\) such that \(\rho_{k+1}<r\le \rho_{k}\). Then by (\ref{Eq (Section 4) Digital Campanato-Growth estimate in Perturbation}), we have
\begin{align}\label{Eq (Section 4) Continuous Campanato-Growth estimate in Perturbation}
\fint_{B_{r}(x_{0})}\mleft\lvert{\mathcal G}_{2\delta,\,\varepsilon}(Du_{\varepsilon})-\Gamma_{\infty}\mright\rvert^{2}\,{\mathrm d}x&\le \tau^{-n}\fint_{B_{\rho_{k}}(x_{0})}\mleft\lvert{\mathcal G}_{2\delta,\,\varepsilon}(Du_{\varepsilon})-\Gamma_{\infty}\mright\rvert^{2}\,{\mathrm d}x \nonumber\\&\le \tau^{2\beta (k+1)}\mu^{2}\le \mleft(\frac{r}{\rho}\mright)^{2\beta}\mu^{2}
\end{align}
for all $r\in(0,\,\rho\rbrack$.
By the Cauchy--Schwarz inequality, we also have
\begin{align*}
\mleft\lvert\mleft({\mathcal G}_{2\delta,\,\varepsilon}(Du_{\varepsilon})\mright)_{x_{0},\,r}-\Gamma_{\infty}\mright\rvert&\le \fint_{B_{r}(x_{0})}\mleft\lvert{\mathcal G}_{2\delta,\,\varepsilon}(Du_{\varepsilon})-\Gamma_{\infty}\mright\rvert^{2}\,{\mathrm d}x\\&\le\mleft(\fint_{B_{r}(x_{0})}\mleft\lvert{\mathcal G}_{2\delta,\,\varepsilon}(Du_{\varepsilon})-\Gamma_{\infty}\mright\rvert^{2}\,{\mathrm d}x \mright)^{1/2}\le \mleft(\frac{r}{\rho}\mright)^{\beta}\mu.
\end{align*}
for all $r\in(0,\,\rho\rbrack$.
This result yields
\[\Gamma_{2\delta,\,\varepsilon}(x_{0})\coloneqq\lim_{r\to 0}\mleft({\mathcal G}_{2\delta,\,\varepsilon}(Du_{\varepsilon})\mright)_{x_{0},\,r}=\Gamma_{\infty},\]
and hence the desired estimate (\ref{Eq (Section 2) Campanato-type growth from Schauder}) clearly follows from (\ref{Eq (Section 4) Continuous Campanato-Growth estimate in Perturbation}). It is noted that (\ref{Eq (Section 2) esssup V-epsilon}) and (\ref{Eq (Section 2) Control of G-2delta-epsilon by G-delta-epsilon}) imply \(\mleft\lvert{\mathcal G}_{2\delta,\,\varepsilon}(Du_{\varepsilon})\mright\rvert\le \mu\) a.e. in \(B_{\rho}(x_{0})\), and therefore (\ref{Eq (Section 2) Bound of Gamma-2delta-epsilon}) is obvious.

Finally, we mention that we may choose a sufficiently small constant \(\tau=\tau(C_{\ast},\,\beta)\in(0,\,1/2)\) verifying (\ref{Eq (Section 4) Determination of tau 1}), (\ref{Eq (Section 4) Determination of tau 2}), and (\ref{Eq (Section 4) Determination of tau 3}). Corresponding to this \(\tau\), we take sufficient small numbers \(\nu\in(0,\,1/6),\,{\hat\rho}\in(0,\,1)\) as in Lemma \ref{Lemma: Shrinking lemma 2}, depending at most on $b$, $n$, $N$, $p$, $q$, $\gamma$, $\Gamma$, $F$, $M$, $\delta$, and $\theta=\tau^{\frac{n+2}{\vartheta}}$. Then, we are able to determine a sufficiently small radius \(\rho_{\star}=\rho_{\star}(C_{\ast},\,\beta,\,{\hat\rho})\in(0,\,1)\) verifying (\ref{Eq (Section 4) Determination of rho-star 1}) and (\ref{Eq (Section 4) Determination of rho-star 2}), and this completes the proof.
\end{proof}

\section{Appendix: Local Lipschitz bounds}\label{Section: Appendix}
In Section \ref{Section: Appendix}, we would like to provide the proof of Proposition \ref{Prop: Lipschitz bounds} for the reader's convenience.

Before showing Proposition \ref{Prop: Lipschitz bounds}, we recall two basic lemmata (see \cite[Lemma 4.3]{MR2777537} and \cite[Chapter 2, Lemma 4.7]{MR0244627} for the proofs).
\begin{lemma}\label{Lemma: Absorbing lemma}
Assume that a non-negative bounded function $H\colon \lbrack 0,\,1\rbrack\rightarrow (0,\,\infty)$ admits constants $\theta\in\lbrack 0,\,1)$ and $\alpha,\,A,\,B\in(0,\,\infty)$ such that
\[H(t)\le \theta H(s)+\frac{A}{(s-t)^{\alpha}}+B\]
holds whenever $0\le t<s\le 1$. Then we have
\[H(t)\le C(\alpha,\,\theta)\mleft[\frac{A}{(s-t)^{\alpha}}+B \mright]\]
for any $0\le t<s\le 1$.
\end{lemma}
\begin{lemma}\label{Lemma: Geometric convergence lemma}
Assume that a sequence $\{a_{m}\}_{m=0}^{\infty}\subset (0,\,\infty)$ satisfies
\[a_{m+1}\le CB^{m}a_{m}^{1+\varsigma}\]
for all $m\in{\mathbb Z}_{\ge 0}$. Here $C,\,\varsigma\in(0,\,\infty)$ and $B\in(1,\,\infty)$ are constants. If $a_{0}$ satisfies
\[a_{0}\le C^{-1/\varsigma}B^{-1/\varsigma^{2}},\]
then $a_{m}\to 0$ as $m\to\infty$.
\end{lemma}

The proof of Proposition \ref{Prop: Lipschitz bounds} is based on De Giorgi's truncation. 
More sophisticated computations than ours concerning local Lipschitz bounds by De Giorgi's truncation are given in \cite[Theorem 1.13]{MR4078712}, where external force terms are assumed to be less regular than $L^{q}\,(n<q\le\infty)$. Compared with \cite{MR4078712}, our proof of Proposition \ref{Prop: Lipschitz bounds} is rather elementary, since we only deal with the case where the external force term $f_{\varepsilon}$ is in a Lebesgue space $L^{q}$ with $q\in(n,\,\infty\rbrack$. To control $f_{\varepsilon}$ by the $L^{q}$-norm, we appeal to standard absorbing arguments as in \cite[Theorem 4.1, Method 1]{MR2777537} (see also \cite[\S 4.3]{MR4201656} for scalar problems). 
\begin{proof}
For notational simplicity, we write $B_{r}\coloneqq B_{r}(x_{0})$ for $r\in(0,\,\rho\rbrack$, and $B\coloneqq B_{\rho}(x_{0})$.
We fix a constant $k\coloneqq 1+\lVert f_{\varepsilon}\rVert_{L^{q}(B)}^{1/(p-1)}\ge 1$, and define a superlevel set
\[A(l,\,r)\coloneqq \mleft\{x\in B_{r}\mathrel{} \middle|\mathrel{} \mleft(V_{\varepsilon}(x)-k\mright)_{+}^{p}>l\mright\}\]
for $l\in(0,\,\infty)$ and $r\in(0,\,\rho\rbrack$. Then, by (\ref{Eq (Section 3) Ellipticity of coefficients}) and $k\ge 1$, there holds
\begin{equation}\label{Eq (Appendix) Uniform ellipticity on Lipschitz bounds}
\gamma V_{\varepsilon}^{p-2}{\mathrm{id}}_{n}\leqslant {\mathcal C}_{\varepsilon}(Du_{\varepsilon})\leqslant {\hat\Gamma}V_{\varepsilon}^{p-2}{\mathrm{id}}_{n}\quad \textrm{a.e. in }A(0,\,\rho)
\end{equation}
with ${\hat\Gamma}\coloneqq b+3\Gamma$.

We first claim that there holds
\begin{equation}\label{Eq (Appendix) First claim on Lipschitz bounds}
\int_{B}\mleft\lvert \nabla(\eta W_{l})\mright\rvert^{2}\,{\mathrm{d}}x\le C(b,\,n,\,p,\,\gamma,\,\Gamma)\mleft[\int_{B}\lvert\nabla\eta\rvert^{2}W_{l}^{2}\,{\mathrm{d}}x+\int_{A(l,\,\rho)}f_{k}\mleft(W_{l}^{2}+l^{p}W_{l}+l^{2p}\mright)\,{\mathrm{d}}x\mright]
\end{equation}
for all $l\in\mleft(k^{p},\,\infty\mright)$ and for any non-negative function $\eta\in C_{c}^{1}(B)$. Here the non-negative functions $W_{l}$ and $f_{k}$ are respectively given by \(W_{l}\coloneqq\mleft(\mleft(V_{\varepsilon}-k\mright)_{+}^{p}-l\mright)_{+}\), and 
\[f_{k}\coloneqq \frac{\lvert f_{\varepsilon}\rvert^{2}}{V_{\varepsilon}^{2(p-1)}}\chi_{A(0,\,\rho)}\in L^{q/2}(B),\quad \textrm{so that}\quad f_{k}V_{\varepsilon}^{2(p-1)}=\lvert f_{\varepsilon}\rvert^{2}\chi_{A(0,\,\rho)}\]
holds a.e. in $B$.
To prove (\ref{Eq (Appendix) First claim on Lipschitz bounds}), we apply Lemma \ref{Lemma: Weak formulation for V-epsilon} with $\psi(\sigma)\coloneqq \mleft((\sigma-k)_{+}^{p}-l\mright)_{+}$ and $\zeta\coloneqq\eta^{2}$, so that $W_{l}=\psi(V_{\varepsilon})$ holds.
Under this setting, we discard a non-negative term $J_{3}$, and carefully compute the other integrals.
To compute $J_{1}$, we may use (\ref{Eq (Appendix) Uniform ellipticity on Lipschitz bounds}), since $W_{l}$ vanishes outside the set $A(l,\,\rho)$.
When estimating $J_{4},\,J_{5},\,J_{6}$, we use
\[\frac{V_{\varepsilon}}{V_{\varepsilon}-k}=1+\frac{k}{V_{\varepsilon}-k}\le 1+\frac{k}{l^{1/p}}\le 2\quad \textrm{a.e. in }A(l,\,\rho),\]
and
\[V_{\varepsilon}^{p}\le 2^{p-1}\mleft((V_{\varepsilon}-k)^{p}+k^{p}\mright)\le 2^{p}\mleft(W_{l}+l\mright)\quad \textrm{a.e. in }A(l,\,\rho),\]
which are easy to deduce by $l>k^{p}$. Here we also note that the identity \(W_{l}+l=(V_{\varepsilon}-k)^{p}\) holds a.e. in $A(l,\,\rho)$.
Combining these, we can compute
\begin{align*}
&p\gamma \int_{B}\eta^{2}\mleft\lvert \nabla V_{\varepsilon}\mright\rvert^{2}V_{\varepsilon}^{p-1}\mleft(V_{\varepsilon}-k\mright)^{p-1}\chi_{A(l,\,\rho)}\,{\mathrm{d}}x\\ &\le J_{2}\le 2\lvert J_{1}\rvert+\frac{1}{\gamma}(n\lvert J_{4}\rvert+\lvert J_{5}\rvert)+2\lvert J_{6}\rvert\\&\le 4{\hat\Gamma}\int_{B}\eta W_{l}V_{\varepsilon}^{p-1}\lvert\nabla V_{\varepsilon}\rvert\lvert\nabla\eta\rvert\,{\mathrm{d}}x\\&\quad+\frac{1}{\gamma}\mleft[n\int_{A(l,\,\rho)}f_{k}W_{l}V_{\varepsilon}^{p}\eta^{2}\,{\mathrm{d}}x+p\int_{A(l,\,\rho)}f_{k}\mleft(V_{\varepsilon}-k\mright)^{p}V_{\varepsilon}^{p}\frac{V_{\varepsilon}}{V_{\varepsilon}-k}{\mathrm{d}}x\mright]\\&\quad\quad+4\int_{A(l,\,\rho)}\lvert f_{\varepsilon}\rvert\lvert\nabla\eta\rvert W_{l}V_{\varepsilon}\eta\,{\mathrm{d}}x\\&\le \frac{p\gamma}{2}\int_{B}\eta^{2}\mleft\lvert \nabla V_{\varepsilon}\mright\rvert^{2}V_{\varepsilon}^{p-1}\mleft(V_{\varepsilon}-k\mright)^{p-1}\chi_{A(l,\,\rho)}\,{\mathrm{d}}x\\&\quad +C(b,\,n,\,p,\,\gamma,\,\Gamma)\mleft[\int_{B}W_{l}^{2}\lvert\nabla \eta\rvert^{2}\,{\mathrm{d}}x+\int_{A(l,\,\rho)}f_{k}\mleft(W_{l}^{2}+lW_{l}+l^{2}\mright)\,{\mathrm{d}}x\mright]
\end{align*}
by Young's inequality. By an absorbing argument, we obtain
\begin{align*}
&\int_{B}\eta^{2}\mleft\lvert\nabla W_{l}\mright\rvert^{2}\,{\mathrm{d}}x=p^{2}\int_{B}\eta^{2}\lvert \nabla V_{\varepsilon}\rvert^{2}\mleft(V_{\varepsilon}-k\mright)^{2(p-1)}\chi_{A_{l}}\,{\mathrm{d}}x \\&\le p^{2}\int_{B}\eta^{2}\mleft\lvert \nabla V_{\varepsilon}\mright\rvert^{2}V_{\varepsilon}^{p-1}\mleft(V_{\varepsilon}-k\mright)^{p-1}\chi_{A_{l}}\,{\mathrm{d}}x\\&\le C(b,\,n,\,p,\,\gamma,\,\Gamma)\mleft[\int_{B}W_{l}^{2}\lvert\nabla \eta\rvert^{2}\,{\mathrm{d}}x+\int_{A(l,\,\rho)}f_{k}\mleft(W_{l}^{2}+lW_{l}+l^{2}\mright)\,{\mathrm{d}}x\mright]
\end{align*}
from which (\ref{Eq (Appendix) First claim on Lipschitz bounds}) immediately follows.

Next, we would like to find a constant $C_{\clubsuit}=C_{\clubsuit}(b,\,n,\,N,\,p,\,q,\,\gamma,\,\Gamma)\in (0,\,\infty)$ and an exponent $\varsigma=\varsigma(n,\,q)\in(0,\,2/n\rbrack$ such that
\begin{equation}\label{Eq (Appendix) 1.5-th claim on Lipschitz bounds}
\int_{A({\hat l},\,{\hat r})}W_{{\hat l}}^{2}\,{\mathrm{d}}x\le C\mleft[\frac{1}{(r-{\hat r})^{2}}+\frac{{\hat l}^{2}}{({\hat l}-l)^{2}} \mright]\frac{1}{({\hat l}-l)^{2\varsigma}}\mleft(\int_{A(l,\,r)}W_{l}^{2}{\mathrm{d}}x \mright)^{1+\varsigma}
\end{equation}
holds for arbitrary numbers $l$, ${\hat l}$, $r$, ${\hat r}$, $\hat\rho$ enjoying $0<{\hat r}<r\le{\hat\rho} \le\rho$ and $0<l_{0}\coloneqq k^{p}+C_{\clubsuit}\lVert W_{k^{p}}\rVert_{L^{2}(B_{\hat\rho})}\le l<{\hat l}<\infty$. 
As a preliminary for proving (\ref{Eq (Appendix) 1.5-th claim on Lipschitz bounds}), we would like to verify that
\begin{equation}\label{Eq (Appendix) Second claim on Lipschitz bounds}
\int_{A(l,\,r)}(\eta W_{l})^{2}\,{\mathrm{d}}x\le C\mleft[\lvert A(l,\,r)\rvert^{\varsigma}\int_{A(l,\,r)}W_{l}^{2}\lvert \nabla \eta\rvert^{2}{\mathrm{d}}x+l^{2}\lvert A(l,\,r)\rvert^{1+\varsigma} \mright]
\end{equation}
holds for all $r\in(0,\,{\hat \rho}\rbrack$, $l\in\lbrack l_{0},\,\infty)$, and for any non-negative function $\eta\in C_{c}^{1}(B_{r})$ with $0\le\eta\le 1$. Here the constants $C\in(0,\,\infty)$ in (\ref{Eq (Appendix) 1.5-th claim on Lipschitz bounds})--(\ref{Eq (Appendix) Second claim on Lipschitz bounds}) depend at most on $b$, $n$, $p$, $q$, $\gamma$ and $\Gamma$.
To deduce (\ref{Eq (Appendix) Second claim on Lipschitz bounds}), we should mention that $\lVert f_{k}\rVert_{L^{q/2}(B)}\le 1$ is clear by the definitions of $k$ and $f_{k}$.
For simplicity, we consider the case $n\ge 3$, where we can use the Sobolev embedding $W_{0}^{1,\,2}(B_{r})\hookrightarrow L^{2^{\ast}}(B_{r})$ with $2^{\ast}=2n/(n-2)$. Combining with H\"{o}lder's inequality and Young's inequality, by (\ref{Eq (Appendix) First claim on Lipschitz bounds}) we get
\begin{align*}
\int_{A(l,\,r)}\eta^{2}f_{k}\mleft(W_{l}^{2}+lW_{l}+l^{2}\mright)\,{\mathrm{d}}x& \le \mleft(\sigma+C(n)\lvert A(l,\,r)\rvert^{2/n-2/q}\mright)\int_{A(l,\,r)}\mleft\lvert\nabla (\eta W_{l})\mright\rvert^{2}\,{\mathrm{d}}x\\&\quad+l^{2}\mleft(\frac{C(n)}{\sigma}\lvert A(l,\,r) \rvert^{1+2/n-4/q}+\lvert A(l,\,r)\rvert^{1-2/q}\mright)
\end{align*}
for any $\sigma>0$.
By the definition of $A(l,\,r)$ and the Cauchy--Schwarz inequality, we can easily check
\[\lvert A(l,\,r)\rvert\le \frac{1}{l-k^{p}}\int_{A(l,\,r)}W_{k^{p}}\,{\mathrm{d}}x\le \frac{\lvert A(l,\,r)\rvert^{1/2}}{l-k^{p}}\lVert W_{k^{p}}\rVert_{L^{2}(B_{\hat\rho})}.\]
Hence, it follows that
\[\lvert A(l,\,r)\rvert\le \mleft(\frac{\lVert W_{k^{p}}\rVert_{L^{2}(B_{\hat\rho})}}{l-k^{p}}\mright)^{2}.\]
Thus, we can choose and fix sufficiently small $\sigma\in(0,\,1)$ and suitably large $C_{\clubsuit}\in(0,\,\infty)$, both of which depend at most on $b$, $n$, $p$, $q$, $\gamma$ and $\Gamma$, so that an absorbing argument can be made. Moreover, by our choice of $C_{\clubsuit}$, we may let $\lvert A(l,\,r)\rvert\le 1$.
As a result, we are able to conclude (\ref{Eq (Appendix) Second claim on Lipschitz bounds}) with $\varsigma\coloneqq 2/n-2/q\in(0,\,2/n)$ when $n\ge 3$. In the remaining case $n=2$, we fix a sufficiently large constant $\kappa\in(2,\,\infty)$, and use the Sobolev embedding $W_{0}^{1,\,2}(B_{r})\hookrightarrow L^{\kappa}(B_{r})$. By similar computations, we can find a constant $C_{\clubsuit}$, depending at most on $b$, $n$, $p$, $q$, $\gamma$, $\Gamma$ and $\kappa$, such that (\ref{Eq (Appendix) Second claim on Lipschitz bounds}) holds for some exponent $\varsigma=\varsigma(n,\,q,\,\kappa)\in(0,\,1-2/q)$.
Now we would like to prove (\ref{Eq (Appendix) 1.5-th claim on Lipschitz bounds}).
Let $l$, ${\hat l}$ and $r$, ${\hat r}$, ${\hat\rho}$ satisfy respectively $l_{0}<l<{\hat l}<\infty$ and $0<{\hat r}<r\le{\hat\rho}\le\rho$.
Corresponding to the radii $r,\,{\hat r}$, we fix a non-negative function $\eta\in C_{c}^{1}(B_{r})$ such that
\[\eta\equiv 1\quad \textrm{on }B_{{\hat r}}\quad \textrm{and}\quad \lvert\nabla\eta\rvert\le \frac{2}{r-{\hat r}}\quad \textrm{in }B_{r}.\]
Since $W_{l}\ge {\hat l}-l$ holds a.e. in $A({\hat l},\,\rho)$, it is easy to check that
\[\lvert A({\hat l},\,\rho)\rvert\le \frac{1}{({\hat l}-l)^{2}}\int_{A(l,\,\rho)}W_{l}^{2} \,{\mathrm{d}}x.\]
Also, the inclusion $A({\hat l},\,{\hat r})\subset A(l,\,r)$ and the inequality $W_{{\hat l}}\le W_{l}$ yield
\[\int_{A({\hat l},\,{\hat r})}W_{\hat l}^{2}\,{\mathrm{d}}x\le\int_{A(l,\,r)}W_{l}^{2}\,{\mathrm{d}}x.\]
From these inequalities and (\ref{Eq (Appendix) Second claim on Lipschitz bounds}), we can easily conclude (\ref{Eq (Appendix) 1.5-th claim on Lipschitz bounds}).

From (\ref{Eq (Appendix) 1.5-th claim on Lipschitz bounds}), we would like to complete the proof of Proposition \ref{Prop: Lipschitz bounds}. 
We fix arbitrary $\theta\in(0,\,1)$, ${\hat\rho}\in(0,\,\rho\rbrack$. For each $m\in{\mathbb Z}_{\ge 0}$, we set
\[l_{m}\coloneqq l_{0}+L_{0}(1-2^{-m}),\quad \rho_{m}\coloneqq \mleft[\theta+2^{-m}(1-\theta)\mright]{\hat\rho},\quad a_{m}\coloneqq \lVert W_{l_{m}}\rVert_{L^{2}(B_{\rho_{m}})},\]
where the constant $L_{0}\in(0,\,\infty)$ is to be chosen later. Then, by (\ref{Eq (Appendix) 1.5-th claim on Lipschitz bounds}), we get
\[a_{m+1}\le C\mleft[\frac{2^{m+1}}{(1-\theta)r}+2^{m+1}\mright]\frac{2^{\varsigma(m+1)}}{L_{0}^{\varsigma}}a_{m}\le \frac{{\tilde C}}{(1-\theta)r}L_{0}^{-\varsigma}2^{(1+\varsigma)m}a_{m}^{1+\varsigma}\]
for every $m\in{\mathbb Z}_{\ge 0}$, where ${\tilde C}\in(0,\,\infty)$ depends at most on $b$, $n$, $p$, $q$, $\gamma$ and $\Gamma$. We set $L_{0}$ by
\[L_{0}\coloneqq C_{\diamondsuit}\lVert W_{k^{p}}\rVert_{L^{2}(B_{\hat\rho})}\quad \textrm{with}\quad C_{\diamondsuit}\coloneqq \mleft[\frac{{\tilde C}}{(1-\theta){\hat\rho}} \mright]^{1/\varsigma}2^{\frac{1+\varsigma}{\varsigma^{2}}},\]
so that we obtain
\[a_{0}=\lVert W_{l_{0}}\rVert_{L^{2}(B_{\hat\rho})}\le \lVert W_{k^{p}}\rVert_{L^{2}(B_{\hat\rho})}\le \mleft[\frac{{\tilde C}L_{0}^{-\varsigma}}{(1-\theta){\hat\rho}} \mright]^{-1/\varsigma}\mleft[2^{1+\varsigma}\mright]^{-1/\varsigma^{2}}.\]
By Lemma \ref{Lemma: Geometric convergence lemma}, we have $a_{m}\to 0$ as $m\to \infty$. In particular, it follows that
\[\int_{A(l_{0}+L_{0},\,\theta{\hat\rho})}W_{l_{0}+L_{0}}^{2}\,{\mathrm{d}}x=0,\]
which implies \(\lVert W_{k^{p}}\rVert_{L^{\infty}(B_{\theta \hat\rho})}\le\mleft(C_{\clubsuit}+C_{\diamondsuit}\mright)\lVert W_{k^{p}}\rVert_{L^{2}(B_{\hat\rho})}\). As a consequence, we obtain
\[\lVert W_{k^{p}}\rVert_{L^{\infty}(B_{\theta\hat\rho})}\le C(b,\,n,\,p,\,q,\,\gamma,\,\Gamma)\frac{\lVert W_{k^{p}}\rVert_{L^{2}(B_{\hat\rho})}}{\mleft[(1-\theta)\hat\rho\mright]^{pd/2}} \quad \textrm{for all }\theta\in(0,\,1),\,\hat\rho\in(0,\,\rho\rbrack\]
with $d\coloneqq 2/(p\varsigma)\in \lbrack n/p,\,\infty)$. By $\lVert W_{k^{p}}\rVert_{L^{2}(B_{\hat\rho})}\le \lVert W_{k^{p}}\rVert_{L^{1}(B)}^{1/2}\lVert W_{k^{p}}\rVert_{L^{\infty}(B_{\hat\rho})}^{1/2}$ and Young's inequality, we get
\[\lVert W_{k^{p}}\rVert_{L^{\infty}(B_{\theta\hat\rho})}\le \frac{1}{2}\lVert W_{k^{p}}\rVert_{L^{\infty}(B_{\hat\rho})}+C(b,\,n,\,p,\,q,\,\gamma,\,\Gamma)\frac{\lVert W_{k^{p}}\rVert_{L^{1}(B)}}{\mleft[(1-\theta)\hat\rho\mright]^{pd}}\]
for all $\theta\in(0,\,1)$, ${\hat\rho}\in(0,\,\rho\rbrack$. By applying Lemma \ref{Lemma: Absorbing lemma} with $H(s)\coloneqq \lVert W_{k^{p}}\rVert_{L^{\infty}(B_{s\rho})}$. As a result, we are able to deduce
\[\lVert W_{k^{p}}\rVert_{L^{\infty}(B_{\theta\rho})}\le C(b,\,n,\,p,\,q,\,\gamma,\,\Gamma)\frac{\lVert W_{k^{p}}\rVert_{L^{1}(B)}}{\mleft[(1-\theta)\rho\mright]^{pd}}\quad \textrm{for all }\theta\in(0,\,1).\]
By the definition of $W_{k^{p}}$, we get
\[\esssup_{B_{\theta\rho}}\,V_{\varepsilon}\le k+\esssup_{B_{\theta\rho}}\,\mleft(W_{k^{p}}+k^{p}\mright)^{1/p}\le 2k+C(b,\,n,\,p,\,q,\,\gamma,\,\Gamma)\frac{\lVert W_{k^{p}}\rVert_{L^{1}(B)}^{1/p}}{[(1-\theta)\rho]^{d}},\]
By $k=1+\lVert f_{\varepsilon}\rVert_{L^{q}(B)}^{1/(p-1)}$ and $\lVert W_{k^{p}}\rVert_{L^{1}(B)}\le \lVert V_{\varepsilon}\rVert_{L^{p}(B)}^{p}$, we complete the proof.
\end{proof}
\begin{Remark}[Higher regularity of regularized solutions]\label{Eq Higher W-2-2 and W-1-infty regularity}\upshape
In this paper, we have often used $u_{\varepsilon}\in W_{\mathrm{loc}}^{1,\,\infty}(\Omega;\,{\mathbb R}^{N})\cap W_{\mathrm{loc}}^{2,\,2}(\Omega;\,{\mathbb R}^{N})$. Following \cite{MR709038}, \cite{MR1634641} and \cite[Chapter 8]{MR1962933}, we briefly describe how to improve this regularity (see also \cite[Chapters 4--5]{MR0244627}). There it is noted that it is not restrictive to let $f_{\varepsilon}\in L^{\infty}(\Omega;\,{\mathbb R}^{N})$, or even $f_{\varepsilon}\in C^{\infty}(\Omega;\,{\mathbb R}^{N})$, since our approximation arguments work as long as (\ref{Eq (Section 2) Weak convergence of f}) holds. 

First, appealing to a standard difference quotient method as in \cite[Theorem 2]{MR1634641} or \cite[\S 8.2]{MR1962933}, we are able to get
\begin{equation}\label{Eq (Appendix) improved regularity from difference quotient}
\int_{\omega} V_{\varepsilon}^{p-2}\mleft\lvert D^{2}u_{\varepsilon}\mright\rvert^{2}\,{\mathrm{d}}x\le C\mleft(\varepsilon,\,\omega,\,\lVert f_{\varepsilon}\rVert_{L^{\infty}(\Omega)}\mright)<\infty\quad \textrm{for every }\omega\Subset \Omega,
\end{equation}
which is possible by (\ref{Eq (Section 3) Uniform ellipticity on approximated density}).
Secondly, we appeal to the Uhlenbeck structure to prove $Du_{\varepsilon}\in L_{\mathrm{loc}}^{\infty}(\Omega;\,{\mathbb R}^{Nn})$. We test
\[\phi\coloneqq \eta^{2}V_{\varepsilon}^{l}D_{\alpha}u_{\varepsilon}\quad \textrm{with}\quad\eta\in C_{c}^{1}(B_{\rho}(x_{0})),\quad 0\le l<\infty\]
into the weak formulation (\ref{Eq (Section 3) Weak formulation differentiated}) for each $\alpha\in\{\,1,\,\dots\,,\,n\,\}$.
We should note that when $l=0$ this test function is admissible by (\ref{Eq (Appendix) improved regularity from difference quotient}), and all of the computations in Lemma \ref{Lemma: Weak formulation for V-epsilon} make sense. By (\ref{Eq (Section 3) Resulting weak formulation}) and standard computations given in \cite[\S 3]{MR709038}, \cite[\S 8.3]{MR1962933}, we can improve local integrability of $V_{\varepsilon}$, and in particular we may test the same $\phi$ with larger $l>0$. By Moser's iteration arguments, we are able to conclude $V_{\varepsilon}\in L_{\mathrm{loc}}^{\infty}(\Omega)$, and hence $Du_{\varepsilon}\in L_{\mathrm{loc}}^{\infty}(\Omega;\,{\mathbb R}^{Nn})$ follows. Finally, $u_{\varepsilon}\in W_{\mathrm{loc}}^{2,\,2}(\Omega;\,{\mathbb R}^{N})$ is clear by (\ref{Eq (Appendix) improved regularity from difference quotient}) and $Du_{\varepsilon}\in L_{\mathrm{loc}}^{\infty}(\Omega;\,{\mathbb R}^{Nn})$.

These computations above are substantially dependent on an approximation parameter $\varepsilon\in(0,\,1)$. When it comes to local uniform $L^{\infty}$-bounds of $Du_{\varepsilon}$, the strategy based on Moser's iteration may not work, since the test function $\phi$ may intersect with the facet of $u_{\varepsilon}$. In contrast, for the scalar problem, local uniform $L^{\infty}$-bounds of gradients $\nabla u_{\varepsilon}=(\partial_{x_{1}}u_{\varepsilon},\,\dots\,,\,\partial_{x_{n}}u_{\varepsilon})$ for weak solutions to (\ref{Eq (Section 2) Approximated system}) is successfully deduced by Moser's iteration \cite[\S 4.2]{MR4201656}.
The strategy therein is to choose test functions far from facets, by truncating the term $\partial_{x_{\alpha}}u_{\varepsilon}$ in a place $\{\lvert \partial_{x_{\alpha}}u_{\varepsilon}\rvert\le k\}$ for some $k\ge 1$. This modification forces us to adapt another modulus that is different from $V_{\varepsilon}=\sqrt{\varepsilon^{2}+\lvert Du_{\varepsilon}\rvert^{2}}$ and that may lack spherical symmetry. For this reason, it appears that the proof of a priori Lipschitz bounds based on Moser's iteration works only in the scalar problem. In the proof of Proposition \ref{Prop: Lipschitz bounds}, we have appealed to De Giorgi's truncation, since we do not have to choose another non-symmetric modulus.
\end{Remark}
\newpage
\bibliographystyle{plain}

\end{document}